\documentclass[11pt]{article}
\usepackage{amsmath,fullpage,amssymb,hyperref,amsthm,graphicx,color}
\usepackage{tikz}

\usepackage{bbm}

\usepackage{subfig}
\usepackage{tabularx}

\newtheorem{theorem}{Theorem}[section]
\newtheorem{proposition}[theorem]{Proposition}
\newtheorem{lemma}[theorem]{Lemma}
\newtheorem{corollary}[theorem]{Corollary}
\newtheorem{definition}[theorem]{Definition}   
\newtheorem{conjecture}[theorem]{Conjecture}

\theoremstyle{remark}
\newtheorem{remark}[theorem]{Remark}

\def\S{\mathcal{S}}   
\DeclareMathOperator\inv{inv}
\DeclareMathOperator\st{st}
\def\eps{\epsilon}

\def\P{\mathbb{P}}
\def\e{\mathbb{E}}

\def\Ov{\mbox{{\rm Ov}}}

\title{The probability of avoiding consecutive patterns in the Mallows distribution}
\author{Harry Crane\thanks{Department of Statistics \& Biostatistics, Rutgers University, 110 Frelinghuysen Avenue, Piscataway, NJ 08854, USA. E-mail: \texttt{hcrane@stat.rutgers.edu}. Partially supported by NSF grants CNS-1523785 and CAREER DMS-1554092.} \and Stephen DeSalvo\thanks{Department of Mathematics, University of California, Los Angeles, 520 Portola Plaza, Los Angeles, CA 90095.}  \and Sergi Elizalde\thanks{Department of Mathematics, Dartmouth College, 6188 Kemeny Hall, Hanover, NH 03755, USA. E-mail: \texttt{sergi.elizalde@dartmouth.edu}. Partially supported by Simons Foundation grant \#280575 and NSA grant H98230-14-1-0125.}}

\date{}

\begin{document}

\maketitle

\begin{abstract}
We use various combinatorial and probabilistic techniques to study growth rates for the probability that a random permutation from the Mallows distribution avoids consecutive patterns.
The Mallows distribution behaves like a $q$-analogue of the uniform distribution by weighting each permutation $\pi$ by $q^{\inv(\pi)}$, where $\inv(\pi)$ is the number of inversions in $\pi$ and $q$ is a positive, real-valued parameter.
We prove that the growth rate exists for all patterns and all $q>0$, and we generalize Goulden and Jackson's cluster method to keep track of the number of inversions in permutations avoiding a given consecutive pattern.
Using singularity analysis, we approximate the growth rates for length-3 patterns, monotone patterns, and non-overlapping patterns starting with 1, and we compare growth rates between different patterns.
We also use Stein's method to show that, under certain assumptions on $q$, the length of $\sigma$, and $\inv(\sigma)$, the number of occurrences of a given pattern $\sigma$ is well approximated by the normal distribution.

\end{abstract}

 \medskip


\section{Introduction}
\label{sec:intro}

Let $\S_n$ be the set of permutations of $[n]=\{1,\ldots,n\}$.
Given a list of $m$ distinct integers $w=w_1\dots w_m$, the {\em standardization} of $w$, written $\st(w)$, is the unique permutation of $[m]$ that is order-isomorphic to $w$.
We obtain $\st(w)$ by replacing the smallest element among $\{w_1,\ldots,w_m\}$ with $1$, the second smallest with $2$, and so on.

We say that $\pi\in\S_n$ {\em contains $\sigma\in\S_m$ consecutively} if there exists an index $j\in[n-m+1]$ such that $\st(\pi_j\pi_{j+1}\dots\pi_{j+m-1})=\sigma$; otherwise, we say that $\pi$ {\em avoids $\sigma$ consecutively}.
In this context, we call $\sigma$ a {\em pattern} and we write $\S_n(\sigma)$ to denote the set of permutations in $\S_n$ that avoid $\sigma$ as a consecutive pattern.   
The systematic study of consecutive patterns originated in \cite{EliNoy,EliNoy2} and has expanded in various directions since.
Although much of the literature on pattern avoidance deals with classical pattern avoidance, whereby $\pi\in\S_n$ avoids $\sigma\in\S_m$ only if $\st(\pi_{i_1}\dots\pi_{i_m})\neq\sigma$ for all length $m$ subsequences $1\leq i_1<\dots<i_m\leq n$, consecutive pattern avoidance has received significant attention in the last 15 years; see~\cite{Elisurvey} for a survey.
Except for some references in the introduction, we speak only of consecutive pattern avoidance in this paper, even if not explicitly stated.

In this paper we connect prior work on consecutive pattern avoidance to ongoing developments at the interface of pattern avoidance and random permutations.
Classical enumerative studies of pattern avoidance admit a straightforward probabilistic interpretation, as the problem of enumerating permutations avoiding a consecutive pattern $\sigma\in\S_m$ is equivalent to finding the probability that a uniform random permutation avoids $\sigma$.
This observation makes a variety of combinatorial and probabilistic techniques available for studying pattern avoidance.
So far, these connections have been mostly confined to the equivalence between enumeration and uniform random permutations.
Below we bring tools from analytic combinatorics and probability theory to bear on the behavior of pattern avoidance for a widely studied class of non-uniform random permutations.

A straightforward and natural way to generate a class of non-uniform distributions is by exponential tilting of the uniform distribution with respect to some statistic $t:\mathcal{S}_n\to\mathbb{R}$ on the space of permutations.
In this generic setting, we assign probability 
\begin{equation}\label{eq:exp}\frac{q^{t(\pi)}}{\sum_{\sigma\in\mathcal{S}_n}q^{t(\sigma)}}\end{equation}
to each $\pi\in\mathcal{S}_n$.
Any such model specializes to the uniform distribution upon setting $q=1$.
For example, taking $t(\pi)$ equal to the number of cycles in $\pi$ yields the well known Ewens sampling formula \cite{CraneESF,Ewens1972}.
Previous authors \cite{cameronkilpatrick,SaganDokos2012} consider certain questions of pattern avoidance, such as analogs of Wilf equivalence, in a setting equivalent to \eqref{eq:exp}.

Owing to its natural importance in the context of ranking statistics, we study pattern avoidance probabilities under the setting of \eqref{eq:exp} with $t(\pi)$ taken to be the number of inversions in $\pi$.
An {\em inversion} in $\pi=\pi_1\dots\pi_n\in\S_n$ is a pair of indices $(i,j)$ with $i<j$ and $\pi_i>\pi_j$.
For a real-valued parameter $q>0$, the {\em Mallows distribution} with parameter $q$ on $\S_n$---sometimes denoted Mallows($q$)---assigns probability 
\begin{equation}\label{eq:Mallows}
\frac{q^{\inv(\pi)}}{[n]_q!}
\end{equation}
to each $\pi\in\S_n$, where $\inv(\pi)$ is the number of inversions in $\pi$, $[n]_q=1+q+\dots+q^{n-1}$, $[n]_q!=[1]_q [2]_q \dots [n]_q$, and $\binom{m+n}{m}_q=\frac{[m+n]_q!}{[m]_q![n]_q!}$ are the usual $q$-analogues with $q$ treated as a real-valued parameter.

Mallows's distribution \cite{Mallows1957} is a canonical statistical model for ranking data, with many useful statistical properties stemming from its exponential family structure endowed from \eqref{eq:exp} and other probabilistic properties; see, e.g., \cite[Chapters 6 and 7]{FV} for more discussion of its statistical applications.
More recent work on the Mallows distribution reveals connections to a range of topical problems in statistical physics \cite{SS}, Markov chain Monte Carlo \cite{DR}, the longest increasing subsequence problem \cite{BasuBhatnagar2016,Bhatnagar2014}, cycle lengths in random permutations \cite{GladkichPeled2016}, and quasi-exchangeable sequences \cite{GnedinOlshanski2009,GnedinOlshanski2010}.

Others have recently studied characteristics of classical pattern avoidance for uniform random permutations \cite{HoffmanRizzolo2016I,HoffmanRizzolo2016II,MinerPak}, with the former two references tying characteristics of certain pattern avoiding permutations to Brownian excursion, a fundamental probabilistic object.
Perarnau \cite{Perarnau} uses probabilistic techniques to give an alternate proof of the fact that, among consecutive patterns of a given length, the monotone pattern is avoided by the largest number of permutations of length $n$, for sufficiently large $n$. 
A proof of this fact using singularity analysis of generating function was previously given in~\cite{EliCMP}.
In \cite{CraneDeSalvo2015}, two of the authors used the technique of Poisson approximation to estimate pattern avoidance probabilities for sufficiently large patterns in random permutations from the Mallows distribution. 
The method in \cite{CraneDeSalvo2015} breaks down for fixed patterns of small size.

Our analysis below builds on these other recent developments at the interface of probability and combinatorics in several distinct ways.
Much of the contents of Sections \ref{section:GJ}-\ref{section:compare} extend the techniques from singularity analysis and analytic combinatorics developed in \cite{EliNoy2,GJ79}.
In Section \ref{section:bounds}, we use probabilistic techniques to obtain bounds on the growth rate for arbitrary patterns, generalizing the outcomes in \cite{Perarnau}. 
In Section \ref{section:Stein}, we conclude our discussion by studying complementary behavior of pattern avoidance in Mallows permutations using the probabilistic technique of Stein's method.
Whereas the techniques of Sections \ref{section:GJ}-\ref{section:bounds} zero in on the growth rate of the probability of pattern avoidance, our analysis in Section \ref{section:Stein} considers the asymptotic distribution of the number of occurrences of a given pattern in a random permutation from the Mallows distribution, which also 
gives the rate of convergence to the Gaussian distribution, providing a complementary perspective to the results of Sections \ref{section:growth}-\ref{section:compare}.
The results of Section \ref{section:Stein} are in the same vein as other work of \cite{NakamuraJanson2015}, who studied the asymptotic distribution of the number of occurrences of patterns in uniform random permutations, and Janson \cite{Janson2016}, who studied the behavior of random permutations drawn uniformly from the set of 132-avoiding permutations.

\subsection{Preliminaries}

For $q>0$ and $\sigma\in\S_m$, we write $P_n(\sigma,q)$ to denote the probability that a random permutation from distribution \eqref{eq:Mallows} avoids $\sigma$ and
\begin{equation}\label{eq:F-sigma}
F_\sigma(q,z)=\sum_{n\ge0} P_n(\sigma,q) z^n\end{equation}
to denote the corresponding generating function.
By definition $$P_n(\sigma,q)=\sum_{\pi\in\S_n(\sigma)} \frac{q^{\inv(\pi)}}{[n]_q!},$$
so that
 $$F_\sigma(q,z)=\sum_{n\ge0}\sum_{\pi\in\S_n(\sigma)} q^{\inv(\pi)} \frac{z^n}{[n]_q!}$$
is the $q$-exponential generating function for $\sigma$-avoiding permutations with respect to the inversion number.
In fact, the same argument works for any family of subsets $A_n\subseteq\S_n$: the ordinary generating function for the probability that a permutation in $\S_n$ belongs to $A_n$ equals the $q$-exponential generating function for permutations in $A_n$ with respect to the inversion number.

The definition of $F_\sigma(q,z)$ can be generalized in order to consider not only permutations that avoid $\sigma$ but also all permutations with respect to the number of occurrences of $\sigma$. Let $c_\sigma(\pi)$ denote the number of occurrences of $\sigma$ in $\pi$ as a consecutive pattern, and let
\begin{equation}
\label{eq:tildeF} \tilde{F}_\sigma(q,u,z)=\sum_{n\ge0}\sum_{\pi} q^{\inv(\pi)} u^{c_\sigma(\pi)}\frac{z^n}{[n]_q!}
\end{equation}
so that $\tilde{F}_\sigma(q,0,z)=F_\sigma(q,z)$.

The Mallows distribution possesses several nice properties that are amenable to the study of pattern avoidance.
One immediately useful observation is that the probability of a permutation $\pi\in\S_n$ under the Mallows distribution with parameter $q^{-1}$ equals
$$\frac{q^{-\inv(\pi)}}{[n]_{q^{-1}}!}=\frac{q^{\binom{n}{2}-\inv(\pi)}}{[n]_{q}!}=\frac{q^{\inv(\pi^r)}}{[n]_{q}!},$$
that is, the probability of the permutation $\pi^r:=\pi_n\dots\pi_1$ under the Mallows distribution with parameter $q$.
The same holds by replacing $\pi^r$ with $\pi^c=(n+1-\pi_1)\dots(n+1-\pi_n)$; whence, 
\begin{equation}\label{eq:reversal}
F_{\sigma^r}(q,z)=F_{\sigma^c}(q,z)=F_{\sigma}(1/q,z)=F_{\sigma^{rc}}(1/q,z).
\end{equation}

Two other crucial properties of the Mallows distribution, whose combinatorial equivalent is given in Lemma~\ref{lem:product}, are weak dissociation and consecutive homogeneity.
Weak dissociation says that for any non-overlapping sets of indices $\alpha = \{i, i+1, \ldots, i+k-1\}$ and $\beta = \{j, j+1, \ldots, j+\ell-1\}$, i.e., $\alpha \cap \beta = \emptyset$, the probability that a given pattern of size~$k$ occurs at the indices in $\alpha$ is independent of the probability that any other pattern of size~$\ell$ occurs at the indices of $\beta$. 
Consecutive homogeneity says that for $\pi=\pi_1\dots\pi_n$ from the Mallows($q)$ distribution on $\S_n$, $1\leq i<i+m-1\leq n$, and any $\sigma\in\S_m$, the pattern $\st(\pi_i\pi_{i+1}\dots\pi_{i+m-1})$ is distributed according to Mallows($q$) on $\S_m$.
These two properties appear repeatedly throughout our analysis below.

\subsection{Outline}
We organize the rest of the paper as follows.
In Section \ref{section:GJ}, we generalize the cluster method of Goulden and Jackson \cite{GJ79} to keep track of the inversion number.
In Section \ref{section:growth}, we show that the growth rate, defined as $\lim_{n\to\infty}P_n(\sigma,q)^{1/n}$, exists for all patterns $\sigma$ and all $q>0$.
In Section \ref{section:monotone}, we apply the method of Section \ref{section:GJ} to approximate the growth rate of pattern avoidance probabilities for monotone patterns (with a precise description for certain values of $q$).
In Section \ref{section:132}, we treat the pattern $132$, for which we can determine the growth rate in all cases, and generalize the argument to all non-overlapping patterns starting with a 1.
In Section \ref{section:compare}, we compare the growth rates between the different patterns of length~3.
In Section \ref{section:bounds}, we prove upper and lower bounds for the growth rate of arbitrary patterns.
In Section \ref{section:Stein}, we apply Stein's method to $m$-dependent random variables to obtain a central limit theorem governing the number of occurrences of a given pattern, with explicit error rates.

\section{The cluster method with respect to inversion number}
\label{section:GJ}

The proof of the following lemma uses the well-known fact that if $W$ is the set of words consisting of $m$ zeros and $n$ ones, then 
\begin{equation}\label{eq:qbin}
\sum_{w \in W} q^{\inv(w)}=\binom{m+n}{m}_q.
\end{equation} 

\begin{lemma}\label{lem:product}
Let $A\subseteq\S_m$, $B\subseteq\S_n$, and 
$$C=\{\pi\in\S_{m+n}:\st(\pi_1\dots\pi_m)\in A, \st(\pi_{m+1}\dots\pi_{m+n})\in B\}.$$
Then
$$\left(\sum_{\pi\in A} \frac{q^{\inv(\pi)}}{[m]_q!}\right) \left(\sum_{\pi\in B} \frac{q^{\inv(\pi)}}{[n]_q!}\right)
= \sum_{\pi\in C} \frac{q^{\inv(\pi)}}{[m+n]_q!}.$$
\end{lemma}

\begin{proof}
Given $\sigma\in A$ and $\tau\in B$, there are $\binom{m+n}{m}$ permutations
$\pi\in\S_{m+n}$ such that $\st(\pi_1\dots\pi_m)=\sigma$ and $\st(\pi_{m+1}\dots\pi_{m+n})=\tau$, obtained by choosing which
$m$ of the values $1,\ldots,m+n$ will be placed in the first $m$ entries. Such a choice can be encoded by a word $w$ of length $m+n$ where $w_i=0$ if the value $i$ appears in the first $m$ entries and $w_i=1$ otherwise.
We then have $\inv(\pi)=\inv(\sigma)+\inv(\tau)+\inv(w)$. Denoting by $W$ the set of words consisting of $m$ zeros and $n$ ones, we have
$$\sum_{\pi\in C} q^{\inv(\pi)}
=\sum_{\sigma\in A} \sum_{\tau\in B} \sum_{w \in W} q^{\inv(\sigma)+\inv(\tau)+\inv(w)}
=\left(\sum_{\sigma\in A} q^{\inv(\sigma)}\right)\left(\sum_{\tau\in B} q^{\inv(\tau)}\right)
\left(\sum_{w \in W} q^{\inv(\tau)}\right).$$
The result now follows from Equation~\eqref{eq:qbin}.
\end{proof}

\begin{remark}
Lemma \ref{lem:product} can be deduced by interpreting the sums as probabilities and appealing to the weak dissociation and consecutive homogeneity properties of the Mallows$(q)$ distribution. See \cite[Section 5]{CraneDeSalvo2015} for a more detailed account of these properties.
\end{remark}

Lemma~\ref{lem:product} allows us to generalize the cluster method of Goulden and Jackson~\cite{GJ79} in order to be able to keep track of the inversion number. 
For $\sigma\in\S_m$, we say that $(\pi;i_1,i_2,\dots,i_k)$ is a {\em $k$-cluster of length $n$ with respect to $\sigma$} if
\begin{itemize}
\item $\pi\in\S_n$,
\item $1=i_1<i_2<\dots<i_k=n-m+1$,
\item $i_{j+1}\le i_j+m-1$ for all $j$, and 
\item $\st(\pi_{i_j}\pi_{i_j+1}\dots\pi_{i_j+m-1})=\sigma$ for all $j$.
\end{itemize}
One can think of a $k$-cluster as a permutation $\pi$ with $k$ marked occurrences of $\sigma$ that start at the positions $i_j$,
each marked occurrence overlaps the next one, and the first and last marked occurrences are all the way at the beginning and at the end of $\pi$, respectively. For example $(192834756;1,3,6)$ is a $3$-cluster with respect to $1423$ and $(12345678;1,4,5)$ is a 
$3$-cluster with respect to $1234$. Note that $\pi$ may have additional occurrences of $\sigma$ aside from the marked ones.
The number of inversions of the cluster is defined to be $\inv(\pi)$.

Let $c_\sigma(i,k,n)$ denote the number of $k$-clusters of length $n$ with respect to $\sigma$ having $i$ inversions. 
For example, $c_\sigma(\inv(\sigma),1,m)=1$ for any $\sigma\in\S_m$ and $c_{132}(4,2,5)=1$ because of the cluster
$(15243;1,3)$. 
Let 
\begin{equation}\label{eq:defC}
C_\sigma(q,t,z)=\sum_{n,k,i}c_\sigma(i,k,n)q^it^k\frac{z^n}{[n]_q!}
\end{equation}
be the corresponding $q$-exponential generating function.

Recall the definitions of $F_{\sigma}(q,z)$ and $\tilde{F}_{\sigma}(q,u,z)$ from Equations \eqref{eq:F-sigma} and \eqref{eq:tildeF}.  The following is a generalization of the cluster method that keeps track of the inversion number.

\begin{theorem}\label{thm:cluster}
For every $\sigma\in\S_m$,
$$\tilde{F}_\sigma(q,u,z)=\left(1-z-C_\sigma(q,u-1,z)\right)^{-1}.$$
In particular, setting $u=0$,
$$F_\sigma(q,z)=\left(1-z-C_\sigma(q,-1,z)\right)^{-1}.$$
\end{theorem}

\begin{proof}
By {\em concatenation} of two permutations $\rho\in\S_a$ and $\tau\in\S_b$ we mean any permutation $\pi\in\S_{a+b}$ such that $\st(\pi_1\dots\pi_a)=\rho$ and $\st(\pi_{a+1}\dots\pi_{a+b})=\tau$.
By {\em permutation with some marked occurrences of $\sigma$} we mean a permutation $\pi$ together with a set $S$ of indices so that $\pi$ has an occurrence of $\sigma$ starting at each $i\in S$, and possibly other occurrences as well.

Every permutation with marked occurrences of $\sigma$ can be decomposed uniquely as a concatenation of clusters with respect to $\sigma$ and single entries (those that do not belong to any marked occurrence). 
It follows that the generating function $\left(1-z-C_\sigma(q,t,z)\right)^{-1}$ counts permutations $\pi$ with marked occurrences of $\sigma$, where the exponent of $z$ is the length of $\pi$, the exponent of $t$ is the number of marked occurrences of $\sigma$ in $\pi$ (which is also an additive parameter), and the exponent of $q$ is the total number of inversions of $\pi$.
Indeed, by Lemma~\ref{lem:product}, the product of the $q$-exponential generating functions of two sets of permutations $A$ and $B$ keeps track of the number of inversions of all the possible concatenations of a permutation from $A$ with a permutation from $B$. For example, in the expansion of $\left(1-z-C_\sigma(q,t,z)\right)^{-1}=\sum_{i\ge0} (z+C_\sigma(q,t,z))^i$, the term $C_\sigma(q,t,z)\, z\, C_\sigma(q,t,z)$ corresponds to all possible concatenations consisting of a cluster, followed by a single entry, followed by another cluster.

To this generating function, a permutation $\pi\in\S_n$ with a total of $\ell$ occurrences of $\sigma$ contributes a term $t^k z^n/n!$ for each of the $\binom{\ell}{k}$ ways to mark $k$ of these $\ell$ occurrences, for every $k\le\ell$.
Thus, the contribution of $\pi$ to $\left(1-z-C_\sigma(q,u-1,z)\right)^{-1}$ is
$$\sum_{k=0}^\ell \binom{\ell}{k} q^{\inv(\pi)}(u-1)^k \frac{z^n}{[n]_q!}=q^{\inv(\pi)}u^\ell \frac{z^n}{[n]_q!},$$
which agrees with the contribution of $\pi$ to $\tilde{F}_\sigma(q,u,z)$.
Setting $u=0$, this contribution is $q^{\inv(\pi)} \frac{z^n}{[n]_q!}$ if $\pi$ avoids $\sigma$, and $0$ otherwise. 
It follows that $\left(1-z-C_\sigma(q,-1,z)\right)^{-1}=F_\sigma(q,z)$.
\end{proof}

\begin{theorem}[\cite{Elisurvey}]\label{thm:monotone}
$$F_{12\dots m}(q,z)=\left( \sum_{j\ge0} \frac{z^{jm}}{[jm]_q!} - \sum_{j\ge0} \frac{z^{jm+1}}{[jm+1]_q!} \right)^{-1}$$
\end{theorem}

\begin{proof}
By the second part of Theorem~\ref{thm:cluster}, it is enough to show that 
\begin{equation}\label{eq:simplifyclusters}
1-z-C_{12\dots m}(q,-1,z)=\sum_{j\ge0} \frac{z^{jm}}{[jm]_q!} - \sum_{j\ge0} \frac{z^{jm+1}}{[jm+1]_q!}.
\end{equation}
Clusters with respect to $12\dots m$ have no inversions, since the underlying permutation is of the form $12\dots n$. Thus,
$$C_{12\dots m}(q,-1,z)=\sum_{n,k}c_{12\dots m}(0,k,n)(-1)^k\frac{z^n}{[n]_q!}.$$
To simplify the alternating sum in $k$, note that $k$-clusters of length $n$ with respect to $12\dots m$ are in bijection with sequences
$(i_1,i_2,\dots,i_k)$ with $i_1=1$, $i_k=n-m+1$, and $1\le i_{j+1}-i_j\le m-1$ for all $j$, which are in bijection with compositions of $n-m$ with $k-1$ parts, each of size at most $m-1$. It follows that the ordinary generating function of $k$ clusters can be expressed as
$$\sum_{n,k}c_{12\dots m}(0,k,n)t^k x^n=\frac{tx^m}{1-t(x+x^2+\dots+x^{m-1})},$$
and so 
$$\sum_{n,k}c_{12\dots m}(0,k,n)(-1)^k x^n=\frac{-x^m}{1+x+x^2+\dots+x^{m-1}}=\frac{-x^m(1-x)}{1-x^m}=-\sum_{j\ge1}(x^{jm}-x^{jm+1}).$$
We conclude that 
$$\sum_{n,k}c_{12\dots m}(0,k,n)(-1)^k\frac{z^n}{[n]_q!}=-\sum_{j\ge1}\left(\frac{z^{jm}}{[jm]_q!}-\frac{z^{jm+1}}{[jm+1]_q!}\right),$$
from where Equation~\eqref{eq:simplifyclusters} follows.
\end{proof}

A pattern $\sigma\in\S_m$ is said to be {\em non-overlapping} (also called {\em minimally overlapping} in the literature) if no permutation of length $n<2m-1$ has two occurrences of $\sigma$. For example, $132$ is non-overlapping.

\begin{theorem}[\cite{Elisurvey,Raw}]\label{thm:1b}
Let $\sigma=\sigma_1\dots\sigma_m$ be a non-overlapping pattern with $\sigma_1=1$ and let $b=\sigma_m$. Then 
$$
\tilde{F}_{\sigma}(q,u,z)=\left(1-z-\sum_{k\ge1}\prod_{j=1}^{k-1}\binom{j(m-1)+m-b}{m-b}_q\frac{q^{k\inv(\sigma)}(u-1)^k z^{k(m-1)+1}}{[k(m-1)+1]_q!}\right)^{-1}.
$$
\end{theorem}

\begin{proof}
Since $\sigma$ is non-overlapping, all $k$-clusters with respect to $\sigma$ have length $n=k(m-1)+1$. Let $$a_k=\sum_i c_\sigma(i,k,k(m-1)+1)q^i$$ be the polynomial that counts $k$-clusters according to their number of inversions. Clearly $a_1=q^{\inv(\sigma)}$, so let us assume that $k\ge2$.
In every $k$-cluster $\pi$, the positions of the entries $\{1,2,\dots,b\}$ are forced, since they must belong to the occurrence of $\sigma$ that starts in position $1$, namely $\pi_1\pi_2\dots\pi_m$. On the other hand, the remaining $m-b$ entries in this occurrence of $\sigma$ can take any of the remaining $k(m-1)+1-b$ values. Such a choice can be encoded by a word $w$ of length 
$k(m-1)+1-b$ where $w_i=0$ if the value $i+b$ appears in $\pi_1\pi_2\dots\pi_m$ and $w_i=1$ otherwise.
Once this word $w$ is chosen, $\pi$ is uniquely determined by the relative order of the entries $\pi_m\pi_{m+1}\dots\pi_{k(m-1)+1-b}$, which, up to standardization, form an arbitrary $(k-1)$-cluster $\tau$ with respect to $\sigma$. Additionally, the number of inversions of $\pi$ equals the number of inversions among the first $m$ entries (which is $\inv(\sigma)$), plus the number of inversions among the entries $\pi_m\pi_{m+1}\dots\pi_{k(m-1)+1-b}$ (which is $\inv(\tau)$), plus the number of inversions between one of the first $m$ entries and one of the remaining entries (which is $\inv(w)$).
It follows that  $$a_k=q^{\inv(\sigma)} \left(\sum_\tau q^{\inv(\tau)} \right) \left(\sum_w q^{\inv(w)}\right),$$
where $\tau$ ranges over all $(k-1)$-clusters with respect to $\sigma$ and $w$ ranges over all $\{0,1\}$-valued words of length $k(m-1)+1-b$ having $m-b$ zeros. Using Equation~\eqref{eq:qbin}, we have
$$a_k=q^{\inv(\sigma)} a_{k-1} \binom{k(m-1)+1-b}{m-b}_q.$$
By induction on $k$, we get
$$a_k=q^{k \inv(\sigma)} \prod_{j=2}^k \binom{j(m-1)+1-b}{m-b}_q=q^{k \inv(\sigma)} \prod_{j=1}^{k-1} \binom{j(m-1)+m-b}{m-b}_q.$$

By Equation~\eqref{eq:defC} and using the above argument,
\begin{multline*}
C_{\sigma}(q,u-1,z)=\sum_{n,k,i}c_\sigma(i,k,n)q^i(u-1)^k\frac{z^n}{[n]_q!}=\sum_{k,i}c_\sigma(i,k,k(m-1)+1)q^i(u-1)^k\frac{z^{k(m-1)+1}}{[k(m-1)+1]_q!}\\
=\sum_{k}a_k(u-1)^k\frac{z^{k(m-1)+1}}{[k(m-1)+1]_q!}=\sum_{k}\prod_{j=1}^{k-1}\binom{j(m-1)+m-b}{m-b}_q\frac{q^{k\inv(\sigma)}(u-1)^k z^{k(m-1)+1}}{[k(m-1)+1]_q!}.$$
\end{multline*}
The result follows now from Theorem~\ref{thm:cluster}.
\end{proof}

\section{Growth rates}\label{section:growth}

In analogy to the fact that $\lim_{n\to\infty} \left(\frac{|\S_n(\sigma)|}{n!}\right)^{1/n}$ exists for every $\sigma$, we can prove the following result about the Mallows distribution.
Recall from Section \ref{sec:intro} that we write $P_n(\sigma,q)$ for the probability that a random permutation of length $n$ from the Mallows distribution with parameter $q$ avoids~$\sigma$.

\begin{theorem}\label{thm:rho}
For every $q>0$ and every pattern $\sigma$, the limit
\begin{equation}\label{eq:rate}
\rho(\sigma,q):=\lim_{n\to\infty} P_n(\sigma,q)^{1/n}\end{equation}
exists.
\end{theorem}

\begin{proof}[Combinatorial proof]
First we show that for every $m,n$, 
\begin{equation}\label{eq:P} P_{m+n}(\sigma,q) \le P_m(\sigma,q) P_n(\sigma,q).\end{equation}
This is a consequence of Lemma~\ref{lem:product}. Indeed, 
taking $A=\S_m(\sigma)$, $B=\S_n(\sigma)$, and
$$C=\{\pi\in\S_{m+n}:\st(\pi_1\dots\pi_m)\in A, \st(\pi_{m+1}\dots\pi_{m+n})\in B\},$$
 we have that
$\S_{m+n}(\sigma)\subseteq C$, since a permutation that avoids $\sigma$ must avoid $\sigma$ in the first $m$ and in the last $n$ entries.
 Thus,
$$P_{m+n}(\sigma,q)=\sum_{\pi\in \S_{m+n}(\sigma)} \frac{q^{\inv(\pi)}}{[m+n]_q!}\le \sum_{\pi\in C} \frac{q^{\inv(\pi)}}{[m+n]_q!}=\left(\sum_{\pi\in A} \frac{q^{\inv(\pi)}}{[m]_q!}\right) \left(\sum_{\pi\in B} \frac{q^{\inv(\pi)}}{[n]_q!}\right)=P_m(\sigma,q) P_n(\sigma,q).$$
Applying Fekete's lemma~\cite[Lemma 1.6]{vLW}, it follows that $\lim_{n\to\infty} P_n(\sigma,q)^{1/n}$ exists.
\end{proof}

\begin{proof}[Probabilistic proof]
Let $\pi=\pi_1\dots\pi_{n+m}$ be a random permutation from the Mallows($q$) distribution on $\S_n$ and let $\sigma\in\S_k$ for $n+m\geq k\geq1$.
For $j=1,\ldots,n+m-k+1$, let $A_j$ be the event that $\st(\pi_j\pi_{j+1}\dots\pi_{j+k-1})\neq\sigma$, that is, the segment of $\pi$ starting at position $j$ avoids $\sigma$.
Then we have
\[P_{n+m}(\sigma,q)= \mathbb{P}\left(\bigcap_{j=1}^{n+m-k+1}A_j\right)\leq\mathbb{P}\left(\bigcap_{i=1}^{n-k+1}A_i\cap\bigcap_{j=n+1}^{n+m-k+1}A_j\right)\]
by removing those events $A_{n-k+2},\ldots,A_n$ that overlap with the first $n$ and last $m$ locations of $\pi$.

By the weak dissociation property of the Mallows distribution, $|i-j|\geq k$ implies that $A_i$ and $A_j$ are independent, that is, $\mathbb{P}(A_i\cap A_j)=\mathbb{P}(A_i)\mathbb{P}(A_j)$.
Thus,
\[
\mathbb{P}\left(\bigcap_{i=1}^{n-k+1}A_i\cap\bigcap_{j=n+1}^{n+m-k+1}A_j\right)=\mathbb{P}\left(\bigcap_{i=1}^{n-k+1}A_i\right)\mathbb{P}\left(\bigcap_{j=n+1}^{n+m-k+1}A_j\right)=P_n(\sigma,q)P_m(\sigma,q),\]
where the last equality is a consequence of the consecutive heterogeneity property of the Mallows distribution. 
This proves~\eqref{eq:P}. To deduce that $\lim_{n\to\infty} P_n(\sigma,q)^{1/n}$ exists we now apply Fekete's lemma as before.
\end{proof}

\begin{remark}
The alternate proofs of Theorem \ref{thm:rho} illustrate the interplay between combinatorial and probabilistic techniques on display throughout the paper.
\end{remark}

The quantity $\rho(\sigma,q)$ defined in Theorem~\ref{thm:rho} is called the {\em growth rate} of $\sigma$, and it is a function of $q$. From the definition it follows that $0\le \rho(\sigma,q)\le 1$ for all $\sigma$ and $q$.
It is clear from Equation~\eqref{eq:reversal} that $\rho(\sigma^r,q)=\rho(\sigma,1/q)$. To draw plots of $\rho(\sigma,q)$ for $0\le q<\infty$, it will be convenient to use the change of variables $x=\frac{q-1}{q+1}$, or equivalently $q=\frac{1+x}{1-x}$. With this transformation, $x$ ranges between $-1$ and $1$, and the values $-1,0,1$ for $x$ correspond to $0,1,\infty$ for $q$, respectively.
Additionally, the symmetry between $q$ and $1/q$ resulting from reversing the pattern corresponds to the symmetry between $x$ and $-x$, and thus the graph for $\sigma^r$ is obtained by reflecting the graph for $\sigma$ with respect to the vertical axis.

Using a well-known fact of singularity analysis~\cite[Theorem IV.7]{FS}, $\rho(\sigma,q)^{-1}$ equals the radius of convergence of $F_\sigma(q,z)$ as a function of a complex variable $z$. Additionally, if this radius is finite, then $F_\sigma(q,z)$ has a real singularity at $z=\rho(\sigma,q)^{-1}$, by Pringsheim's Theorem~\cite[Theorem IV.6]{FS} and the fact that this generating function has non-negative coefficients.

By Theorem~\ref{thm:cluster}, defining 
\begin{equation}\label{eq:omega-sigma}\omega_\sigma(q,z)=1-z-C_\sigma(q,-1,z)\end{equation}
allows us to write $F_\sigma(q,z)=\frac{1}{\omega_\sigma(q,z)}$. In particular, the smallest singularity of $F_\sigma(q,z)$, when it exists, is either a zero or a singularity of $\omega_\sigma(q,z)$. In some cases there is no such singularity, that is, $F_\sigma(q,z)$ has infinite radius of convergence.

\section{Monotone patterns}\label{section:monotone}

In this section we study the growth rate $\rho(\sigma,q)$ when $\sigma=12\dots m$ for any $m\ge3$.  (When $m=2$, there is exactly 1 permutation of $n$, namely $n(n-1)\dots1$, that avoids $12$, yielding $\rho(12,q)\equiv0$.) Note that $\rho(m\dots 21,q)=\rho(12\dots m,1/q)$, so our results apply to the monotone decreasing pattern as well.

By the arguments in Section~\ref{section:growth},  $\rho(12\dots m,q)^{-1}$ equals the smallest positive singularity of the generating function $F_{12\dots m}(q,z)=1/\omega_{12\dots m}(q,z)$ given in Theorem~\ref{thm:monotone}, for $\omega_{12\dots m}(q,z)$ as defined in \eqref{eq:omega-sigma}.
Let $a_n=\frac{z^n}{[n]_q!}$, so that 
\begin{equation}\label{eq:omega_mon}\omega_{12\dots m}(q,z)=\sum_{j\ge0} \left(\frac{z^{jm}}{[jm]_q!} - \frac{z^{jm+1}}{[jm+1]_q!}\right)=\sum_{j\ge0} (a_{jm}-a_{jm+1}).
\end{equation} 
Note that $\frac{a_{n}}{a_{n-1}}=\frac{z}{[n]_q}$. We consider two cases depending on whether $q$ is larger or smaller than $1$.

\medskip

\noindent {\bf Case $q\ge1$.} In this case, $[n]_q=1+q+\dots+q^{n-1}$ goes to infinity as $n\to\infty$. Thus, for any fixed $z$, we have $\frac{a_{n}}{a_{n-1}}<1$ for $n$ large enough. It follows that the series~\eqref{eq:omega_mon}
 converges, and thus $\rho(12\dots m,q)^{-1}$ is the smallest positive zero of $\omega_{12\dots m}(q,z)$, which can be easily approximated by truncating the sum, since the factorial in the denominator makes the terms quickly go to zero.

Figure~\ref{fig:monotone_x01} shows the plots of $\rho(12\dots m,q)$ for $q\in[1,\infty)$ and $m\in\{3,4,5\}$, with the horizontal axis rescaled by
$q=\frac{1+x}{1-x}$, so that $x\in[0,1)$. The values of $\rho(12\dots m,q)$ have been obtained by truncating the series~\eqref{eq:omega_mon} at $j=10$ and finding its zero numerically.

\begin{figure}[h]
\centering
\includegraphics[height=6cm]{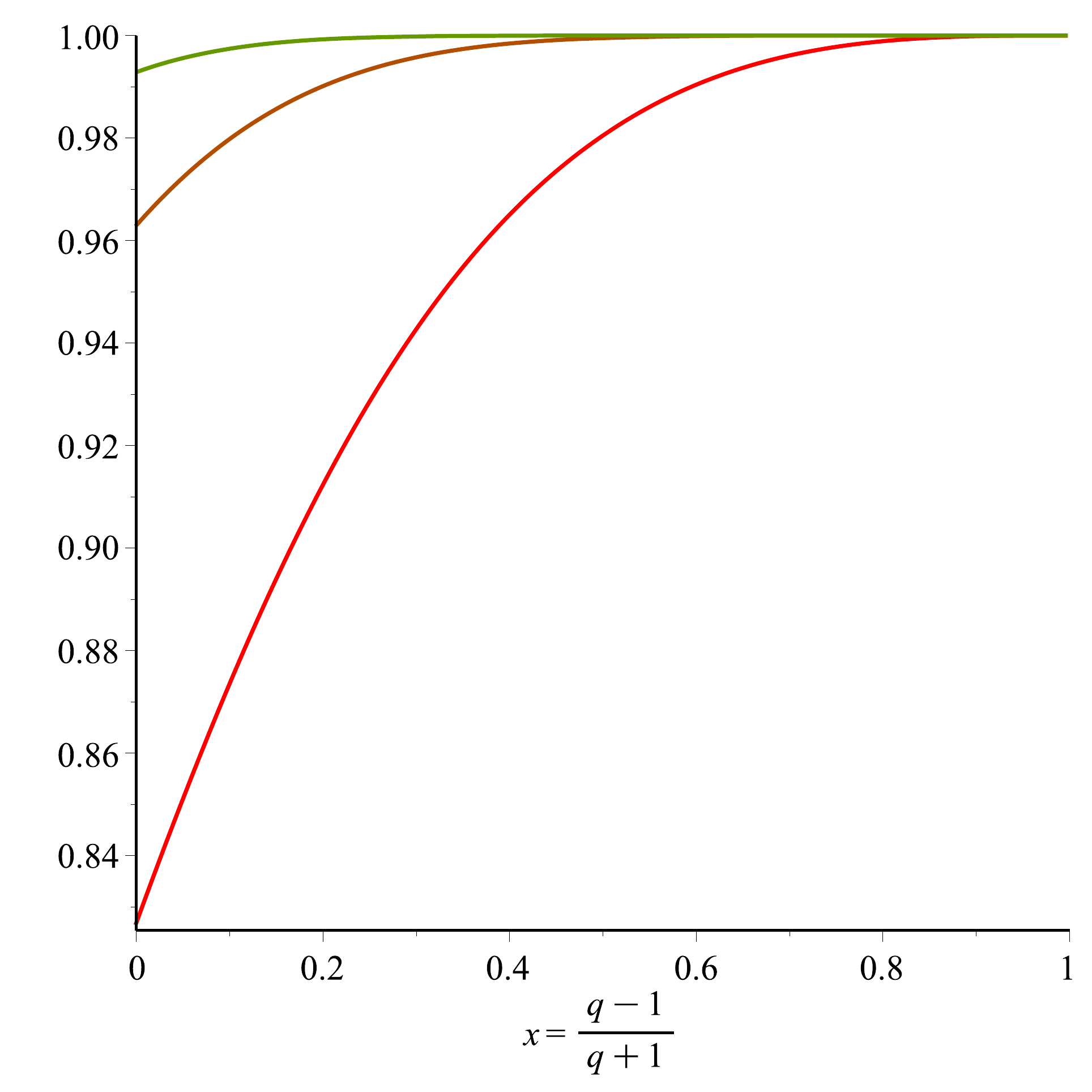}
\caption{Plots of $\rho(12\dots m,q)$ with $q=\frac{1+x}{1-x}>1$ for $m=3$ {\color[rgb]{1,0,0} (red)}, $m=4$ {\color[rgb]{.7,.3,0} (brown)} and $m=5$ {\color[rgb]{.4,.6,0} (green)}.}
\label{fig:monotone_x01}
\end{figure}

If we now let $m$ go to infinity, and we write the smallest positive zero of $\omega_{12\dots m}(q,z)$ as $z_0=1+\epsilon$, then we have that
$\epsilon= \frac{1}{[m]_q!}-\frac{1}{[m+1]_q!}+ O\left(\frac{1}{[m+2]_q!}\right)$, yielding
$$\rho(12\dots m,q)= 1-\frac{1}{[m]_q!}+\frac{1}{[m+1]_q!}+O\left(\frac{1}{[m+2]_q!}\right).$$

\medskip

\noindent {\bf Case $q<1$.} In this case,
$$\frac{a_{n}}{a_{n-1}}=\frac{z(1-q)}{1-q^{n}}\rightarrow z(1-q)$$ 
as $n\to\infty$. Thus, the radius of convergence of $\omega_{12\dots m}(q,z)$ is $\frac{1}{1-q}$. 

For $q$ above a certain threshold, the function $\omega_{12\dots m}(q,z)$ has a zero $z_0$ with $z_0<\frac{1}{1-q}$. In this case, the smallest positive such zero can be estimated by truncating the series as before, giving the growth rate $\rho(12\dots m,q)=z_0^{-1}$. However, for small values of $q$, the function $\omega_{12\dots m}(q,z)$ does not have a positive zero smaller than $\frac{1}{1-q}$, so we use a different method to approximate its smallest zero~$z_0$.

We first approximate the tail of $\omega_{12\dots m}(q,z)$ as follows. First note that
$$[n]_q!=\frac{\prod_{i=1}^n(1-q^i)}{(1-q)^n}.$$
As $n\to\infty$, the numerator approaches a constant, whose reciprocal we denote by 
\begin{equation}\label{eq:cq}
c_q:=\frac{1}{\prod_{i=1}^\infty(1-q^i)}.
\end{equation}
Using the approximation 
\begin{equation}\label{eq:approximation} 
a_n=\frac{z^n}{[n]_q!}\approx c_q(z(1-q))^n, \end{equation}
we get, for large $K$,
$$\sum_{j\ge K} (a_{jm}-a_{jm+1}) \approx \sum_{j\ge K} c_q(z(1-q))^{jm}(1-z(1-q))=c_q(1-z(1-q))\frac{(z(1-q))^{Km}}{1-(z(1-q))^m},$$
and so 
\begin{equation}\label{eq:omega_trunc_tail}
\omega_{12\dots m}(q,z)\approx \sum_{j=0}^{K-1} (a_{jm}-a_{jm+1})+c_q(1-z(1-q))\frac{(z(1-q))^{Km}}{1-(z(1-q))^m}.
\end{equation}
We can find the smallest positive zero of the right-hand side of~\eqref{eq:omega_trunc_tail} numerically, obtaining an approximation of the smallest zero $z_0$ of $\omega_{12\dots m}(q,z)$.

Using this method for $m=3$, we have computed an approximation of $\rho(123,q)$ as the reciprocal of the smallest positive zero of the right-hand side of~\eqref{eq:omega_trunc_tail} for $K=40$. The resulting plot appears in Figure~\ref{fig:123_x_neg_tail}. We note that the smallest zero satisfies $z_0<\frac{1}{1-q}$ roughly when $0.4124<q<1$ (equivalently,  $-0.4160<x<0$ with the rescaling $q=\frac{1+x}{1-x}$). However, for smaller values of $q$, our method still yields an approximation of $\rho(123,q)$.

\begin{figure}[h]
\centering
\includegraphics[height=6cm]{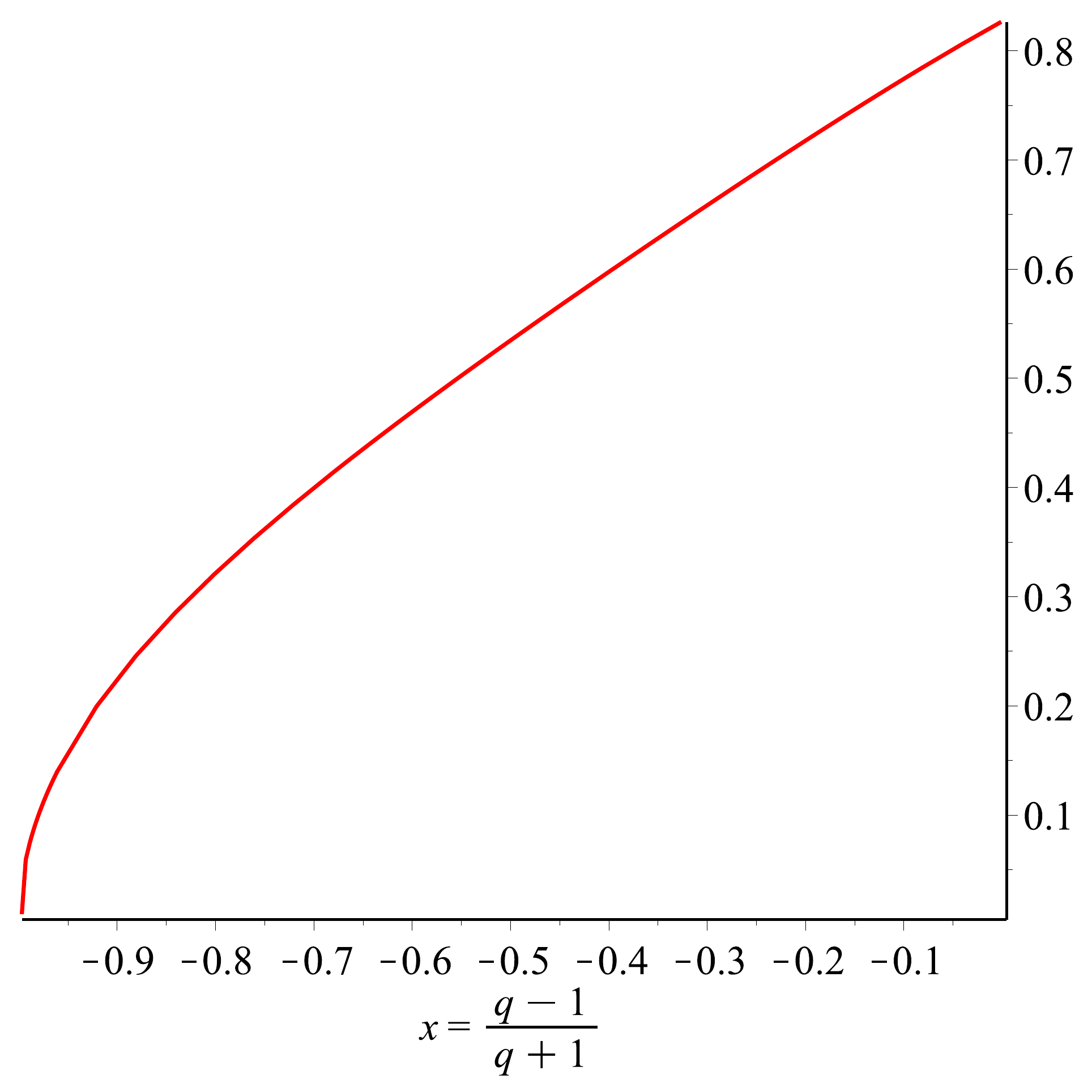}
\caption{Plot of $\rho(123,q)$ for $0<q<1$ obtained by computing $z_0^{-1}$ where $z_0$ is the smallest positive zero of $\omega_{123}(q,z)$, approximated by~\eqref{eq:omega_trunc_tail} with $K=40$. The horizontal axis has been rescaled by $q=\frac{1+x}{1-x}$.}
\label{fig:123_x_neg_tail}
\end{figure}

In the limit as $q\to0$, the Mallows distribution assigns probability one to monotone increasing permutations, and so $\lim_{q\to0} \rho(12\dots m,q)=0$, which agrees with Figure~\ref{fig:123_x_neg_tail}.

For comparison purposes, Figure~\ref{fig:123_x_all_tail} shows the plots of $P_{n}(123,q)^{1/n}$ for $n=30$ and $n=50$. As $n$ goes to infinity, these plots approach $\rho(123,q)$. In general, $P_{n}(12\dots m,q)^{1/n}$ can be obtained by computing the coefficient of $z^n$ in the expression for $F_{12\dots m}(q,z)$ in Theorem~\ref{thm:monotone}.

\begin{figure}[h]
\centering
\includegraphics[height=6cm]{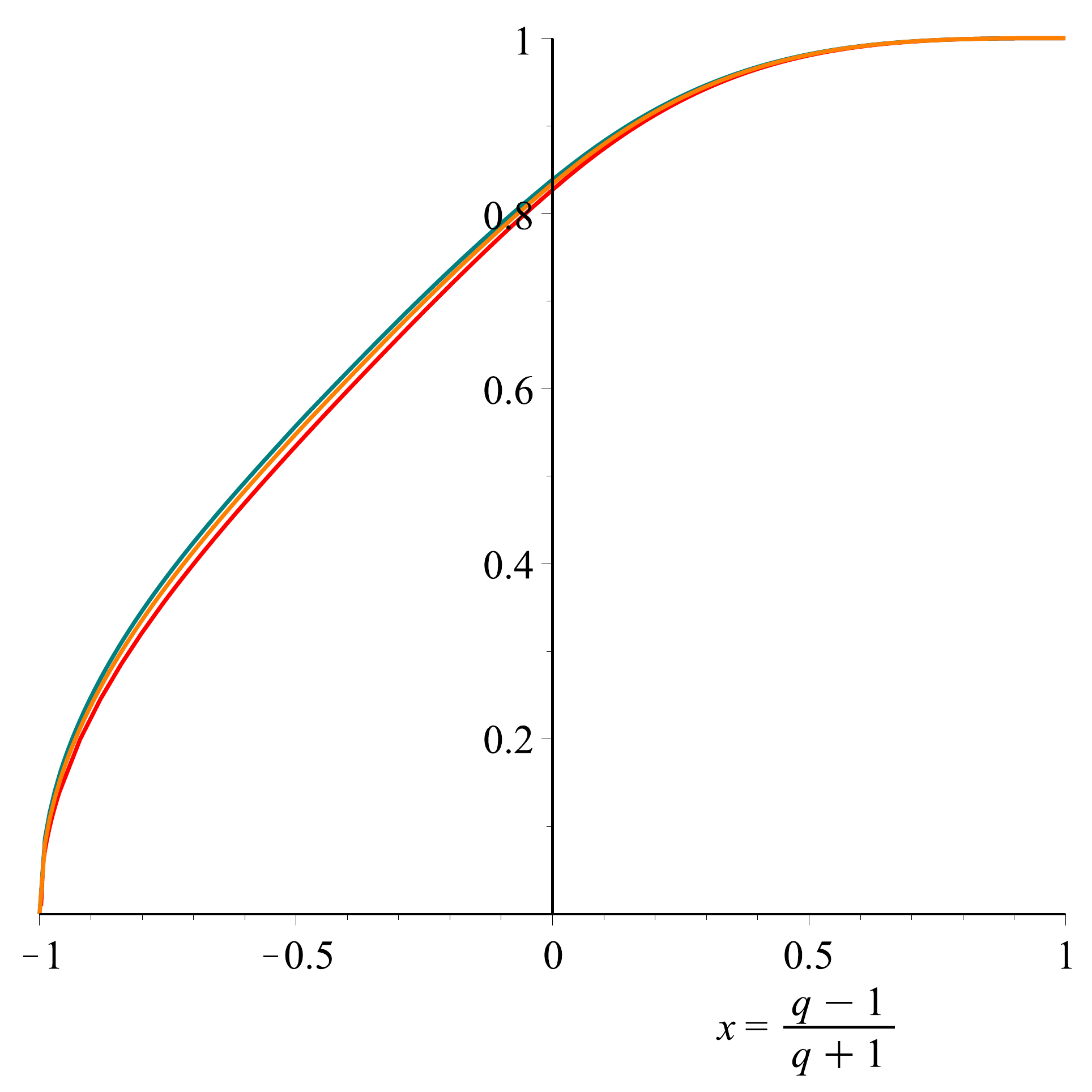}
\caption{Plots of $P_{30}(123,q)^{1/30}$ {\color[rgb]{0,.5,.5} (turquoise)} and $P_{50}(123,q)^{1/50}$ {\color[rgb]{1,0.5,0} (orange)} as functions of $x$, overlaid with $\rho(123,q)$ {\color[rgb]{1,0,0} (red)} found using our method.}
\label{fig:123_x_all_tail}
\end{figure}

Let us finish this section by analyzing the behavior of the growth rate $\rho(12\dots m,q)$ for fixed $q<1$ as $m$ goes to infinity. For large $m$, we can use the approximation
$$\omega_{12\dots m}(q,z)\approx 1-z+c_q(1-z(1-q))\frac{(z(1-q))^{m}}{1-(z(1-q))^m}.$$
Writing the smallest positive zero of $\omega_{12\dots m}(q,z)$ as $z_0=1+\epsilon$, then we have that
$$\epsilon\approx c_q(1-(1+\epsilon)(1-q))\frac{(1+\epsilon)^m(1-q)^{m}}{1-(1+\epsilon)^m(1-q)^{m}}.$$
The possibility that $(1+\epsilon)(1-q)\ge1$ can be easily ruled out because it gives a contradition with the fact that $\epsilon\to0$ as $m\to\infty$. Thus, using that $(1+\epsilon)(1-q)<1$, we can write 
$$\epsilon\approx c_q(1-(1+\epsilon)(1-q))(1+\epsilon)^m(1-q)^{m}=c_q(q-\epsilon(1-q))(1+\epsilon)^m(1-q)^{m}$$
from where we get
$$\epsilon= c_q q(1-q)^{m}+O(m(1-q)^{2m}).$$

\section{Non-overlapping patterns starting with a $1$}\label{section:132}

In this section we study the growth rates $\rho(\sigma,q)$ when $\sigma=\sigma_1\cdots\sigma_m$ is a non-ovelapping pattern with $\sigma_1=1$. 
The generating function $F_{\sigma}(q,z)=1/\omega_{\sigma}(q,z)$ for such a pattern is given in Theorem~\ref{thm:1b}, and $\rho(\sigma,q)^{-1}$ equals the smallest positive singularity of this generating function.

\subsection{The pattern $132$}
\label{pattern:132}

Let $$h_k=\frac{q^{k} z^{2k+1}}{[2k+1]_q \prod_{j=1}^k [2j]_q},$$ so that, by Theorem~\ref{thm:1b},
\begin{equation}\label{eq:omega_1b}
\omega_{132}(q,z)=1-z-\sum_{k\ge1}\left(\prod_{j=1}^{k-1} [2j+1]_q\right) \frac{q^{k}(-1)^k z^{2k+1}}{[2k+1]_q!}
=1-z-\sum_{k\ge1} (-1)^k h_k.
\end{equation}
We consider the quotient
\begin{equation}\label{eq:bkratio}
\frac{h_{k}}{h_{k-1}}=\frac{qz^2[2k-1]_q}{[2k+1]_q[2k]_q},
\end{equation}
and distinguish two cases depending on whether $q$ is larger or smaller than $1$.

\medskip

\noindent {\bf Case $q\ge1$.}  In this case, $\frac{h_{k}}{h_{k-1}}\to 0$ as $k\to\infty$, and so the series $\omega_{132}(q,z)$ is convergent. It follows that $\rho(132,q)^{-1}$ is the smallest positive zero of $\omega_{132}(q,z)$, and it can be approximated by truncating the sum~\eqref{eq:omega_1b}.

Figure~\ref{fig:132_x} shows a plot of $\rho(132,q)$ with the rescaling $q=\frac{1+x}{1-x}$. For $q\ge1$, $\rho(132,q)$ has been found by taking the reciprocal of the smallest zero of a truncation of $\omega_{132}$.

\begin{figure}[h]
\centering
\includegraphics[height=6cm]{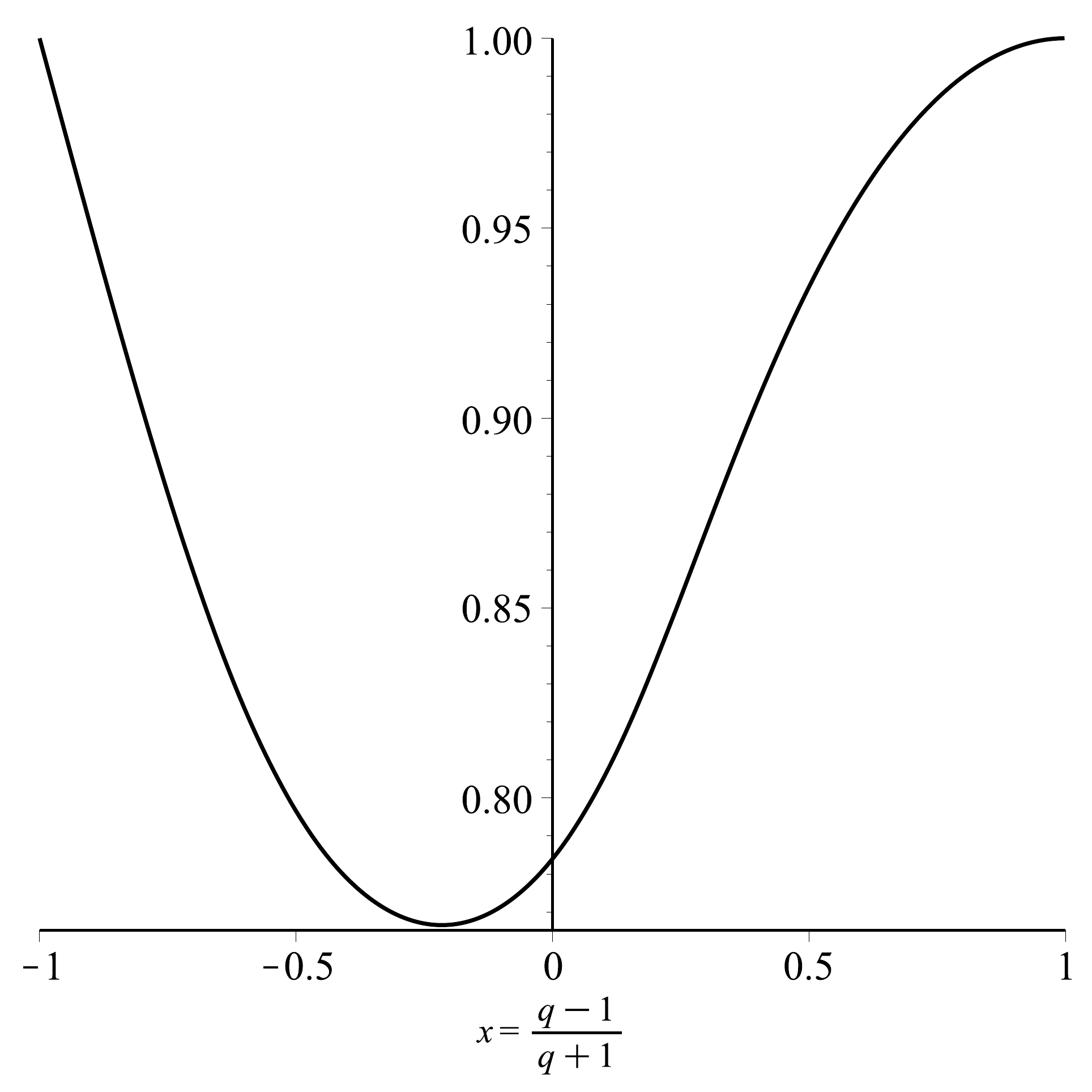}
\caption{Plot of $\rho(132,q)$ as a function of $x$, obtained by taking the reciprocal of the smallest positive zero of the series for $\omega_{132}$ truncated at $j=20$.}
\label{fig:132_x}
\end{figure}

If we write $z_0=1+\epsilon$, we have that $\omega_{132}(z_0)=0$ is equivalent to
\begin{equation}\label{eq:eps132}
\epsilon=\sum_{k\ge1}\frac{(-1)^{k-1} q^{k} (1+\eps)^{2k+1}}{[2k+1]_q \prod_{j=1}^k [2j]_q}.
\end{equation}

As $q\to\infty$, the dominant term of $[n]_q$ is $q^{n-1}$, and so $\frac{q^{k}}{[2k+1]_q \prod_{j=1}^k [2j]_q}=O(q^{-k(k+1)})$. It follows that
$$\eps=\frac{q(1+\epsilon)^3}{[2]_q[3]_q}+O(q^{-6})=\frac{q}{[2]_q[3]_q}+\frac{3q^2}{[2]_q^2[3]_q^2}+O(q^{-6})=q^{-2}-2q^{-3}+O(q^{-4})$$
and
$$\rho(132,q)=(1+\eps)^{-1}=1-q^{-2}+2q^{-3}+O(q^{-4}).$$
Equivalently, as $q\to0$,
$$\rho(231,q)=1-q^2+2q^{3}+O(q^{4}).$$

\medskip

\noindent {\bf Case $q<1$. } In this case,
$$\frac{h_{k}}{h_{k-1}}\rightarrow q(1-q)z^2$$ 
as $k\to\infty$. Thus, for $z<\frac{1}{\sqrt{q(1-q)}}$, the series $\omega_{132}(q,z)$ is convergent. We will show that, for every $q<1$, this series has a zero with $z_0<\frac{1}{\sqrt{q(1-q)}}$; therefore, this zero equals $\rho(132,q)^{-1}$ and can be approximated by truncating the series as before. 

In fact, we will show that $\omega_{132}(q,z)$ always has a zero with $1<z_0<\sqrt{3}$, which implies the above statement because
$\sqrt{3}<2\le\frac{1}{\sqrt{q(1-q)}}$ for $q\in(0,1)$. First, note that for $0<z<\sqrt{3}$, Equation~\eqref{eq:bkratio} gives
$$\frac{h_{k}}{h_{k-1}}\le \frac{3q[2k-1]_q}{[2k+1]_q[2k]_q}<\frac{3q}{[2k+1]_q}\le \frac{3q}{[3]_q}\le1$$
for every $k\ge1$ and $q\in[0,1]$, where in the last inequality we have used the fact that $3\le \frac{1+q+q^2}{q}$ for all $q\in[0,1]$,
which follows from the fact that $0\le (1-q)^2$.  By truncating an alternating series whose terms decrease in absolute value, we have that, for any odd $N$,
$$\omega_{132}(z)<1-z-\sum_{k=1}^{N}\frac{(-1)^k q^{k} z^{2k+1}}{[2k+1]_q \prod_{j=1}^k [2j]_q}=:p_N(z)$$
for $0<z<\sqrt{3}$. Since $\omega_{132}(0)=p_N(0)=1$, it follows that $z_0<z_1$, where $z_1$ is the smallest positive zero of $p_N(z)$.  Finally, by truncating at $N=5$ (any larger odd $N$ works too), we obtain that $z_1<\sqrt{3}$, as can be seen by checking that $p_5(\sqrt{3})<0$ for all $q\in(0,1)$.

Thus, we have proved that $0<z_0<z_1<\sqrt{3}$ as desired.

\medskip

The above argument shows that for $q\ge1$, the value of $\rho(132,q)$ can be approximated by taking the reciprocal of the smallest zero of a truncation of $\omega_{132}$, as we have done in Figure~\ref{fig:132_x}.
An interesting feature of this graph is that the minimum of $\rho(132,q)$ is attained for some value of $q$ strictly between $0$ and $\infty$. A numerical approximation of this value is $q\approx 0.6447045$, which gives a growth rate of $\rho(132,q)\approx0.7665452$.  
As $q\to0$, Equation~\eqref{eq:eps132} gives the approximation
$$\eps=
\frac{q(1+\epsilon)^3}{[2]_q[3]_q} - \frac{q^2 (1+\eps)^5}{[2]_q[4]_q[5]_q} +O(q^3)=
\frac{q}{[2]_q[3]_q}+\frac{3q^2}{[2]_q^2[3]_q^2}-\frac{q^2}{[2]_q[4]_q[5]_q}+O(q^3)=q+O(q^3),$$
and
$$\rho(132,q)=(1+\eps)^{-1}=1-q+q^{2}+O(q^3).$$
Note that some care must be taken when expanding out expressions containing $q$-analogues in powers of $q$.

\subsection{Generalizations}
\label{section:generalizations}
The arguments in the previous subsection generalize to non-overlapping patterns that start with a $1$. Letting $\sigma\in\S_m$ be a non-overlapping pattern with $\sigma_1=1$ and $\sigma_m=b$, Theorem~\ref{thm:1b} gives 
$\omega_\sigma(q,z)=1-z-\sum_{k\ge1} (-1)^k h^\sigma_k$, where we define
\begin{equation}\label{eq:h}
h^\sigma_k:=\prod_{j=1}^{k-1}\binom{j(m-1)+m-b}{m-b}_q\frac{q^{k\inv(\sigma)} z^{(m-1)k+1}}{[k(m-1)+1]_q!}.
\end{equation}
The quotient of consecutive terms can be simplified as
\begin{align*}\label{eq:ckratio}
\frac{h^\sigma_{k+1}}{h^\sigma_{k}}&=\binom{k(m-1)+m-b}{m-b}_q \frac{q^{\inv(\sigma)}z^{m-1}}{[k(m-1)+m]_q\cdots [k(m-1)+2]_q}\\
&=\frac{[k(m-1)+m-b]_q\cdots [k(m-1)+2]_q[k(m-1)+1]_q \, q^{\inv(\sigma)}z^{m-1}}{[m-b]_q!\, [k(m-1)+m]_q\cdots [k(m-1)+3]_q[k(m-1)+2]_q}\\
&=\frac{[k(m-1)+1]_q\, q^{\inv(\sigma)}z^{m-1}}{[m-b]_q! \prod_{i=0}^{b-1}[k(m-1)+m-i]_q}.
\end{align*}

\medskip

\noindent {\bf Case $q\ge1$.}  In this case, noting that $b\ge2$ we see that $\frac{h^\sigma_{k+1}}{h^\sigma_{k}}\to 0$ as $k\to\infty$. It follows that the series $\omega_{\sigma}(q,z)$ is convergent, and that $\rho(\sigma,q)^{-1}$ is the smallest positive zero  of $\omega_{\sigma}(q,z)$, which can be approximated by truncating the sum, analogously to what we did for the pattern $132$. In Figure~\ref{fig:1b_m4}, the parts of the graphs of $\rho(\sigma,q)$ for $q\ge1$ (equivalently, $x\ge0$) have been computed in this way truncating at $k=15$.

\medskip
\noindent {\bf Case $q<1$. } In this case,
$$\frac{h^\sigma_{k+1}}{h^\sigma_{k}}\rightarrow \frac{(1-q)^{b-1} q^{\inv(\sigma)} z^{m-1}}{[m-b]_q!}$$ 
as $k\to\infty$. Thus, for 
\begin{equation}\label{eq:bound} z<\left(\frac{[m-b]_q!}{(1-q)^{b-1} q^{\inv(\sigma)}}\right)^{\frac{1}{m-1}},
\end{equation} the series $\omega_{\sigma}(q,z)$ is convergent. If this series has a positive zero $z=z_0$ satisfying \eqref{eq:bound}, then this zero equals $\rho(\sigma,q)^{-1}$ and can be approximated by truncating the series as we did for the pattern $132$. 

However, even without the assumption that the smallest positive zero of $\omega_{\sigma}(q,z)$ satisfies~\eqref{eq:bound}, we can use the following method to estimate this zero. First, we write~\eqref{eq:h} as
$$h^\sigma_k=\frac{q^{k\inv(\sigma)} z^{(m-1)k+1}}{[k(m-1)+1]_q!\left([m-b]_q!\right)^{k-1}}\prod_{j=1}^{k-1}\prod_{i=1}^{m-b}[j(m-1)+i]_q.$$
As in the case of monotone patterns, we use the approximation $$[k(m-1)+1]_q!\approx \frac{1}{c_q (1-q)^{k(m-1)+1}}$$ as $k\to\infty$, with $c_q$ as defined in Equation~\eqref{eq:cq}. Similarly, we can approximate 
$$\prod_{j=1}^{k-1}\prod_{i=1}^{m-b}[j(m-1)+i]_q\approx\frac{1}{d_{q,m,b}(1-q)^{(k-1)(m-b)}}$$
as $k\to\infty$, with $d_{q,m,b}$ being the constant defined by $$d_{q,m,b}:=\frac{1}{\prod_{j=1}^{\infty}\prod_{i=1}^{m-b}(1-q^{j(m-1)+i})}.$$
Thus, for large $k$, we can write
$$h^\sigma_k\approx \frac{c_q (1-q)^{m+1-b} [m-b]_q!\, z}{d_{q,m,b}}\left(\frac{q^{\inv(\sigma)}(1-q)^{b-1}z^{m-1}}{[m-b]_q!} \right)^k.$$
And for $K$ large enough,
\begin{equation}\label{eq:omega1b_trunc_tail}\omega_\sigma(q,z)\approx 1-z-\sum_{k=1}^{K-1} (-1)^k h^\sigma_k-\frac{c_q (1-q)^{m+1-b} [m-b]_q!\, z}{d_{q,m,b}}\frac{\left(-\frac{q^{\inv(\sigma)}(1-q)^{b-1}z^{m-1}}{[m-b]_q!} \right)^K}{1+\frac{q^{\inv(\sigma)}(1-q)^{b-1}z^{m-1}}{[m-b]_q!}}.
\end{equation}

In Figure~\ref{fig:1b_m4}, $\rho(\sigma,q)$ for $q<1$ (equivalently, $x<0$) has been approximated for the three  non-overlapping patterns of length 4 starting with a 1 as the reciprocal of the smallest positive zero of~\eqref{eq:omega1b_trunc_tail} with $K=15$. 

\begin{figure}[h]
\centering
\includegraphics[height=6cm]{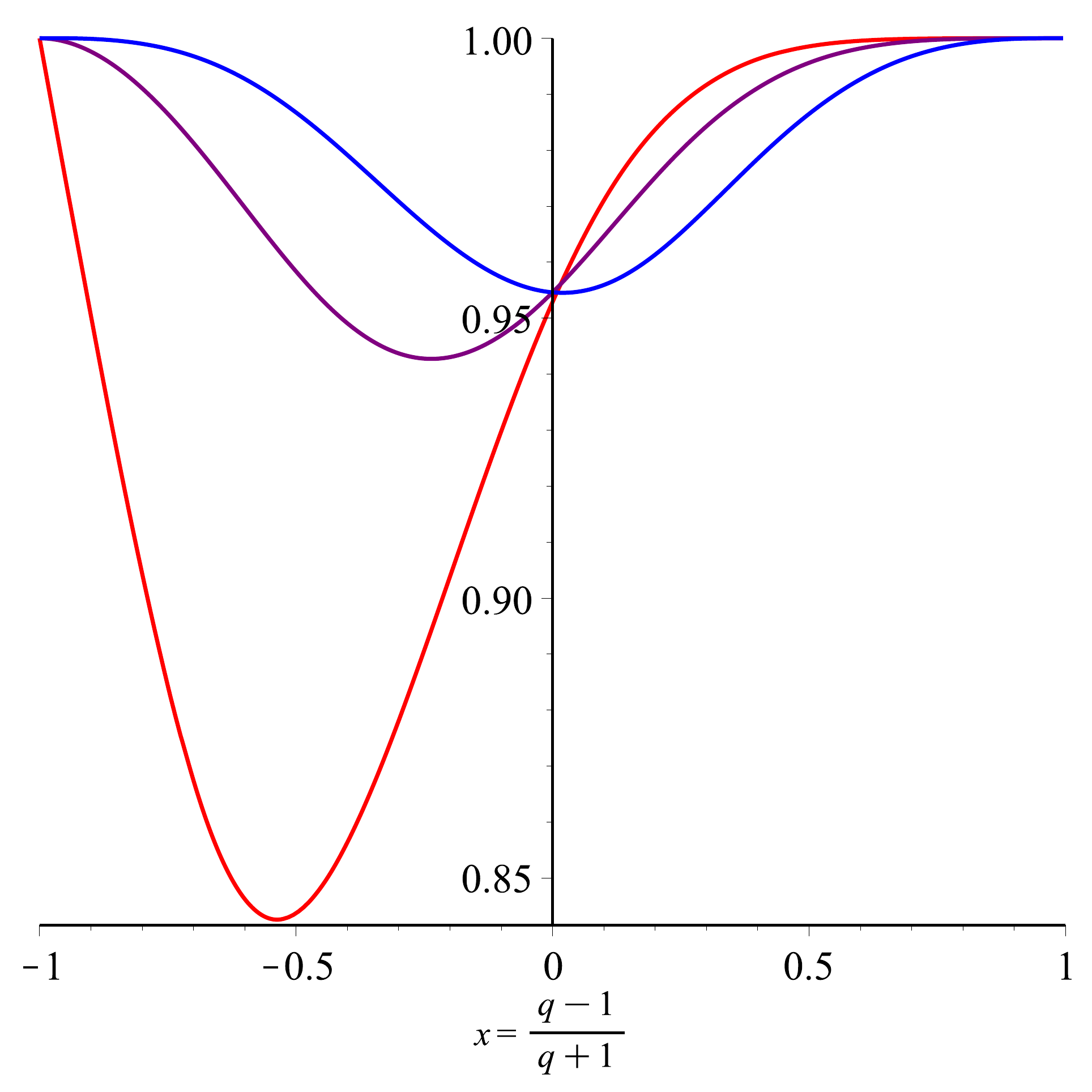}
\caption{Plots of $\rho(\sigma,q)$ with $q=\frac{1+x}{1-x}$ for $\sigma=1243$ {\color[rgb]{1,0,0} (red)}, $\sigma=1342$ {\color[rgb]{.5,0,0.5} (purple)}, and $\sigma=1432$ {\color[rgb]{0,0,1} (blue)}. The intersection of these curves with the $y$-axis coincides for $1342$ and $1432$ (see~\cite{EliCMP} for an explanation) but it is slightly lower for the pattern $1243$. We also point out that the curve for $1432$ is not perfectly symmetric with respect to the $y$-axis.}
\label{fig:1b_m4}
\end{figure}

\section{Comparisons among patterns}\label{section:compare}

\begin{theorem}\label{comparison:theorem}
For $q\ge1$ and every $n$, we have $$P_n(132,q)\le P_n(123,q),$$ and so $\rho (132,q)\le \rho(123,q)$ for $q\ge1$.
Equivalently, for $q\le1$ and every $n$, we have $P_n(231,q)\le P_n(321,q)$, and so $\rho (231,q)\le \rho(321,q)$ for $q\le1$.
\end{theorem}

\begin{proof}
In~\cite{EliNoy}, the authors give an injection $\Psi:\S_n\setminus\S_n(123)\to\S_n\setminus\S_n(132)$ defined as follows. If $\pi$ contains both $123$ and $132$, let $\Psi(\pi)=\pi$, otherwise, let $\Psi(\pi)$ be the permutation obtained by traversing $\pi$ from left to right and replacing each occurrence of $123$ with one of $132$, by switching the entries playing the role of 2 and 3. It is clear that $\inv(\Psi(\pi))\ge \inv(\pi)$ for every $\pi\in\S_n\setminus\S_n(123)$, and so $q^{\inv(\Psi(\pi))}\ge q^{\inv(\pi)}$ for $q\ge1$. Since $\Psi$ is an injection, we have
$$\sum_{\tau\in\S_n\setminus\S_n(132)} q^{\inv(\tau)}\ge \sum_{\pi\in\S_n\setminus\S_n(123)} q^{\inv(\Psi(\pi))}\ge \sum_{\pi\in\S_n\setminus\S_n(123)} q^{\inv(\pi)}$$ for $q\ge1$, and so
$$\sum_{\tau\in\S_n(132)} q^{\inv(\tau)}\le \sum_{\pi\in\S_n(123)} q^{\inv(\pi)}.$$
Dividing by $[n]_q!$ on both sides, we get that  $P_n(132,q)\le P_n(123,q)$ for $q\ge1$. The other statements now follow immediately.
\end{proof}

In Section \ref{section:Stein}, we extend our discussion of pattern avoidance, which specializes to the setting in which a pattern occurs exactly zero times, by considering the distribution of the number of times a pattern occurs in a random permutation.
For example, by looking just at the expected number of occurrences of the patterns $132$ and $231$ in a random permutation of size~$n$ from the Mallows distribution with parameter $q>0$, we obtain $(n-2) q/[3]_q!$ and $(n-2)q^2/[3]_q!$, respectively. 
From this observation, we see that there are more occurrences \emph{on average} of $231$ than $132$ when $q >1$ and vice versa for $q < 1$.
It is natural to conjecture that the probability of avoiding $231$ ought to be smaller than the probability of avoiding $132$ when $q>1$, and vice versa when $0<q<1$. 
The corresponding statement in the limit, at the level of growth rates, can be verified by looking at Figure \ref{fig:132_x}, which shows that for $q=\frac{1+x}{1-x}<1$ (equivalently, $x<0$), we have
$\rho(132,q)<\rho(132,1/q)=\rho(231,q)$, since the transformation $q\leftrightarrow 1/q$ corresponds to the reflection $x\leftrightarrow -x$ in the picture.

There are, however, interesting characteristics other than the growth rate $\rho(\sigma,q)$ of the pattern avoidance probability, such as the asymptotic distribution of the number of occurrences of a pattern and its rate of convergence.
These properties rely on calculations other than the marginal probability $P_n(\sigma,q)$ needed to obtain the expected number of occurrences, and therefore reveal more intricate structural features of pattern avoiding permutations.
We illustrate this further in our comparison of patterns $1432$, $2341$ and $2413$ in Sections \ref{section:bounds} and \ref{section:Stein}.

\subsection{Monotonicity of avoiding monotone patterns}

Increasing the value of $q$ in the Mallows distribution gives higher probability to permutations with more inversions. 
It is therefore natural to conjecture that the probability of avoiding the pattern $12\dots m$ increases with $q$. This conjecture is supported by our numerical evidence in Figures~\ref{fig:monotone_x01}, \ref{fig:123_x_neg_tail}, and~\ref{fig:123_x_all_tail}.

\begin{conjecture}\label{conj:123increasing}
If $q<q'$, then $P_n(12\dots m,q)<P_n(12\dots m,q')$ and $\rho(12\dots m,q)<\rho(12\dots m,q')$.
\end{conjecture}

This conjecture would follow if we could show that the sequence $$\frac{|\{\pi\in\S_n(12\dots m):\inv(\pi)=k\}|}{|\{\pi\in\S_n:\inv(\pi)=k\}|}$$ is weakly increasing in $k$. 
Note that, using inversion tables to represent permutations, the above denominator equals the number of sequences $(a_1,\dots,a_n)$ with $0\le a_i\le i-1$ for all $i$ and $\sum_i a_i =k$, whereas the numerator equals the number of such sequences that additionally do not contain three consecutive entries with $a_i\ge a_{i+1}\ge a_{i+2}$.

A more probabilistic approach, which appears to suffer from similar difficulties, is to obtain a coupling between the two bumping processes driving the Mallows$(q)$ distribution in order to establish stochastic dominance between the number of occurrences of $12\dots m$ in Mallows($q$) and Mallows($q'$) distributions for $q<q'$; see~\cite{CraneDeSalvo2015} for the relevant definitions. 

We have been unable to establish the conjecture using either of these techniques, and so we leave this as an open problem.

\section{Bounds on $\rho(\sigma, q)$}\label{section:bounds}
Our techniques above allow us to approximate the growth rate for those patterns whose generating function is sufficiently tractable to analyze. 
In the absence of explicit generating functions, we establish upper and lower bounds for the growth rate using different techniques.
In this section, we generalize many of the results in~\cite{Perarnau} to bound $\rho(\sigma,q)$, for arbitrary $\sigma\in\S_m$ and $q>0$, using Suen's inequality and a version of the Lov\'asz local lemma.  
The analysis in~\cite{Perarnau} specializes to $\rho(\sigma,1)$ only.

For any permutation $\pi=\pi_1\dots\pi_n$ and any pattern $\sigma\in\S_m$, we can mark the occurrences of $\sigma$ in $\pi$ by mapping $\pi$ to a $\{0,1\}$-valued sequence $(x_1,\ldots,x_{n-m+1})$, with
\begin{equation}\label{eq:markings}x_j=\left\{\begin{array}{cc} 1,& \st(\pi_j\pi_{j+1}\dots\pi_{j+m-1})=\sigma,\\ 0,& \text{otherwise.}\end{array}\right.\end{equation}
If $\pi$ is a random permutation of $n$ drawn from the Mallows($q$) distribution, then the associated sequence of markings $(x_1,\ldots,x_{n-m+1})$ is also random and $N_n(\sigma,q):=\sum_{j=1}^{n-m+1}x_j$ is the random variable counting the number of occurrences of $\sigma$ in $\pi$.

For each $j=1,\ldots,n-m+1$, we define $A_j\equiv A_j(\sigma,q,n):=\{x_j=1\}$ to be the event that $\sigma$ occurs in position $j$ of a Mallows($q$) permutation.
We write $A_j^c=\{x_j=0\}$ to denote the complement of $A_j$, and $\mathcal{A}=\{A_1,\ldots,A_{n-m+1}\}$.
We define the {\em dependency graph} $H$ of $\mathcal{A}$ to be the graph with vertex set $V(H) = \{1, \ldots, n-m+1\}$ and edge set $E(H)$ with an edge between nodes $i$ and $j$ if and only if $1 \leq |i-j| \leq m-1$. 
By the weak dissociation property of the Mallows distribution, if any two disjoint subsets $S, T \subset \{1,\ldots, n-m+1\}$ are such that there are no edges between the nodes in $S$ and $T$ in the dependency graph $H$, then the sets of events $\{A_i\}_{i \in S}$ and $\{A_j\}_{j \in T}$ are independent.

\begin{proposition}\label{prop:generic:upper}
Fix any $m \geq 3$, $\sigma \in S_m$, and $q>0$.  Then we have 
\begin{equation}\label{generic:upper:bound} \rho(\sigma, q) \leq \left(1 - \frac{q^{\inv(\sigma)}}{[m]_q!}\right)^{1/m}. \end{equation}
In particular, the right-hand side of Equation~\eqref{generic:upper:bound} is the same for all patterns $\sigma \in S_m$ with the same number of inversions. 
\end{proposition}
\begin{proof}
We follow the same routine as in~\cite[p.\ 1002]{Perarnau}.
We define the index set $I := \{1+k\,m : 0 \leq k < n/m\}$ so that 
\[ \P(N_n(\sigma, q) = 0) = \P\left(\bigcap_{i=1}^{n-m+1} A_i^c\right) \leq \P\left(\bigcap_{i \in I} A_i^c\right) = \prod_{i \in I} \left(1-\P\left(A_i\ \middle|\ \bigcap_{j \in I, j<i} A_j^c\right) \right).\]
By the weak dissociation property, events $A_i$ and $A_j$ are independent for $i \in S = \{i\}$ and $j \in T = \{j:j\in I, j<i\}$, and by consecutive homogeneity $P(A_j)=P(A_1)=q^{\inv(\sigma)}/[m]_q!$ for every $j$.
Combining this with the fact that $|I| \leq n/m$, we have 
\[P\left(A_i\ \middle|\ \bigcap_{j\in I, j<i}A_j^c\right)=P(A_i)=P(A_1)=\frac{q^{\inv(\sigma)}}{[m]_q!}\]
and
\[ \P(N_n(\sigma,q) = 0) \leq \prod_{i \in I} \left(1 - \frac{q^{\inv(\sigma)}}{[m]_q!}\right) \leq \left(1 - \frac{q^{\inv(\sigma)}}{[m]_q!}\right)^{n/m}.  \qedhere\]
\end{proof}

Suen's inequality provides a notable improvement to the above upper bound in many instances by taking into account interactions between dependent events.  
Let $E(H)$ denote the set of edges in the previously defined dependency graph $H$. 
We define
\[ \Delta := \sum_{\{i,j\}\in E(H)} \P(A_i \cap A_j), \] 
and 
\[ \delta := \max_{1 \leq i \leq n-m+1} \sum_{j : \{i,j\} \in E(H)} \P(A_j). \]

The following is an improvement to Suen's original inequality due to Janson~\cite[Theorem~2]{SuenByJanson}.
\begin{theorem}[Suen's inequality]\label{Suen:inequality}
Let $\{x_i\}_{i =1,\ldots,n-m+1}$ be a finite family of indicator random variables with dependency graph $H$ and define $N = \sum_{i=1}^{n-m+1} x_i$. 
Let $\mu_i := \e\, x_i$, $i=1,\ldots, n-m+1$.
Then 
\[ \P(N = 0) \leq \exp\left( -\sum_{i=1}^{n-m+1}\mu_i + \Delta  e^{2\delta}\right). \]
\end{theorem}

We are thus able to produce an improved upper bound for $\rho(\sigma,q)$ using the overlap set of $\sigma$, defined below.

\begin{definition}[Overlap set]\label{overlap}
For every $\sigma \in S_m$ and $1 \leq s \leq m-1$, we define the \emph{overlap set of $\sigma$ of size~$s$} as the set of permutations $\tau \in S_{2m-s}$ such that $\st(\tau_1\cdots\tau_m) = \st(\tau_{m-s+1}\cdots\tau_{2m-s}) = \sigma$, denoted by $ \Ov_s(\sigma)$.  We similarly define $ \Ov(\sigma) := \bigcup_{s}  \Ov_s(\sigma)$ as the \emph{overlap set of $\sigma$}. 
\end{definition}

\begin{remark}
In the language of Section \ref{section:GJ}, each $\tau\in\Ov_s(\sigma)$ corresponds to a $2$-cluster $(\tau;1,m-s+1)$ of $\sigma$, whose only marked occurrences of $\sigma$ overlap in exactly $s$ positions and occur at the beginning and end of $\tau$.
\end{remark}  

\begin{proposition}\label{prop:improved}
Fix any $m \geq 3$, $\sigma \in S_m$ and $q>0$, and define 
\begin{equation}\label{Ts:def} T(s,\sigma,q) := \sum_{\tau \in  {\small\Ov}_s(\sigma)} \frac{q^{\inv(\tau)}}{[2m-s]_q!}.
\end{equation}
Then
\begin{equation}\label{improved:upper:bound} \rho(\sigma, q) \leq \exp\left(-\frac{q^{\inv(\sigma)}}{[m]_q!}  + \exp\left( 4(m-1) \frac{q^{\inv(\sigma)}}{[m]_q!}\right)\ \sum_{s=1}^{m-1}T(s,\sigma,q) \right). \end{equation}
\end{proposition}
\begin{proof}
For fixed $m \geq 3$, $\sigma\in S_m$, and $q>0$, with $n \geq 2m-1$, let $x_1,\dots,x_{n-m+1}$ be the markings defined in \eqref{eq:markings} for a Mallows($q)$ permutation and let $A_j=\{x_j=1\}$ be as defined above.
We have
\[ \P(A_j) = \frac{q^{\inv(\sigma)}}{[m]_q!}, \qquad \P(A_j \cap A_i) = \e x_i x_j, \]
whence 
\begin{equation}\label{eq:sumexx}
\Delta = \sum_{s=1}^{m-1} \sum_{i=1}^{n-2m+1+s} \e \, x_i \, x_{i+m-s} = \sum_{s=1}^{m-1} (n-2m+1+s) T(s,\sigma,q), 
\end{equation}
and also 
\[ \delta =  2(m-1) \frac{q^{\inv(\sigma)}}{[m]_q!}. \]

We then apply Theorem~\ref{Suen:inequality} by noting that $(n-m+1)/n \leq 1$ and $(n-m+1)/n \to 1$  as $n\to\infty$,  and also $\frac{\Delta}{n} \leq \sum_{s=1}^{m-1} T(s,\sigma,q)$ and 
\[ \frac{\Delta}{n} \to \sum_{s=1}^{m-1} T(s,\sigma,q). \qedhere\]
\end{proof}

The lower bound is a bit more delicate, and uses a version of the Lov\'asz local lemma~\cite[Lemma~2.1]{PeresSchlag} particularly suited to this type of problem.  

\begin{theorem}[Lov\'asz local lemma \cite{PeresSchlag}]\label{LLL}
Let $\{A_j\}_{j=1}^r$ be events in some probability space, and let $\{z_j\}_{j=1}^r$ be a sequence of numbers in $(0,1)$.  For each $i \leq r$, suppose there is an integer $m(i)$ satisfying $0 \leq m(i) \leq i$ such that 
\begin{equation}\label{lower:condition} 
\P\left(A_i \middle| \bigcap_{j < m(i)} A_j^c\right) \leq z_i  \prod_{j=m(i)}^{i-1}(1-z_j). 
\end{equation}
Then, for any $t \in \{1,2,\ldots, r\}$, we have 
$$ \P\left(\bigcap_{i = 1}^t A_j^c\right) \geq \prod_{\ell = 1}^t (1-z_\ell). $$
\end{theorem}

\begin{proposition}\label{lower:bound}
Fix any $m \geq 3$, $\sigma \in S_m$, $q>0$, and assume that 
\[ f(q,m) := \frac{1}{2}\left(1 - \frac{q^{\inv(\sigma)}}{[m]_q!} - \sqrt{1-(4m-2)\frac{q^{\inv(\sigma)}}{[m]_q!}+\frac{q^{2\inv(\sigma)}}{[m]_q!^2}}\right) \]
is such that $\frac{q^{\inv(\sigma)}}{[m]_q!} e^{f(q,m)} \in (0,1)$. 
Then 
\begin{equation}\label{analytical:lower:bound} \rho(\sigma,q) \geq 1 - \frac{q^{\inv(\sigma)}}{[m]_q!} \exp\left(\frac{1}{2}\left(1 - \frac{q^{\inv(\sigma)}}{[m]_q!} - \sqrt{1-(4m-2)\frac{q^{\inv(\sigma)}}{[m]_q!}+\frac{q^{2\inv(\sigma)}}{[m]_q!^2}}\right)\right). \end{equation}
\end{proposition}

\begin{proof}
In order to apply Theorem~\ref{LLL} in our setting, we first choose $m(i) = i-m+1$, so that the left-hand side~of Equation~\eqref{lower:condition} is 
\[ \P\left(A_i \middle| \bigcap_{j < i-m+1} A_j^c\right) = \P(A_i) =  \frac{q^{\inv(\sigma)}}{[m]_q!}, \]
 as in the proof of Proposition~\ref{prop:generic:upper}. 
Then, by symmetry of the events in index $i$, it suffices to find a single $z \in (0,1)$ such that 
\begin{equation}\label{q:bound} \frac{q^{\inv(\sigma)}}{[m]_q!} \leq z (1-z)^{m-1}. \end{equation}

Following ~\cite[p.~12]{PerarnauArXiv}, we consider $z$ of the form $z =  \frac{q^{\inv(\sigma)}}{[m]_q!} e^{f(q,m)}$, for some positive function $f(q,m)$.  
Then~\eqref{q:bound} is equivalent to 
\begin{align}
\nonumber 1 \leq &\ e^{f(q,m)}\left(1- \frac{q^{\inv(\sigma)}}{[m]_q!}e^{f(q,m)} \right)^{m-1}, \\
\nonumber e^{-\frac{f(q,m)}{m-1}} \leq & \ 1 -  \frac{q^{\inv(\sigma)}}{[m]_q!}  e^{f(q,m)}, \\
\label{before:inequality}-\frac{f(q,m)}{m-1} \leq & \ \log\left(1 -  \frac{q^{\inv(\sigma)}}{[m]_q!}  e^{f(q,m)}\right).
\end{align}
At this point we utilize the inequality $\log(1-x) \geq -\frac{x}{1-x}$, valid for all $x<1$, and so we must assume that $\frac{q^{\inv(\sigma)}}{[m]_q!} e^{f(q,m)} < 1$.  
Equation~\eqref{before:inequality} will then hold assuming the following holds: 
\begin{equation}\label{after:inequality} f(q,m) \geq (m-1)\ \frac{\frac{q^{\inv(\sigma)}}{[m]_q!}  e^{f(q,m)}}{1-\frac{q^{\inv(\sigma)}}{[m]_q!}  e^{f(q,m)}}. \end{equation}
Rearranging, this is equivalent to 
\[ f(q,m) \left(e^{-f(q,m)} - \frac{q^{\inv(\sigma)}}{[m]_q!}\right) \geq (m-1) \frac{q^{\inv(\sigma)}}{[m]_q!}. \]
Note that $e^{-f(q,m)} > 1-f(q,m)$ for $f(q,m) < 1$, and so the above inequality will be satisfied if the following holds:
\begin{equation}\label{quadratic:inequality} f(q,m) \left( (1-f(q,m)) - \frac{q^{\inv(\sigma)}}{[m]_q!} \right) \geq (m-1)\frac{q^{\inv(\sigma)}}{[m]_q!}. \end{equation}
Since this is quadratic in $f(q,m)$, by solving using equality in place of inequality in~\eqref{quadratic:inequality}, we can obtain the region in which the inequality is satisfied. 
The two solutions are 
\[ f(q,m) = \frac{1}{2}\left(1 - \frac{q^{\inv(\sigma)}}{[m]_q!} \pm \sqrt{1-(4m-2)\frac{q^{\inv(\sigma)}}{[m]_q!}+\frac{q^{2\inv(\sigma)}}{[m]_q!^2}}\right). \]
Since the coefficient of $f(q,m)^2$ is negative in~\eqref{quadratic:inequality}, we know that inequality~\eqref{quadratic:inequality} is satisfied for all values in between the solutions. 
Since both solutions are positive, we take the smaller solution to obtain the optimal value 
\[ z = \frac{q^{\inv(\sigma)}}{[m]_q!} \exp\left(\frac{1}{2}\left(1 - \frac{q^{\inv(\sigma)}}{[m]_q!} - \sqrt{1-(4m-2)\frac{q^{\inv(\sigma)}}{[m]_q!}+\frac{q^{2\inv(\sigma)}}{[m]_q!^2}}\right)\right). \]
This completes the proof.
\end{proof}

\begin{remark}
There is a minor technical issue in~\cite[Proof~of~Theorem~5]{Perarnau}, which does not affect the validity of~\cite[Theorem~5]{Perarnau} as stated, since the resulting bounds are stated in terms of asymptotic estimates as $m$ tends to infinity. 
However, it does impact the final form and proof of our Proposition~\ref{lower:bound}, since we are adapting the argument and using it in its preasymptotic form. 

Specifically, the argument used to obtain $z = \frac{1}{m!} e^{(m-1)/m!}$, see~\cite[p.\ 12]{PerarnauArXiv}, is flawed.
It seems that instead of using the inequality $1 - x  > e^{-\frac{x}{1-x}}$, as we have done following Equation~\eqref{before:inequality}, the reverse inequality $1-x < e^{-x}$ was used. 
We have verified numerically that for $3 \leq m \leq 500$,
inequality~\eqref{q:bound} fails for $q=1$ and this value of $z$.
\end{remark}

Rather than rely on the analytic lower bound obtained in Proposition \ref{lower:bound}, we can instead compute the optimal value of $z$ in~\eqref{q:bound} numerically for each $\sigma$ and $q$ to obtain a better estimate, as we have done in the examples of the next section.

\subsection{Examples}

We next compute the explicit bound implied by Suen's inequality in Proposition~\ref{prop:improved} for some specific patterns. 
See Figure~\ref{fig:bounds} for plots of this improved bound, together with the bounds implied by Propositions~\ref{prop:generic:upper} and~\ref{lower:bound}.

\begin{figure}[t!]
	\centering
	\subfloat[$\sigma = 1234$]{\includegraphics[scale=0.18]{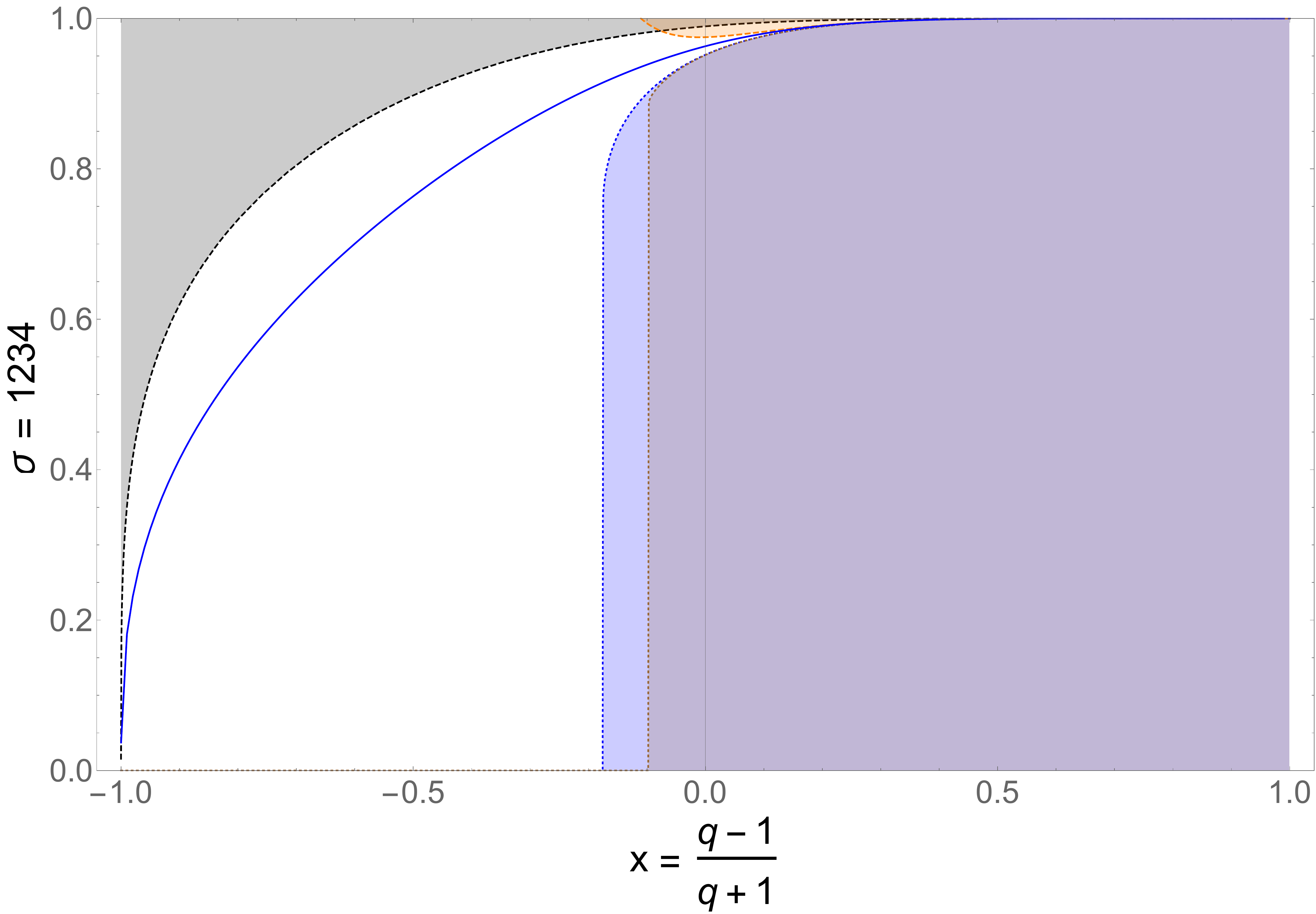}\label{fig:1234:bounds}} \quad 
	\subfloat[$\sigma = 1243$]{\includegraphics[scale=0.18]{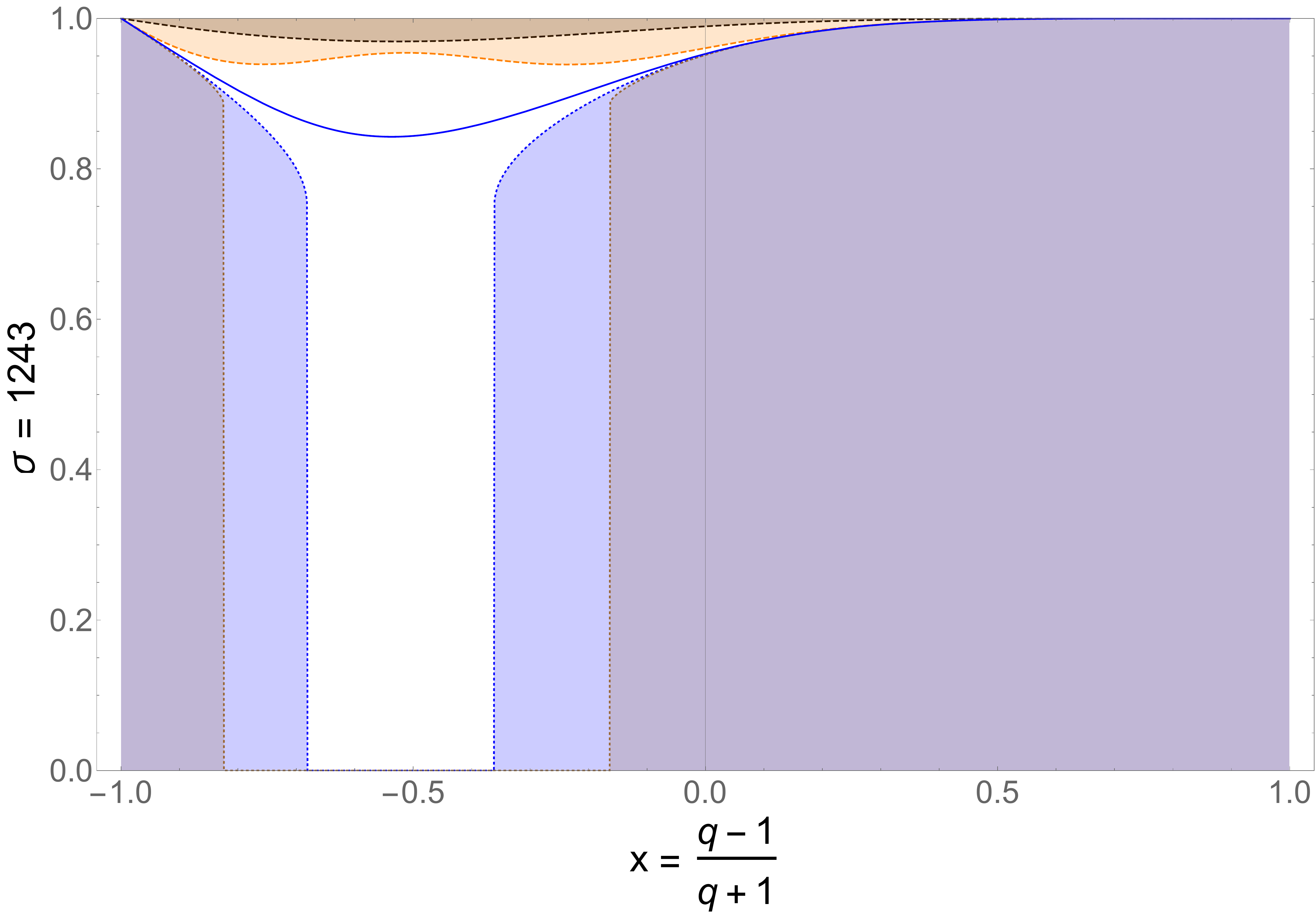}\label{fig:1243:bounds}} \\
	\subfloat[$\sigma = 1342$]{\includegraphics[scale=0.18]{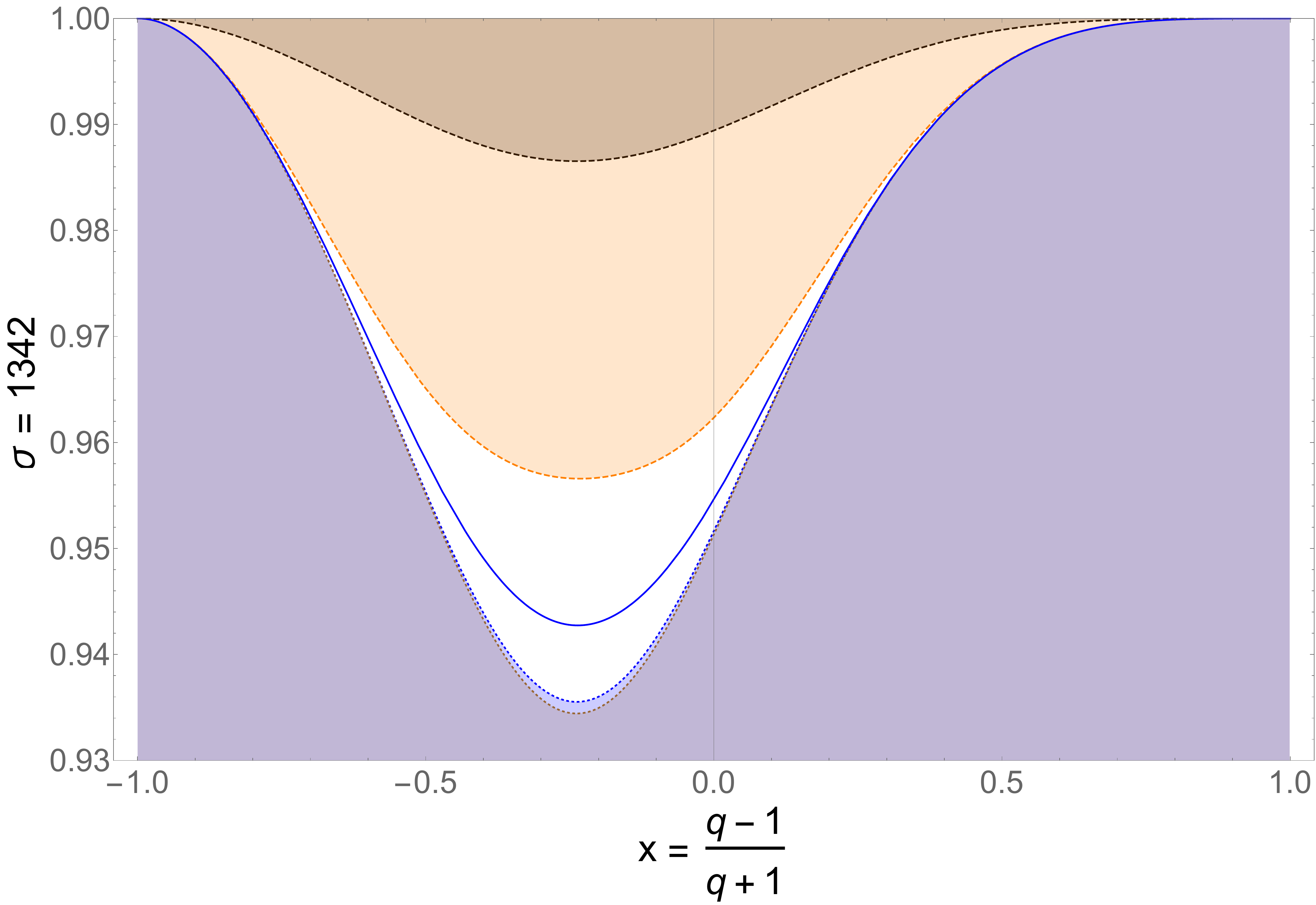}\label{fig:1342:bounds}} \quad 
	\subfloat[$\sigma = 1432$]{\includegraphics[scale=0.18]{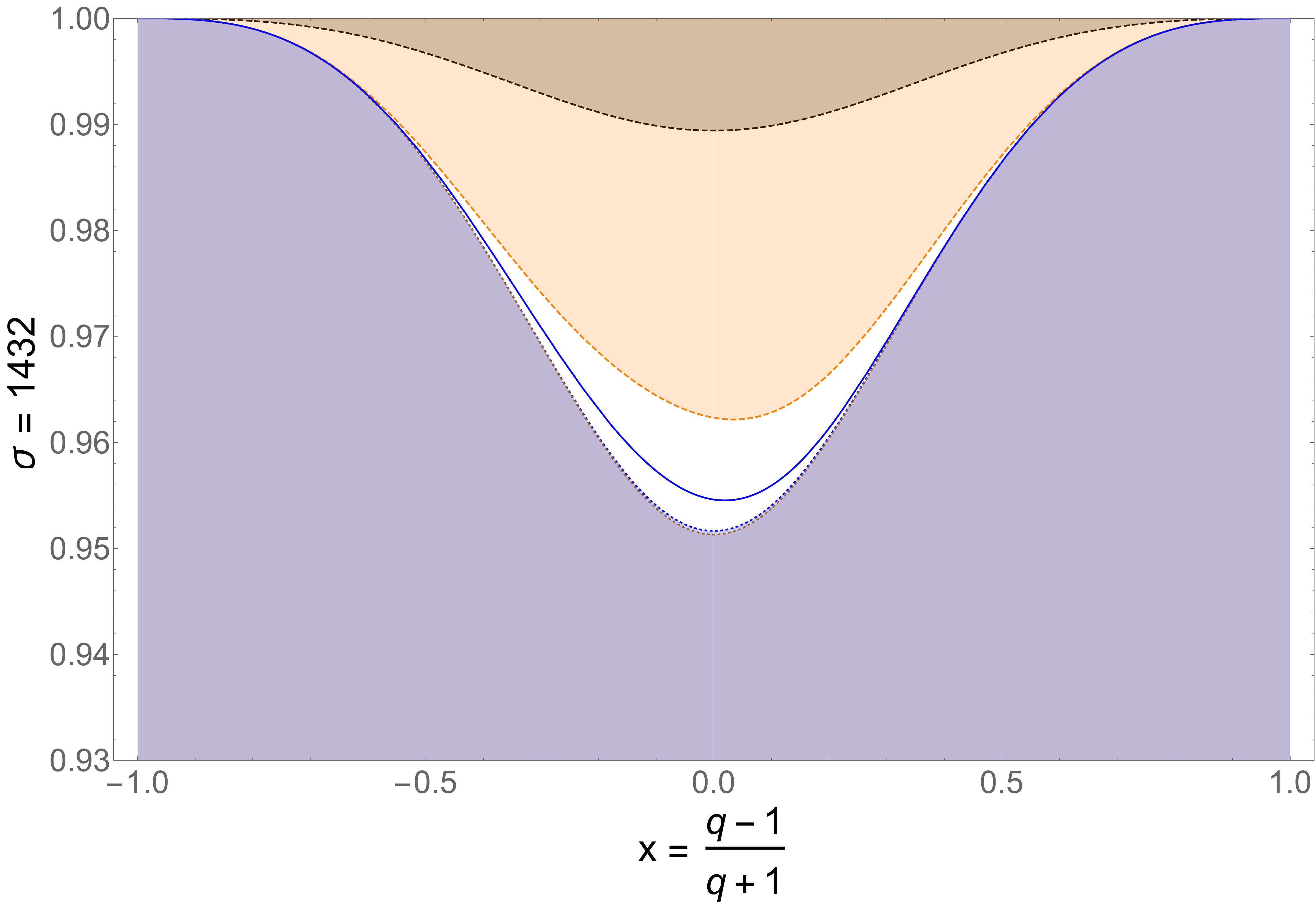}\label{fig:1432:bounds}} \\
	\caption{In each of the graphs, the solid blue curve is a plot of $\rho(\sigma,q)$ with $q = \frac{1+x}{1-x}$, found using the techniques from Sections~\ref{section:monotone} and~\ref{section:132}. 
Looking at $x=0$, the top-most dotted curve with top shading is a universal upper bound valid for all patterns of length 4 with exactly $\inv(\sigma)$ inversions.  The dotted curve underneath with top shading is an upper bound specifically for the indicated pattern $\sigma$ using Suen's inequality.  The bottom-most dotted curve is the lower bound implied by the analytical solution~\eqref{analytical:lower:bound}.  The dotted curve just above the bottom-most dotted line is a universal lower bound valid for all patterns of length 4 with exactly $\inv(\sigma)$ inversions, computed by solving inequality~\eqref{q:bound} numerically.}
\label{fig:bounds}
\end{figure}

\subsubsection{Pattern 1432}\label{sect:1432}
First, we consider the pattern $\sigma = 1432$ of length $4$ with $3$ inversions previously encountered in Section~\ref{section:generalizations}. 
The upper and lower bounds given in Propositions~\ref{prop:generic:upper} and~\ref{lower:bound} apply to all other patterns of length~$4$ with~$3$ inversions, namely, $2341$, $2413$, $3142$, $3214$, $4123$. 
The upper bound from Proposition~\ref{prop:improved}, however, depends on the set of overlaps  $\Ov(1432)$, which we have presented in Table~\ref{Ov:1432} along with the number of inversions for each pattern.
The quantity $\sum_{s=1}^{m-1} T(s,1432,q)$, with $T(s,\sigma,q)$ defined in Equation~\eqref{Ts:def}, is then 
\[ \sum_{s=1}^{m-1} T(s,1432,q) = \frac{q^{12}+q^{11}+2 q^{10}+2 q^9+2 q^8+q^7+q^6}{[7]_q!} =: T_{1432}(q),
\]
allowing us to compute the upper bound for all $q>0$:
\[ \rho(1432,q) \leq \exp \left(-\frac{q^3}{[4]_q!} + T_{1432}(q)\ e^{\frac{12 q^3}{[4]_q!}}\right). \]

\begin{table}[htb]\[ 
\begin{array}{c@{\qquad}c@{\qquad}c@{\qquad}c}
 
\begin{array}{|c|c|} \hline
\tau\in\Ov(1432) & \inv(\tau) \\ \hline
 1  4  3  2  7  6  5 & 6 \\
 1  5  3  2  7  6  4 & 7 \\
 1  5  4  2  7  6  3 & 8 \\
 1  6  3  2  7  5  4 & 8 \\
 1  6  4  2  7  5  3 & 9 \\
 1  6  5  2  7  4  3 & 10 \\
 1  7  3  2  6  5  4 & 9 \\
 1  7  4  2  6  5  3 & 10 \\
 1  7  5  2  6  4  3 & 11 \\
 1  7  6  2  5  4  3 & 12 \\ \hline
 \end{array}
& 
\begin{array}{|c|c|} \hline 
\tau\in\Ov(1342) & \inv(\tau) \\ \hline
 1  3  4  2  6  7  5 & 4\\ 
 1  3  5  2  6  7  4 & 5\\
 1  3  6  2  5  7  4 & 6\\
 1  3  7  2  5  6  4 & 7\\
 1  4  5  2  6  7  3 & 6\\
 1  4  6  2  5  7  3 & 7\\
 1  4  7  2  5  6  3 & 8\\
 1  5  6  2  4  7  3 & 8\\
 1  5  7  2  4  6  3 & 9\\
 1  6  7  2  4  5  3 & 10\\ \hline
\end{array}
&
\begin{array}{|c|c|} \hline 
\tau\in\Ov(2341) &  \inv(\tau) \\ \hline
 3  4  5  2  6  7  1 & 9 \\
 3  4  6  2  5  7  1 & 10\\
 3  4  7  2  5  6  1 & 11\\
 3  5  6  2  4  7  1 & 11\\
 3  5  7  2  4  6  1 & 12\\
 3  6  7  2  4  5  1 & 13\\
 4  5  6  2  3  7  1 & 12\\
 4  5  7  2  3  6  1 & 13\\
 4  6  7  2  3  5  1 & 14\\
 5  6  7  2  3  4  1 & 15\\ \hline
\end{array} 
\end{array} \]

\[
\begin{array}{c@{\qquad}c}
\begin{array}{|c|c|} \hline 
\tau\in\Ov(2413) &  \inv(\tau) \\ \hline
362514 & 9 \\
462513 & 10 \\
2514736 & 7 \\
2614735 & 8 \\
2714635 & 9 \\
3514726 & 8 \\
3524716 & 9 \\
3614725 & 9 \\
3624715 & 10 \\
3714625 & 10 \\
3724615 & 11 \\ \hline
\end{array} 
&
\begin{array}{|c|c|} \hline 
\tau\in\Ov(1243) & \inv(\tau) \\ \hline
1243576 & 2\\ 
1253476 & 3\\
1263475 & 4\\
1273465 & 5\\ \hline
\end{array}
 \end{array}
 \]
 \caption{Permutations in $\Ov(1432)$, $\Ov(1342)$, $\Ov(2341)$, $\Ov(2413)$, and $\Ov(1243)$, with their corresponding number of inversions.}
 \label{Ov:1432}\label{Ov:1342}\label{Ov:1243}\label{Ov:2341}\label{Ov:2413}
 \end{table}

\subsubsection{Pattern 1342}
We next consider 1342, which has length 4 and 2 inversions.  The other patterns of length 4 with 2 inversions are $1423$, $2143$, $2314$, and $3124$. 
As in the previous example, we compute $\Ov(1342)$ along with the number of inversions for each element in the set, summarized in Table~\ref{Ov:1342}.
We have 
\[ \sum_{s=1}^{m-1} T(s,1342,q) = \frac{q^{10}+q^9+2 q^8+2 q^7+2 q^6+q^5+q^4}{[7]_q!} =: T_{1342}(q),\]
and so 
\[ \rho(1432,q) \leq \exp \left(-\frac{q^2}{[4]_q!}+T_{1342}(q)\ e^{\frac{12 q^2}{[4]_q!}}\right).
\]

\subsubsection{Pattern 1243}
The pattern 1243 demonstrates that the condition in Proposition~\ref{lower:bound} is essential, and that sometimes it is best to evaluate the relevant quantities numerically, as the interplay between $q$ and $m$ can be quite complicated when attempting to solve inequality~\eqref{q:bound}, which does not always have a solution. 
In this case, pattern $1243$ has length $4$ with exactly $1$ inversion, which is also the case for patterns $1342$ and $2134$.  
The overlap set is much smaller than in previous examples, consisting of just the four elements given in Table~\ref{Ov:1243}. 
We have 
\[ \sum_{s=1}^{m-1} T(s,1243,q) = \frac{q^5+q^4+q^3+q^2}{[7]_q!} =: T_{1243}(q), \]
and so 
\[ \rho(1243,q) \leq \exp \left(-\frac{q}{[4]_q!} + T_{1243}(q)\ e^{\frac{12 q}{[4]_q!}}\right). \]
In particular, this is the first case covered so far where there is a range of values of $q$ for which inequality~\eqref{q:bound} has no solution.

\subsubsection{Other patterns}

Figure~\ref{fig:2413:bounds} plots the bounds for the pattern $2413$, which has 3 inversions and, therefore, shares the universal upper and lower bounds of Figure~\ref{fig:1432:bounds}, but which differs from pattern 1432 with its own custom improved upper bound.  

Figure \ref{fig:5:bounds} plots the various universal bounds for patterns of length 5 with inversions $1, 2, \ldots, 10$.
The symmetry in these bounds for $(q,\inv(\sigma))\leftrightarrow(1/q,\binom{m}{2}-\inv(\sigma))$ reflects the symmetry in $\rho(\sigma,q)$.

Figure~\ref{fig:10:5:bounds} showcases the versatility and computational efficiency of the generic bounds by plotting the bounds for any pattern of length 10 with exactly 5 inversions. 

Finally, we revisit the monotonically increasing pattern $1234$, which is the only pattern of length 4 with 0 inversions.  
In this case, there is also a large range of values of $q$ for which inequality~\eqref{q:bound} has no solution. 
In addition, Suen's inequality rises above 1 for $q$ small enough, whereas the universal bound is more useful. 
Figure~\ref{fig:1234:bounds} plots these bounds, and Figure~\ref{fig:12m:bounds} plots the bounds for the monotonically increasing patterns $12\ldots m$ for $m=6, 8, 10$.

\begin{figure}[htb]
\centering
\includegraphics[scale=0.2]{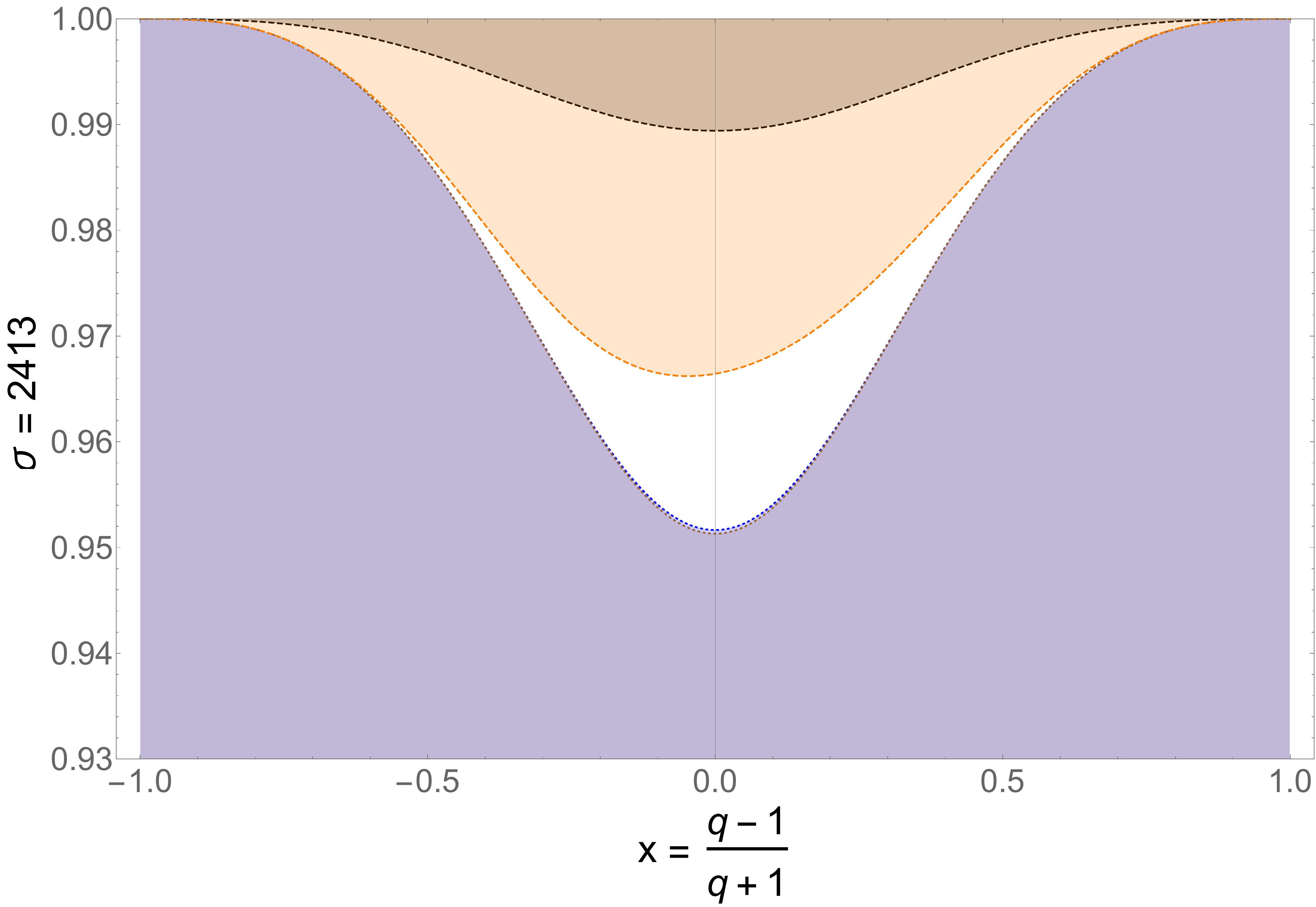}
\caption{Plot of the bounds for $\rho(2413,q)$, with $q = \frac{1+x}{1-x}$, as described in the caption of Figure~\ref{fig:bounds}. 
}
\label{fig:2413:bounds}
\end{figure}

\begin{figure*}[t!]
	\centering
	\subfloat[$\inv(\sigma)=0$]{\includegraphics[scale=0.18]{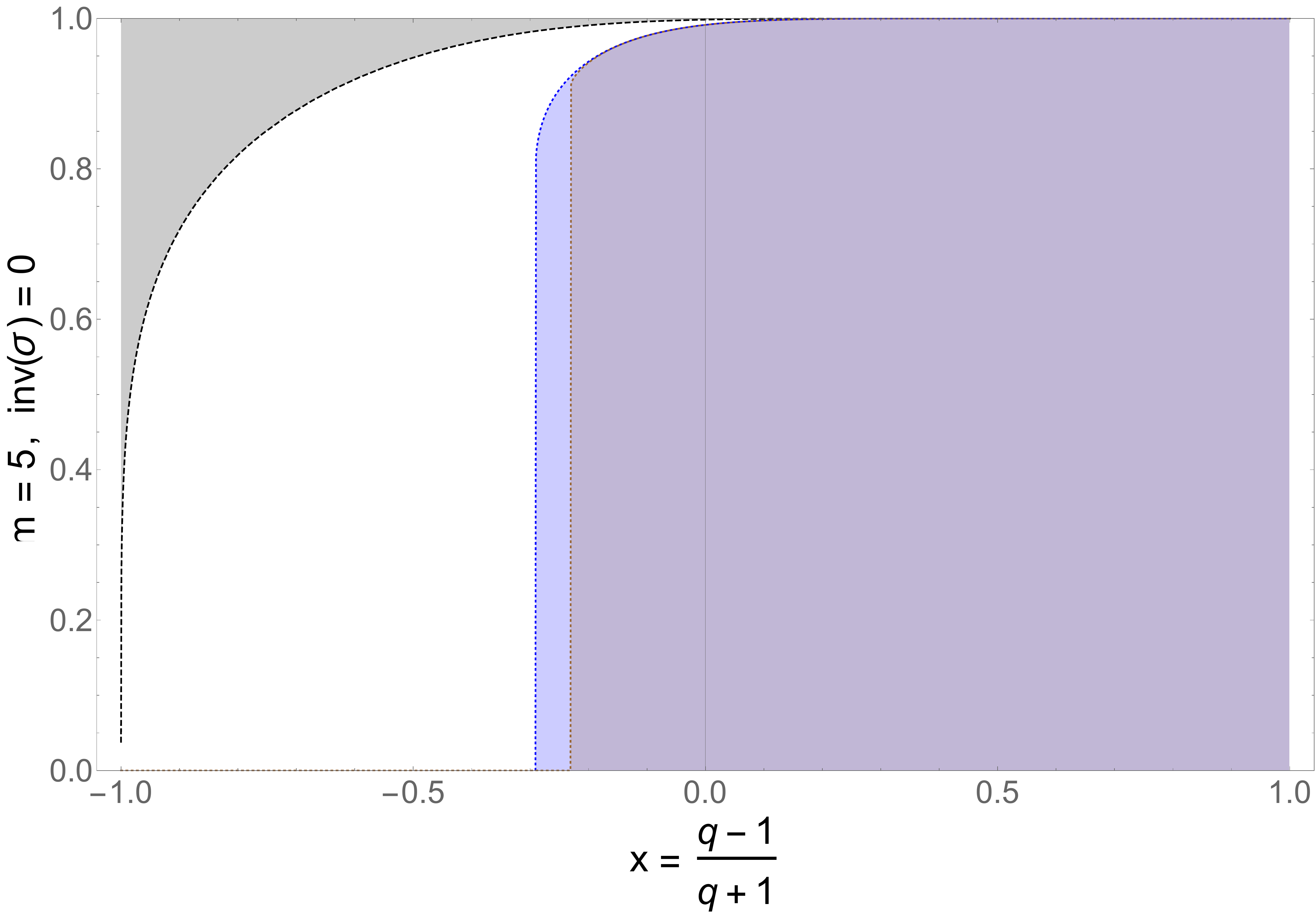}} \quad 
	\subfloat[$\inv(\sigma)=1$]{\includegraphics[scale=0.18]{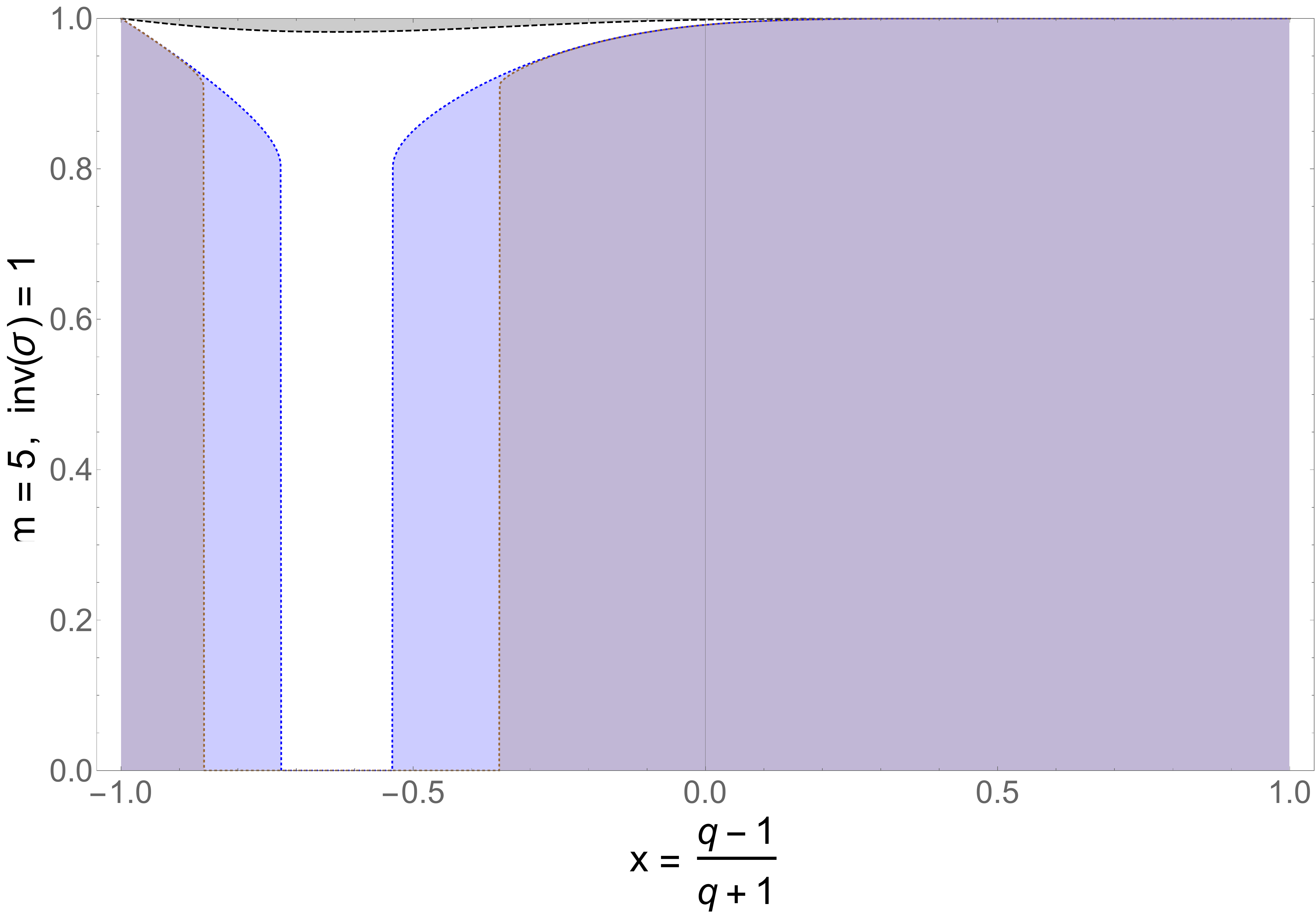}} \\
	\subfloat[$\inv(\sigma)=2$]{\includegraphics[scale=0.18]{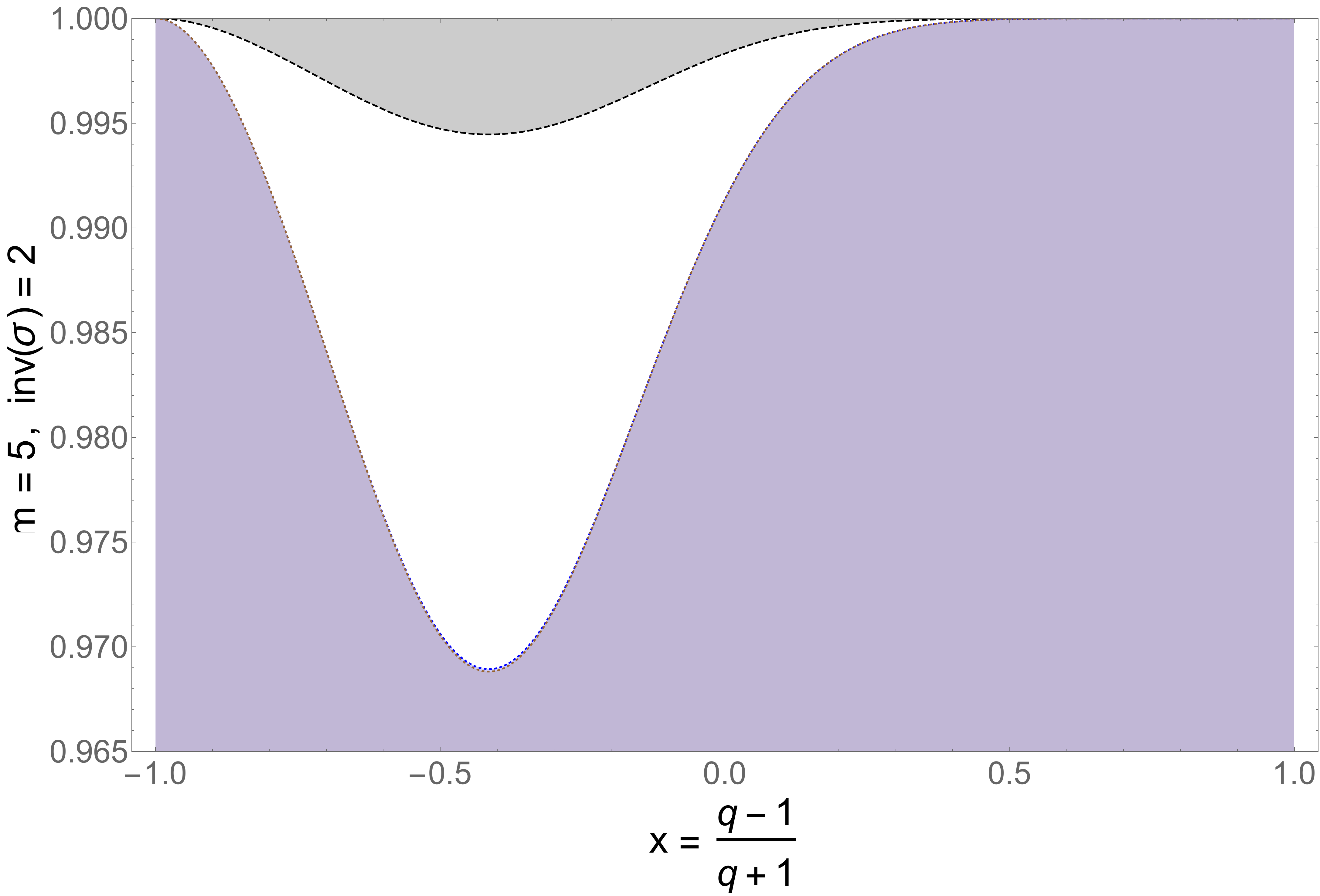}} \quad
	\subfloat[$\inv(\sigma)=3$]{\includegraphics[scale=0.18]{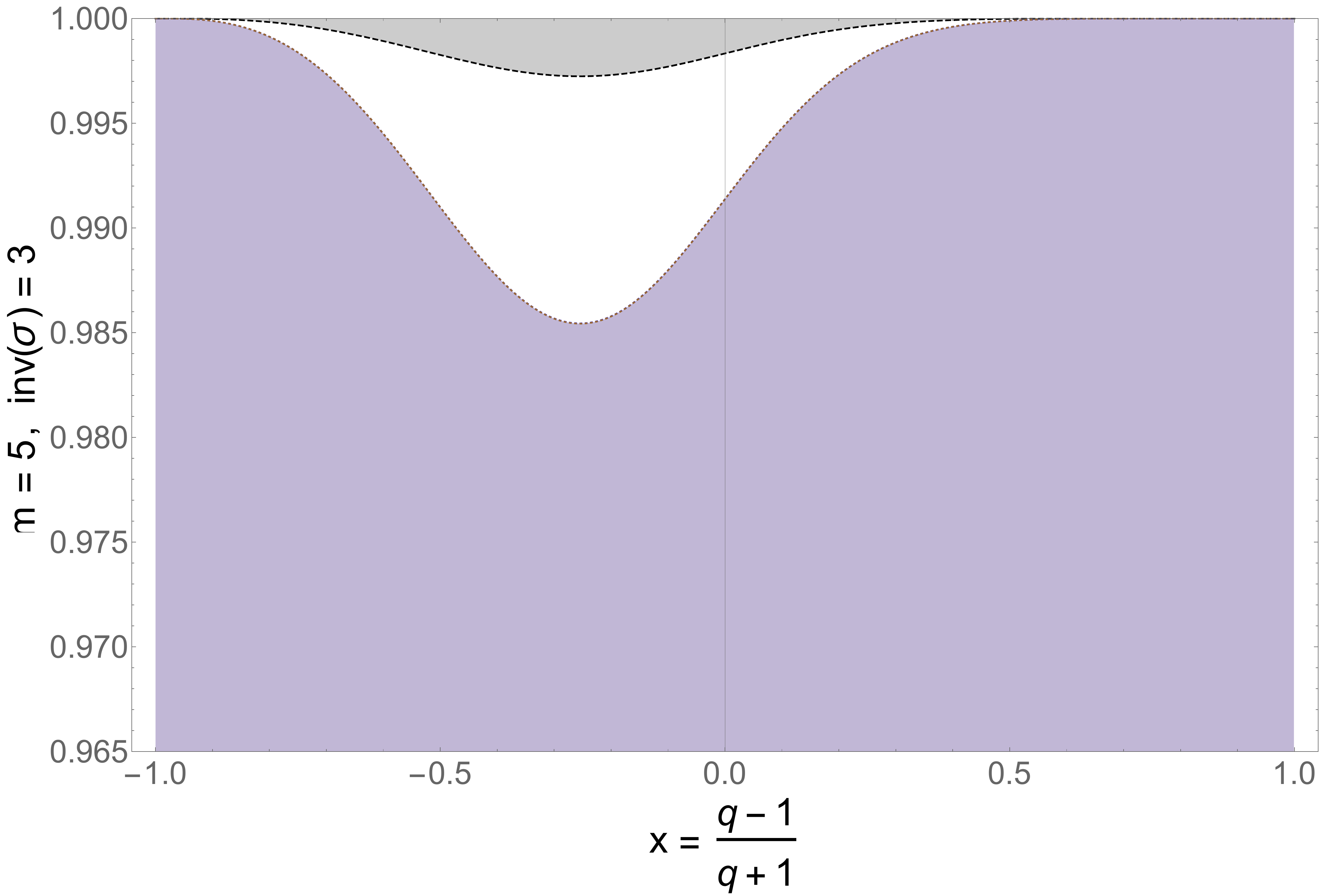}} \\
	\subfloat[$\inv(\sigma)=4$]{\includegraphics[scale=0.18]{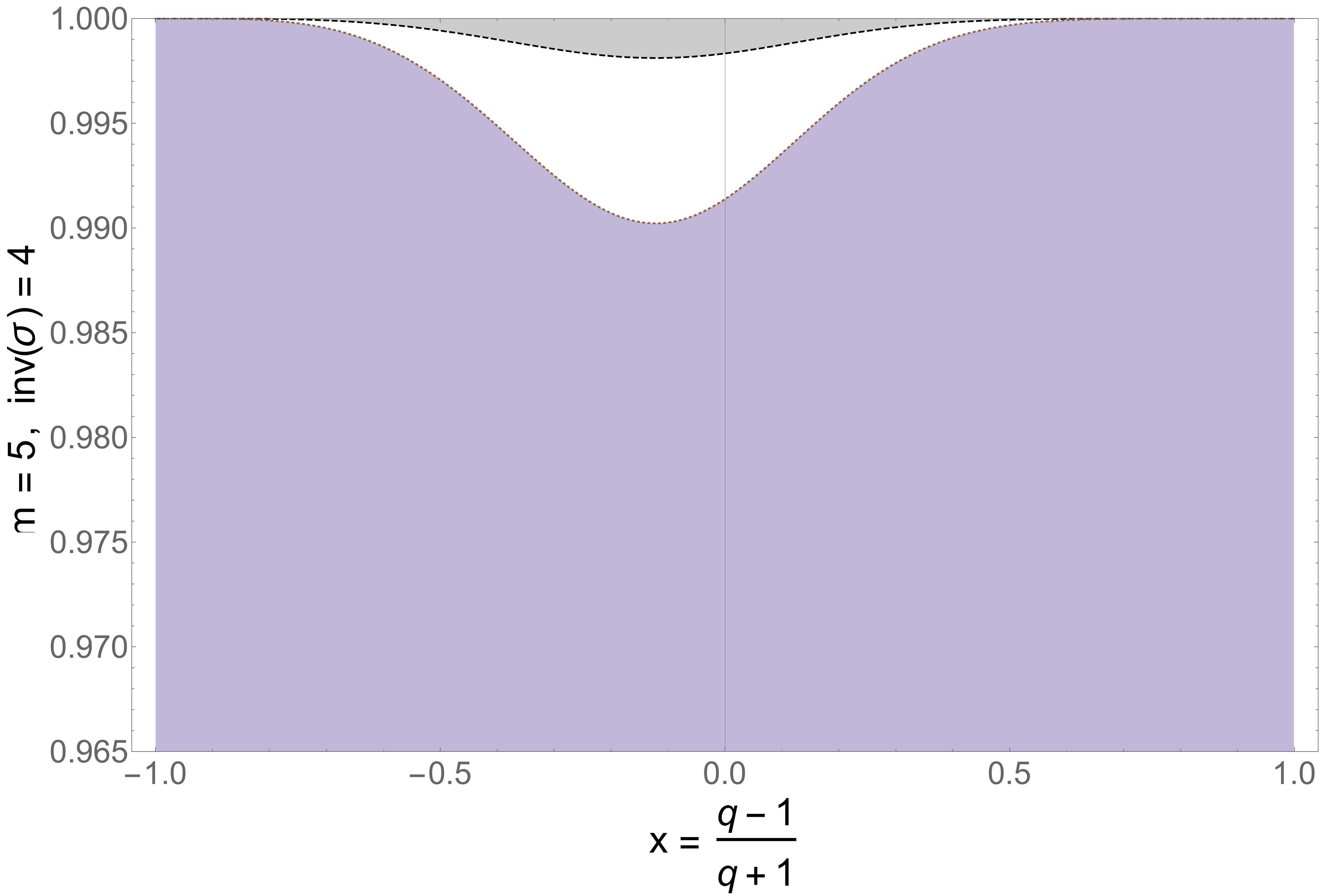}} \quad
	\subfloat[$\inv(\sigma)=5$]{\includegraphics[scale=0.18]{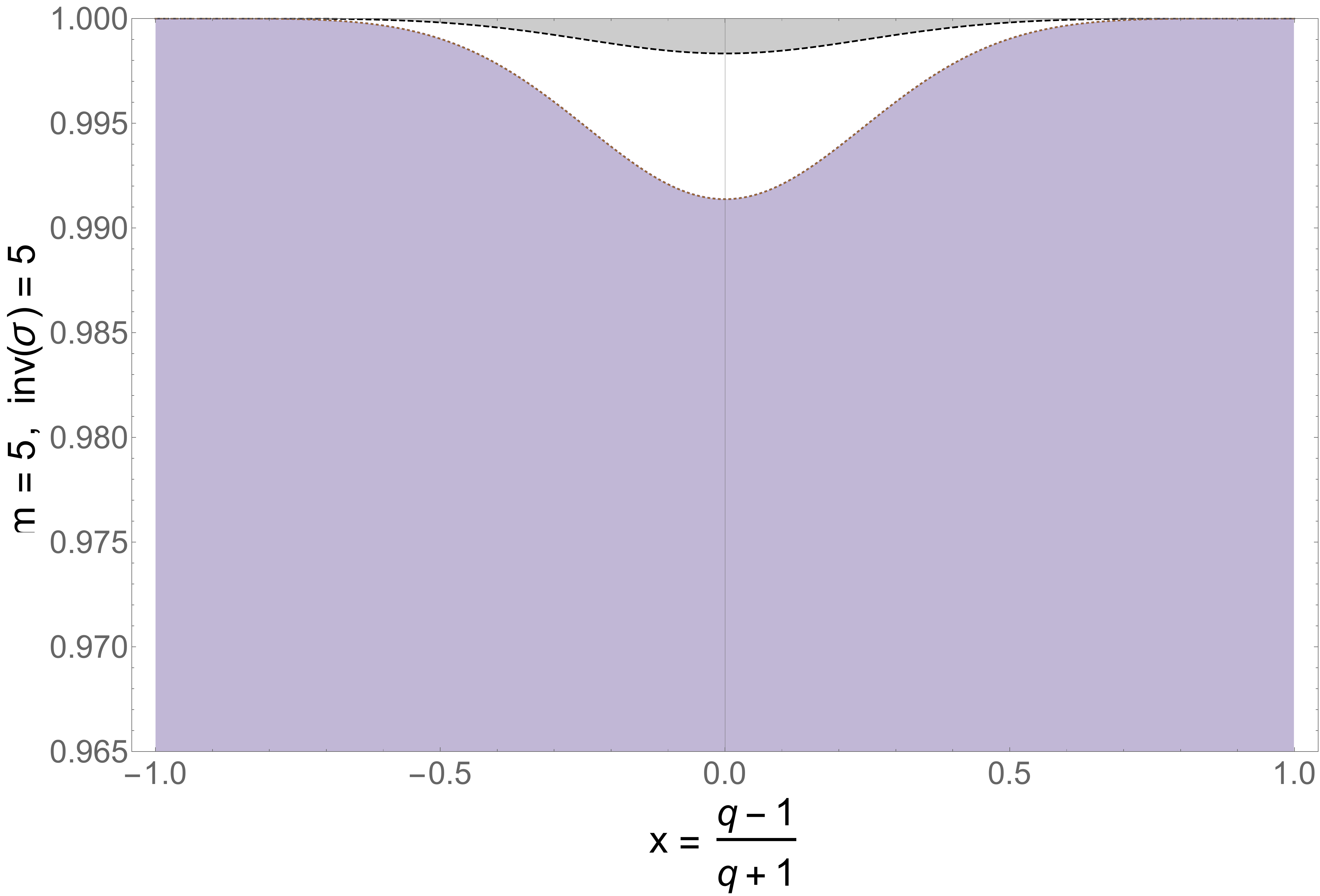}} \\
\caption{Plots of the bounds for $\rho(\sigma,q)$ for patterns~$\sigma$ of length~$m=5$ with $0,1,\ldots,5$ inversions, respectively. 
The plots for patterns with $10, 9, \ldots, 6$ inversions are obtained by reflecting the first five by the line $x=0$.
}
\label{fig:5:bounds}
\end{figure*}

\begin{figure}
\centering\includegraphics[height=6cm]{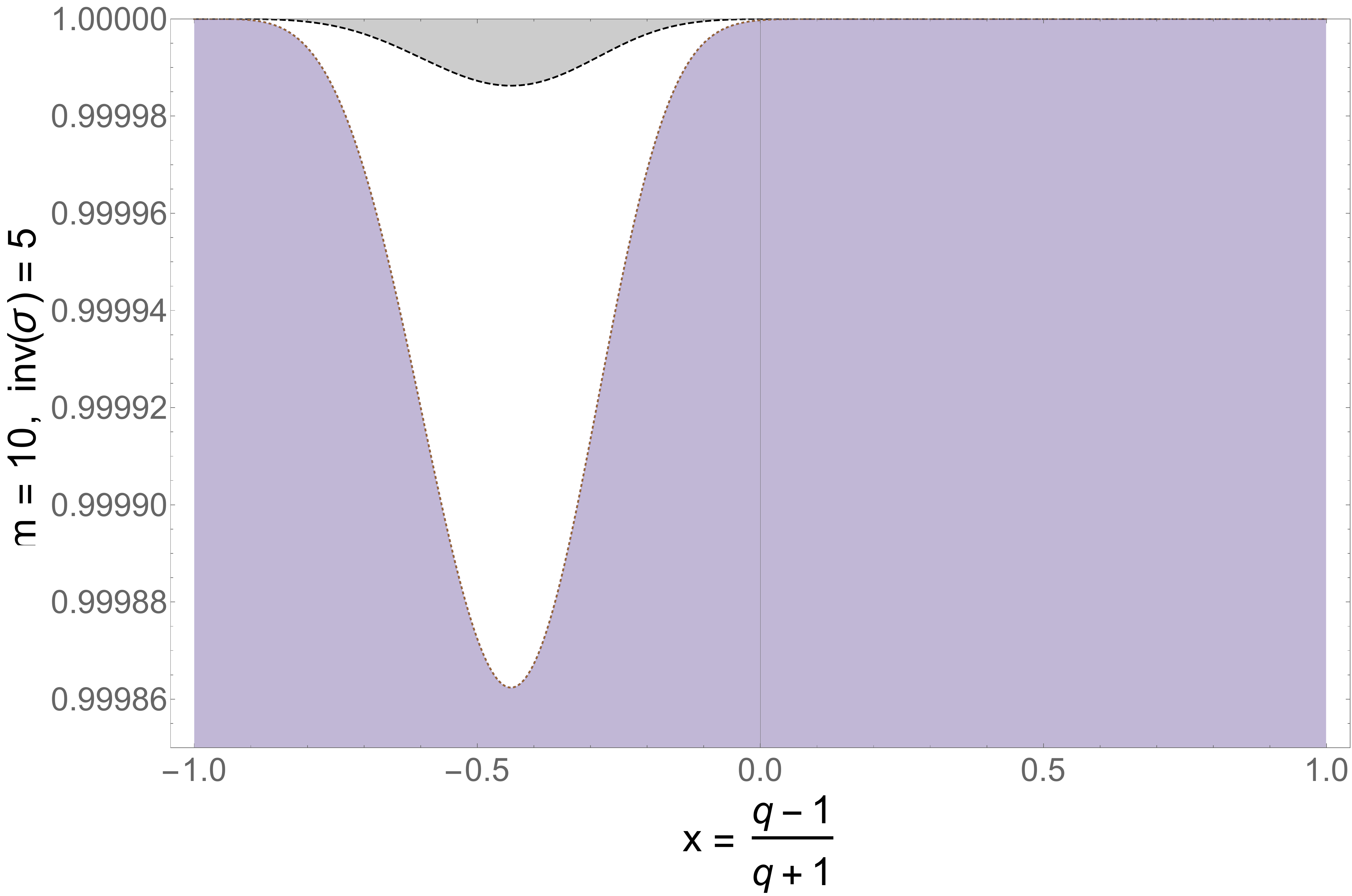}
\caption{Plot of the bounds for $\rho(\sigma, q)$ for any $\sigma \in S_{10}$ with exactly 5 inversions.}
\label{fig:10:5:bounds}
\end{figure}

\begin{figure}[htb]
\centering
\subfloat[$\sigma=12\ldots6$]{\includegraphics[scale=.12]{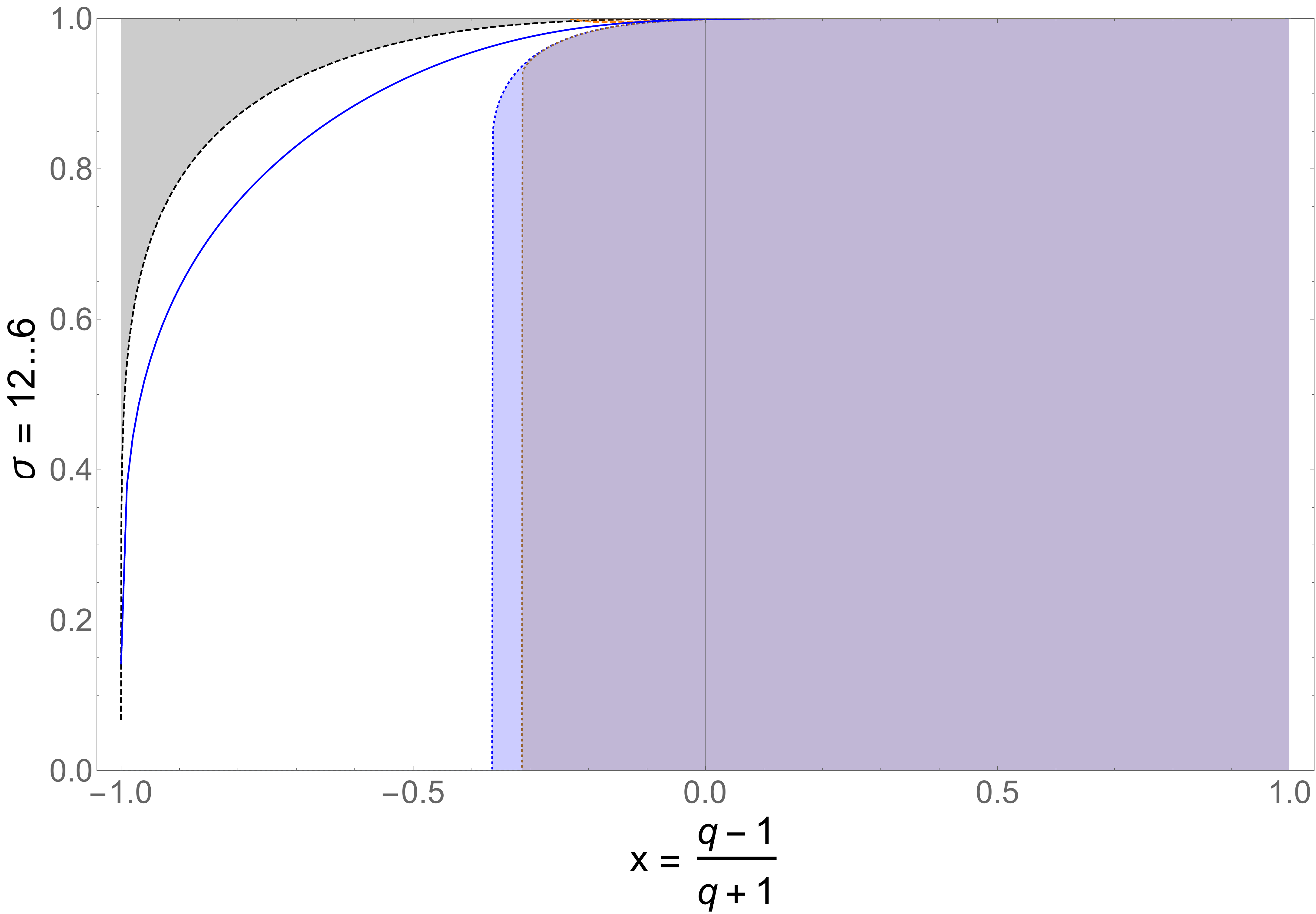}}
\subfloat[$\sigma=12\ldots8$]{\includegraphics[scale=.12]{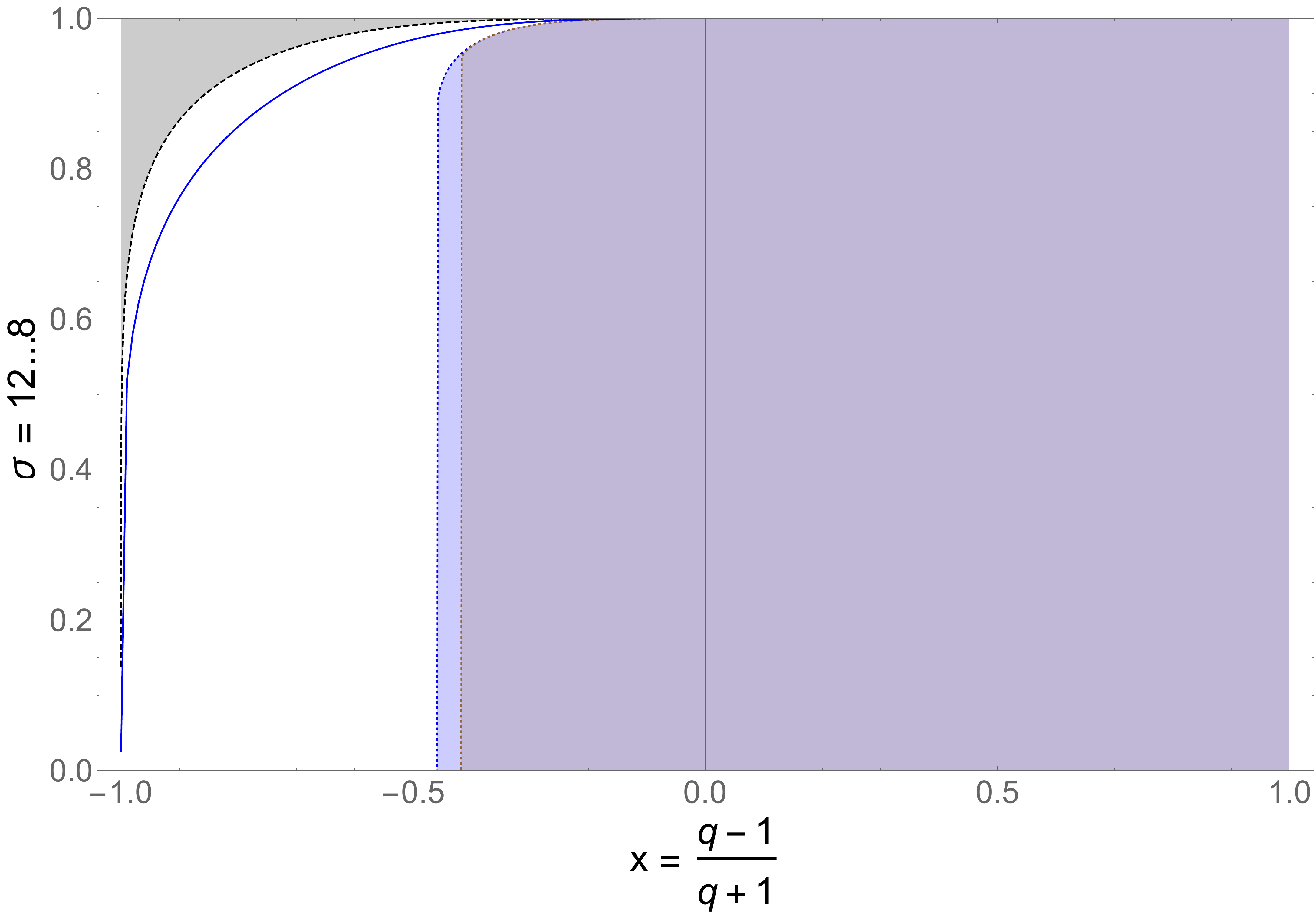}}
\subfloat[$\sigma=12\ldots10$]{\includegraphics[scale=.12]{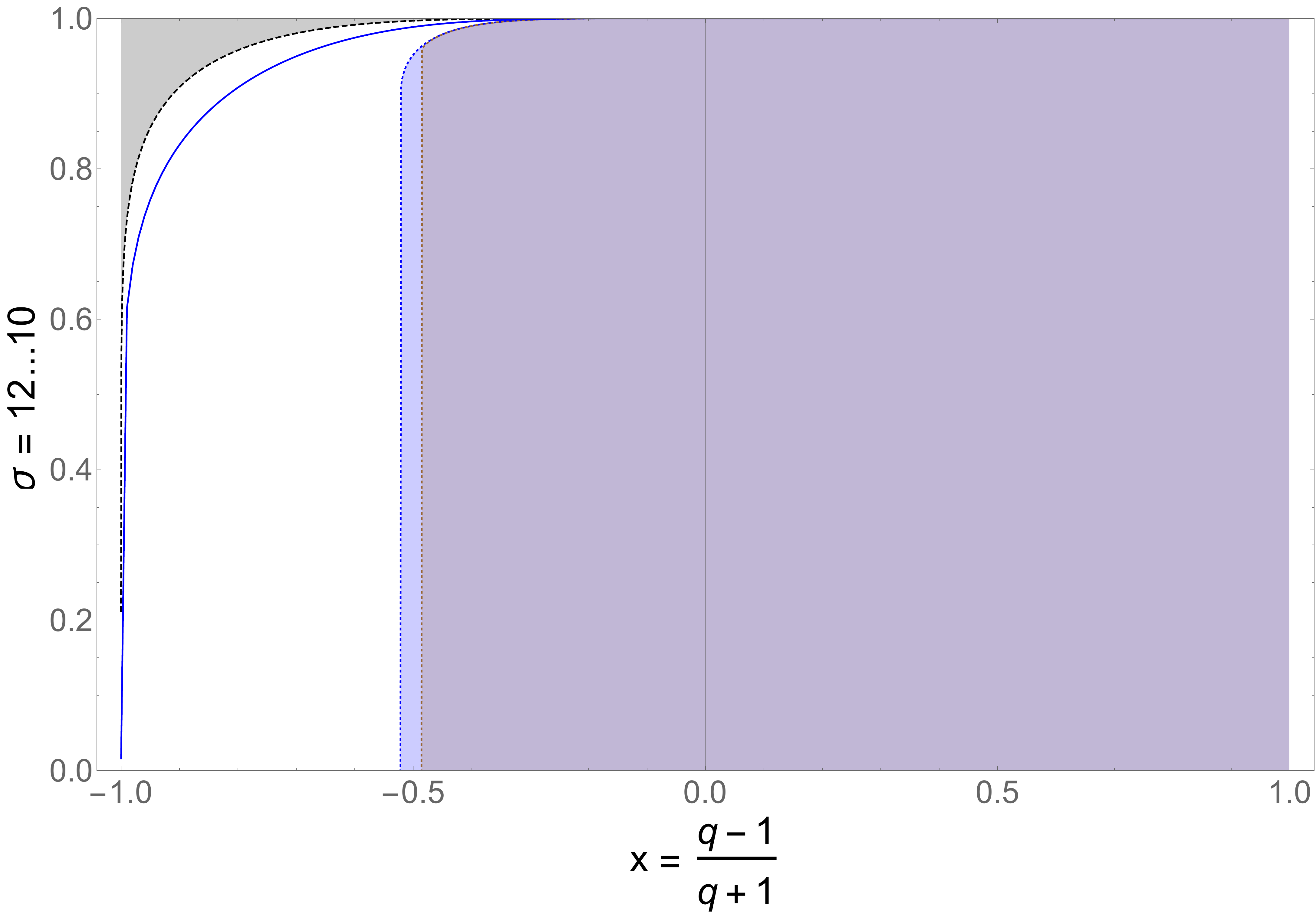}}
\caption{The solid blue curve is a plot of $\rho(12\ldots m,q)$ with $q = \frac{1+x}{1-x}$, for $m=6, 8, 10$.  The other curves are as described in the caption of Figure~\ref{fig:bounds}.}
\label{fig:12m:bounds}
\end{figure}

\section{Distribution of the number of occurrences}
\label{section:Stein}

For any permutation $\pi=\pi_1\dots\pi_n$ and $\sigma\in\S_m$, recall the definition of $x_1,\ldots,x_{n-m+1}$ as the indicator variables defined in \eqref{eq:markings} and $N_n(\sigma,q)=\sum_{j=1}^{n-m+1}x_j$ as the random variable counting the number of occurrences of $\sigma$ in $\pi$ from the Mallows($q$) distribution on $\S_n$.
When concerned only with the binary event of whether $\pi$ avoids $\sigma$, the only relevant information in $x_1,\ldots,x_{n-m+1}$ is whether the event $\{N_n(\sigma,q)=0\}$ occurs, and 
our analysis above focuses on the asymptotic behavior of $P_n(\sigma,q) = \P(N_n(\sigma, q)=0)$. 
By Theorem \ref{thm:rho}, this probability decays exponentially fast in $n$, implying that for all practical purposes $P_n(\sigma,q)$ 
is negligible for even moderately large values of $n$.
As a complement of our prior analysis, we now consider how far a permutation strays from avoiding a given pattern by studying the distribution of the number of occurrences of $\sigma$ in a random permutation from the Mallows($q$) distribution.
Other authors, e.g., \cite{CraneDeSalvo2015,NakamuraJanson2015,Nakamura2013}, have studied similar questions using different techniques and under different assumptions.

Computing the distribution of the number of occurrences, however, is complicated by the dependence among the locations at which a pattern is allowed to occur.
For a simple example, notice that the pattern $213$ can occur at most once in a permutation of length 4, meaning that the variables $x_1,x_2$ indicating occurrence starting in positions 1 and 2, respectively, are negatively correlated.  
On the other hand, the pattern $123$ can occur 0, 1, or 2 times in a permutation of length 4.

Complications due to this dependence, however, are mitigated by the weak dissociation property of the Mallows($q$) distribution, according to which $x_i$ and $x_j$ are independent as long as ${|i-j| > m-1}$.
In probabilistic terms, $(x_1,\dots,x_{n-m+1})$ is an {\em $(m-1)$-dependent} sequence, for which much is already known, including quantitative bounds on convergence rates to a central limit theorem; see~\cite{ChenShao2004}. 
The consecutive homogeneity property of the Mallows($q$) distribution implies that $x_i$ and $x_j$ have the same distribution.  
Both the weak dissociation and consecutive homogeneity properties of the Mallows($q$) distribution play a critical role, implicitly or explicitly, throughout all of our above analysis, including the probabilistic proof that the growth rate always exists (Theorem \ref{thm:rho}) and the bounds obtained in Section \ref{section:bounds}, as when deriving the improved upper bounds in~Proposition~\ref{prop:improved}.

These properties also play a role in the following theorem, which applies to the entire distribution of $N_n(\sigma,q)$ and, therefore, says something about its behavior in high probability regions that are most likely to occur. 

For any $\sigma\in\S_m$, recall the definition of the overlap set $\Ov(\sigma)$ from Definition \ref{overlap}.

\begin{theorem}\label{theorem:Stein}
Fix any $m \geq 3$, $\sigma\in\S_m$, $q>0$, $n \geq 2m-1$, and let $N_n(\sigma,q)$ denote the number of times $\sigma$ occurs in a random permutation of length $n$ from the Mallows$(q)$ distribution. 
For $1\leq s \leq m-1$, define 
\begin{equation}\label{eq:mu} \mu(\sigma, q) := \frac{q^{\inv(\sigma)}}{[m]_q!}; \qquad a_{n}(\sigma,q) := (n-m+1)\, \mu(\sigma,q); \qquad T(s,\sigma,q) := \sum_{\tau \in  {\small\Ov}_s(\sigma)} \frac{q^{\inv(\tau)}}{[2m-s]_q!};\end{equation}
\begin{equation}\label{b:def} b_{n}(\sigma,q)^2 := (n-m+1)\mu(\sigma,q)(1-\mu(\sigma,q)) + 2\sum_{s=1}^{m-1} (n-2m+1+s) \left(T(s,\sigma,q) - \mu(\sigma,q)^2\right);\end{equation}
\[ \theta_n(\sigma,q) := \frac{\mu(\sigma,q)(1-\mu(\sigma,q))^3 + (1-\mu(\sigma,q))\mu(\sigma,q)^3}{b_{n}(\sigma,q)^{3}}. \] 
Then we have 
\begin{equation}\label{bound} \sup_k \left| \P(N_n(\sigma,q) \leq k) - \P\left(Z \leq \frac{k-a_{n}(\sigma,q)}{b_{n}(\sigma,q)}\right) \right| \leq 75(10(m-1)+1)^{2} (n-m+1)\theta_n(\sigma,q),\end{equation}
where $Z$ is a Gaussian random variable with mean 0 and variance 1, that is,
\[\mathbb{P}(Z\leq z)=\int_{-\infty}^z\frac{1}{\sqrt{2\pi}}e^{-x^2/2}dx,\quad -\infty<z<\infty.\]
\end{theorem}

\begin{proof}
Fix any $m\geq 3$, $\sigma \in S_m$, $q>0$, and $n \geq 2m-1$.  We use the notation $\mu \equiv \mu(\sigma, q)$, $b_n \equiv b_n(\sigma, q)$, and $\theta_n \equiv \theta_n(\sigma, q)$. 
By linearity of expectation, it immediately follows that $\e N_n(\sigma,q)=(n-m+1)\mu$.
For $b_{n}$, we use the decomposition of the variance of $N_n(\sigma,q)=\sum_{i=1}^{n-m+1}x_i$ by
\[ \mbox{Var}(N_n(\sigma,q)) = \sum_{i=1}^{n-m+1} \mbox{Var}(x_i) + 2\sum_{i<j} \mbox{Cov}(x_i, x_j),\]
where $\mbox{Cov}(x_i,x_j)=\e(x_ix_j)-\e(x_i)\e(x_j)$ is the covariance of $x_i$ and $x_j$.
By the weak dissociation property of the Mallows($q$) distribution, $\mbox{Cov}(x_i, x_j) = 0$ whenever $|i-j|> m-1$, and so we need only consider pairs $i$ and $j$ for which $1 \leq |i-j| = s\leq m-1$.  
Using Definition~\ref{overlap}, we have 
\[ \e\, x_i\, x_{i+m-s} = \sum_{\tau \in \Ov_s(\sigma)} \frac{q^{\inv(\tau)}}{[2m-s]_q!}, \qquad 1 \leq s \leq m-1, \quad 1 \leq i \leq n- 2m +1 + s. \]
Since $x_i$ is an indicator random variable, we also have $\mbox{Var}(x_i) = \mu(1-\mu)$ for all $i=1,2,\ldots, n-m+1$ by the consecutive homogeneity property of the Mallows distribution. 
Whence,
\begin{align*}
b_n^2 = \mbox{Var}\left(\sum_{i=1}^{n-m+1} x_i\right) & = \sum_{i=1}^{n-m+1} \mbox{Var}(x_i) + 2\sum_{i<j} \mbox{Cov}(x_i, x_j). \\
   & = \sum_{i=1}^{n-m+1} \mu(1-\mu) + 2 \sum_{s=1}^{m-1} \sum_{i=1}^{n-2m+1+s} \left(\e \, x_i \, x_{i+m-s} - \mu^2\right) \\
   & = (n-m+1)\mu(1-\mu) + 2\sum_{s=1}^{m-1} (n-2m+1+s) \left(T(s,\sigma,q) - \mu^2\right). 
\end{align*}
where we used~\eqref{eq:sumexx} in the last equality.
Next, we define the random variables
\[ \xi_i := \frac{x_i - \mu}{b_n}, \qquad i=1,2,\ldots,n-m+1, \]
and
\[  W := \sum_{i=1}^{n-m+1} \xi_i. \]
We have $\e\, \xi_i = 0$ for $i=1,2,\ldots, n-m+1$, and by our choice of $b_{n}$ above we have $\mbox{Var}(W) = 1$. 
We also have, since $x_i$ is an indicator random variable, that $0 \leq \mu \leq 1$, and so 
\[ \theta_n  = \e |\xi_i |^3 = \frac{ \e |x_i - \mu|^3}{b_n^3} = \frac{ \mu(1-\mu)^3 + \mu^3(1-\mu)}{b_n^3}, \qquad \mbox{for all $i=1,2,\ldots, n-m+1$}. \]

By \cite[Theorem~2.6]{ChenShao2004}, since the collection $\{\xi_i\}_{1 \leq i \leq n}$ is $(m-1)$-dependent, it satisfies $\e\, \xi_i = 0$ for all $i=1,2,\ldots,n$, and since $\mbox{Var}(W) = 1$, we obtain the following bound for $W$: 
\begin{equation*}
\sup_x \left| \P(W \leq x) - \P\left(Z \leq x\right) \right| \leq 75(10(m-1)+1)^{2} \sum_{i=1}^{n-m+1} \e |\xi_i|^3. 
\end{equation*}
Inequality~\eqref{bound} now follows since $\theta_n = \e |\xi_i|^3$ for all $i=1,2,\ldots, n-m+1$.
\end{proof}

Theorem~\ref{theorem:Stein} is stated in terms of an inequality that is valid for all finite values of parameters with explicitly defined constants, rather than as a limit theorem. 
A distinct advantage of Stein's method is that it yields explicit preasymptotic bounds, from which quantitative statements can be derived for all finite values of the parameters.
Theorem~\ref{theorem:Stein} also yields many corollaries. 
We assume the notation from Theorem~\ref{theorem:Stein} in the rest of this section.

\begin{proposition}
Fix any $m \geq 3$, $\sigma \in S_m$, and $q>0$.  
Assume 
\[ 2 \sum_{s=1}^{m-1}T(s,\sigma,q) - (2m-1)\mu(\sigma,q)^2 +\mu(\sigma,q) > 0, \]
i.e., it is \emph{strictly positive}. 
Then $\frac{N_n(\sigma, q)-a_{n}(\sigma, q)}{b_{n}(\sigma,q)}$ converges in distribution to a normal random variable with mean 0 and variance 1 as $n$ tends to infinity, with rate $O\left(n^{-1/2}\right)$. 
\end{proposition}

\begin{proof}
For fixed $m \geq 3$, $\sigma \in S_m$, and $q>0$, we have
\[ a_n(\sigma, q) = n\, \mu(\sigma,q) + \phi_1(\sigma, q), \]
where $\phi_1(\sigma, q) := -(m-1)\mu(\sigma,q)$ is a function which does not depend on $n$. 
Similarly, we have 
\[ b_n(\sigma, q)^2 = n \, \phi_2(\sigma,q)  + \phi_3(\sigma,q), \]
where $\phi_2(\sigma,q) := 2 \sum_{s=1}^{m-1}T(s,\sigma,q) - (2m-1)\mu^2 +\mu$ and $\phi_3(\sigma,q)$ do not depend on $n$. 
Note that $b_n(\sigma, q)^2$ equals the variance of a sum of random variables, which must always nonnegative.
Thus, by assuming that $\phi_2(\sigma, q) > 0,$ we have 
\begin{equation}\label{ab} a_n(\sigma, q) = O(n) \quad\text{and}\quad b_n(\sigma, q) = O(\sqrt{n})\quad\text{as }n\to\infty, \end{equation}
and also 
\[ \theta_n(\sigma, q) = O\left(\frac{1}{b_n(\sigma,q)^3}\right)\quad\text{as }n\to\infty. \]
Thus, asymptotically as $n$ tends to infinity, the right-hand side~of Equation~\eqref{bound} is $O(n / b_n(\sigma,q)^3)$, which by~\eqref{ab} is also $O(n^{-1/2})$.
\end{proof}

\begin{corollary}
Fix any $m \geq 3$, $\sigma \in S_m$, and take $q=1$.  
Assume 
\[ 2 \sum_{s=1}^{m-1}T(s,\sigma,1) - (2m-1)\mu(\sigma,1)^2 +\mu(\sigma,1) > 0. \]
Let $p_n(\sigma, k)$ denote the number of permutations of $n$ with exactly $k$ occurrences of $\sigma$. 
For any $M \geq 0$, define $x_M := \frac{M-a_{n}(\sigma,1)}{b_{n}(\sigma,1)}$. 
Then \emph{for each} $n \geq 2m-1$ and $M \geq 0$, we have 
\[ \left| \frac{\sum_{k=0}^M p_n(\sigma,k)}{n!} - \int_{-\infty}^{x_M} \frac{e^{-y^2/2}}{\sqrt{2\pi}}\, dy\right| \leq 75(10m+1)^{2} (n-m+1)\theta_n(\sigma,1). \]
In particular, as $M$ and $n$ tend to infinity, suppose there is an $x \in \mathbb{R}$ such that also $x_M \to x$, then 
\[ \frac{\sum_{k=0}^M p_n(\sigma,k)}{n!} \to \int_{-\infty}^{x} \frac{e^{-y^2/2}}{\sqrt{2\pi}}\, dy. \]
\end{corollary}

\begin{figure}
\centering
\includegraphics[scale=0.12]{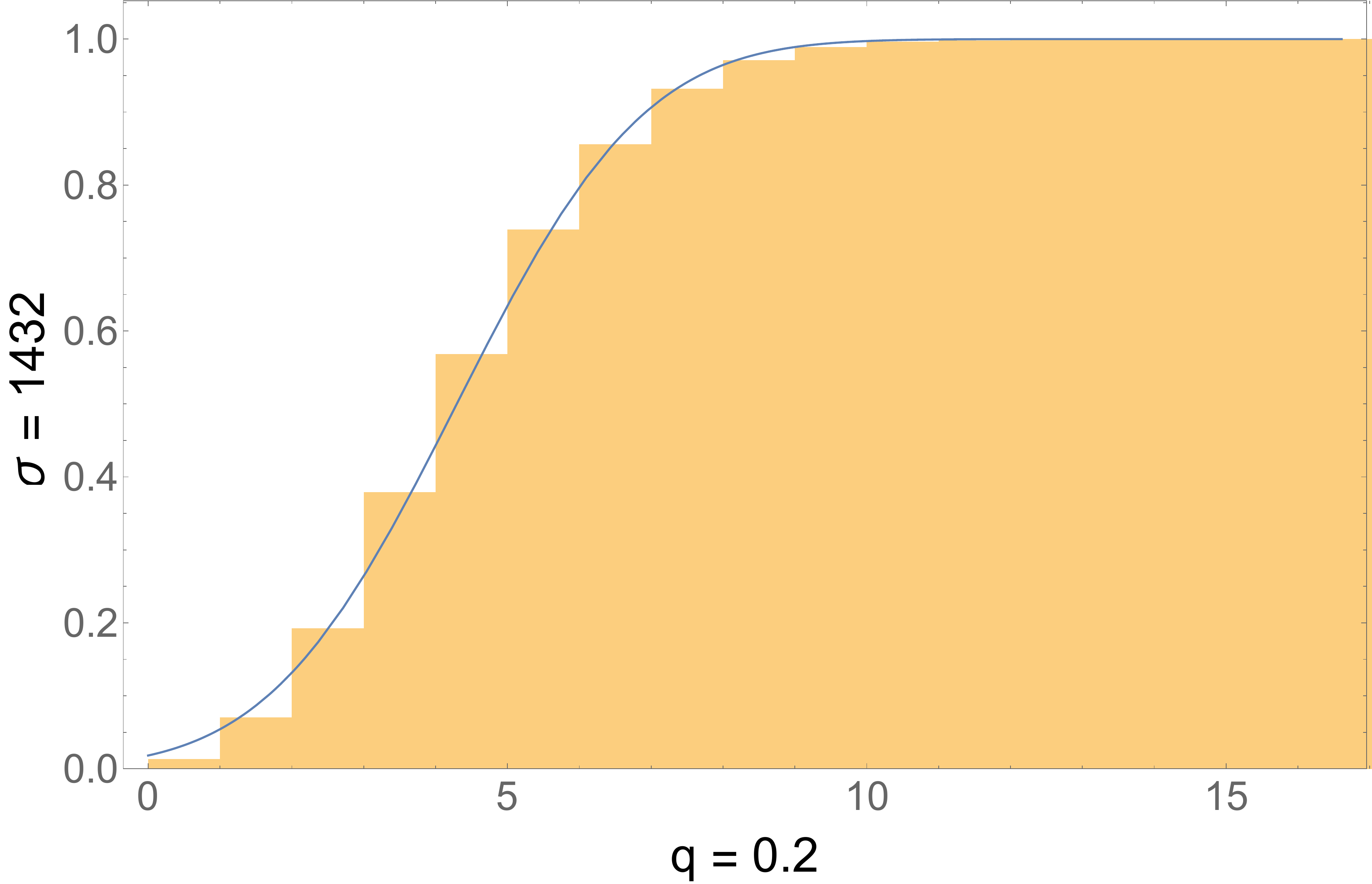} \qquad \includegraphics[scale=0.12]{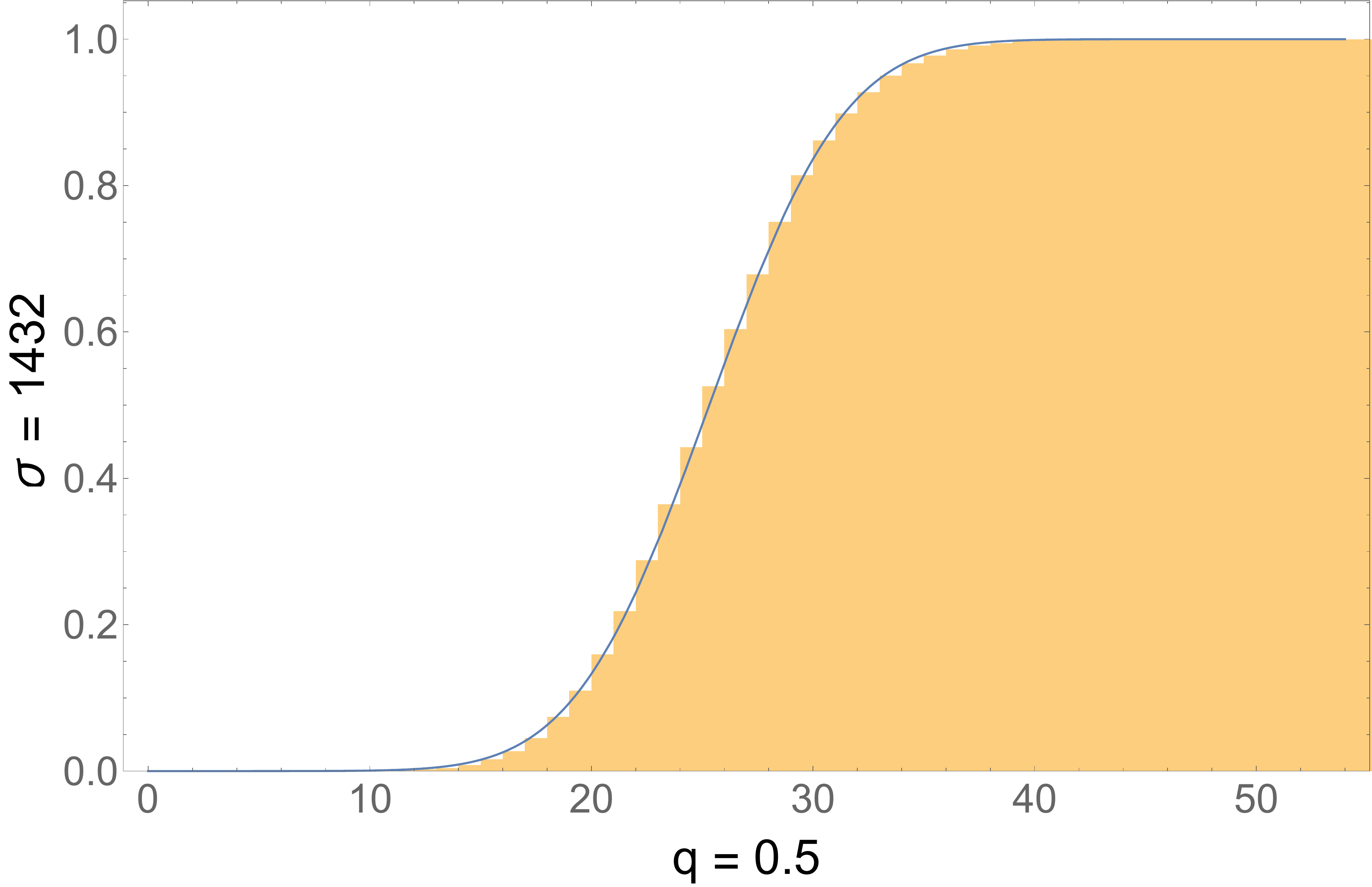} \qquad \includegraphics[scale=0.12]{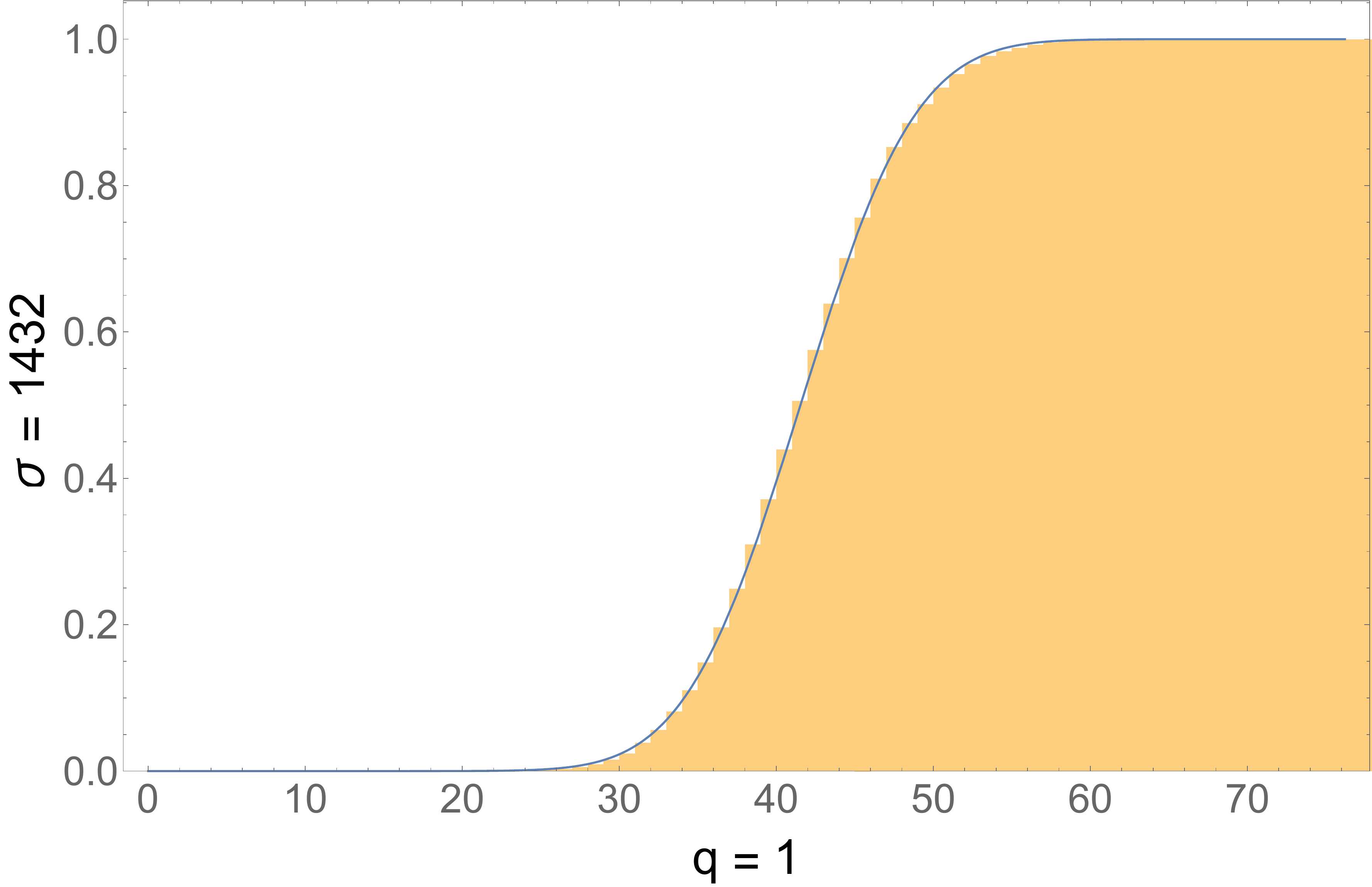} \\ \vskip .2in
 \includegraphics[scale=0.12]{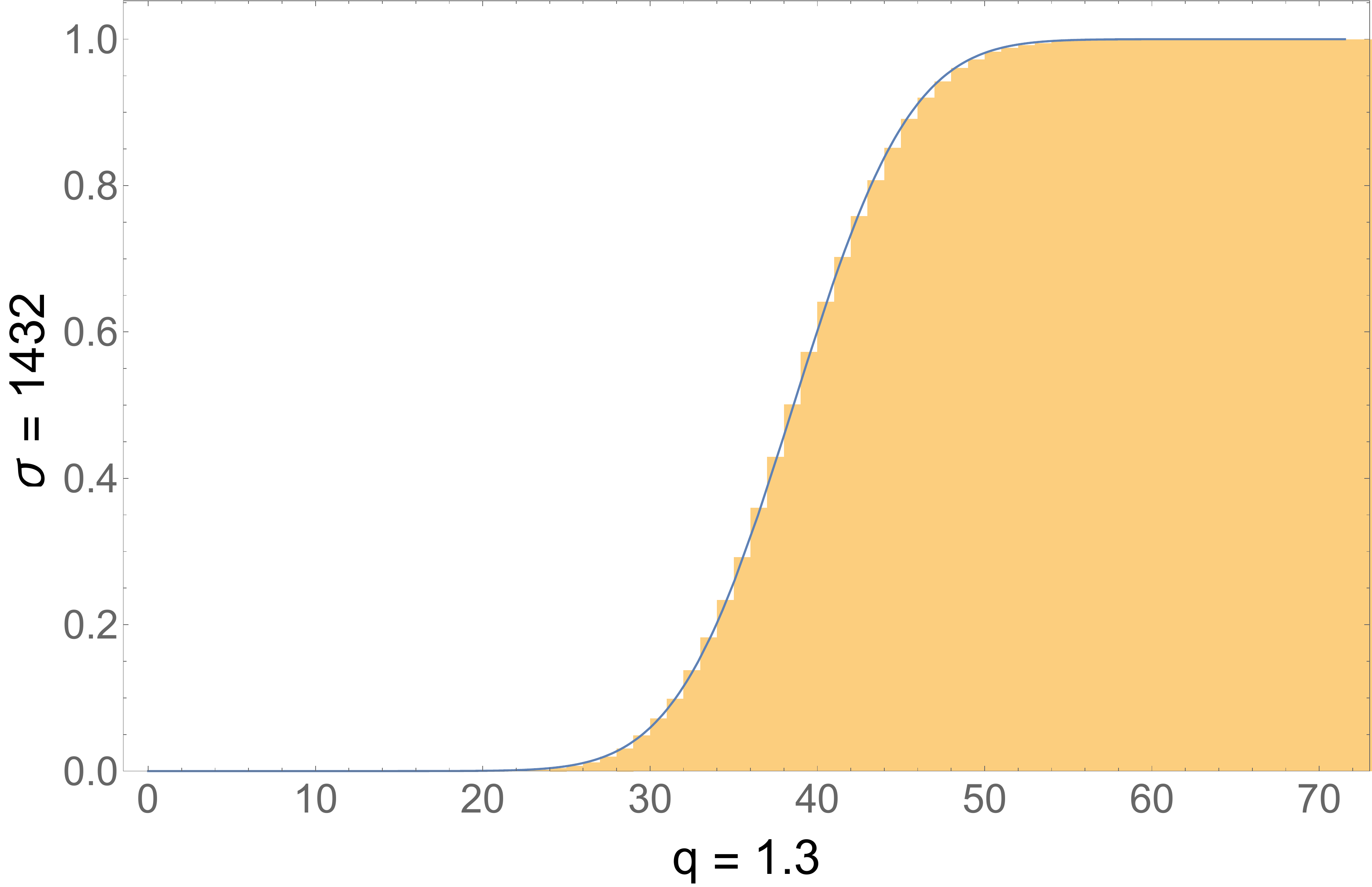} \qquad \includegraphics[scale=0.12]{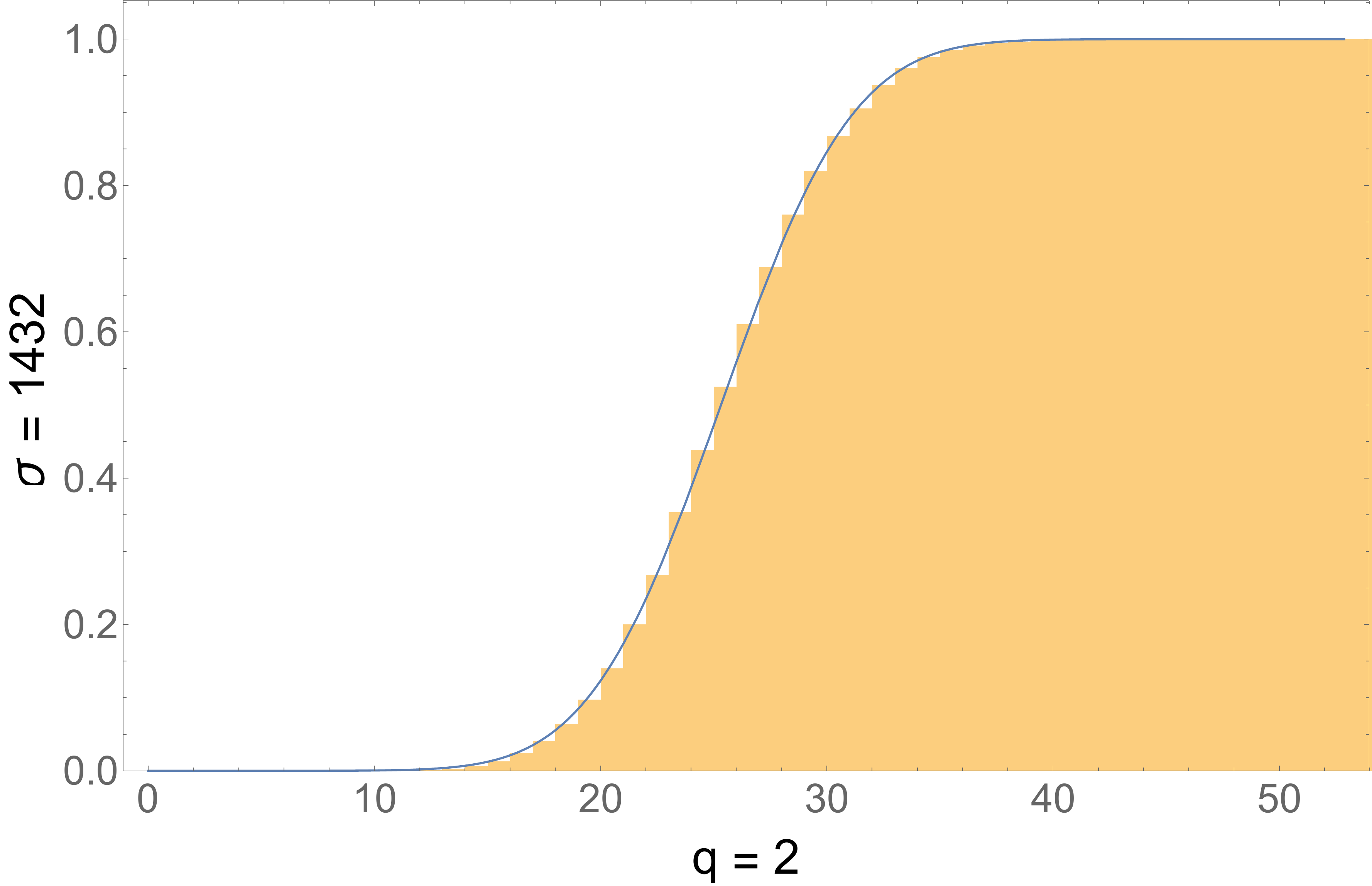} \qquad \includegraphics[scale=0.12]{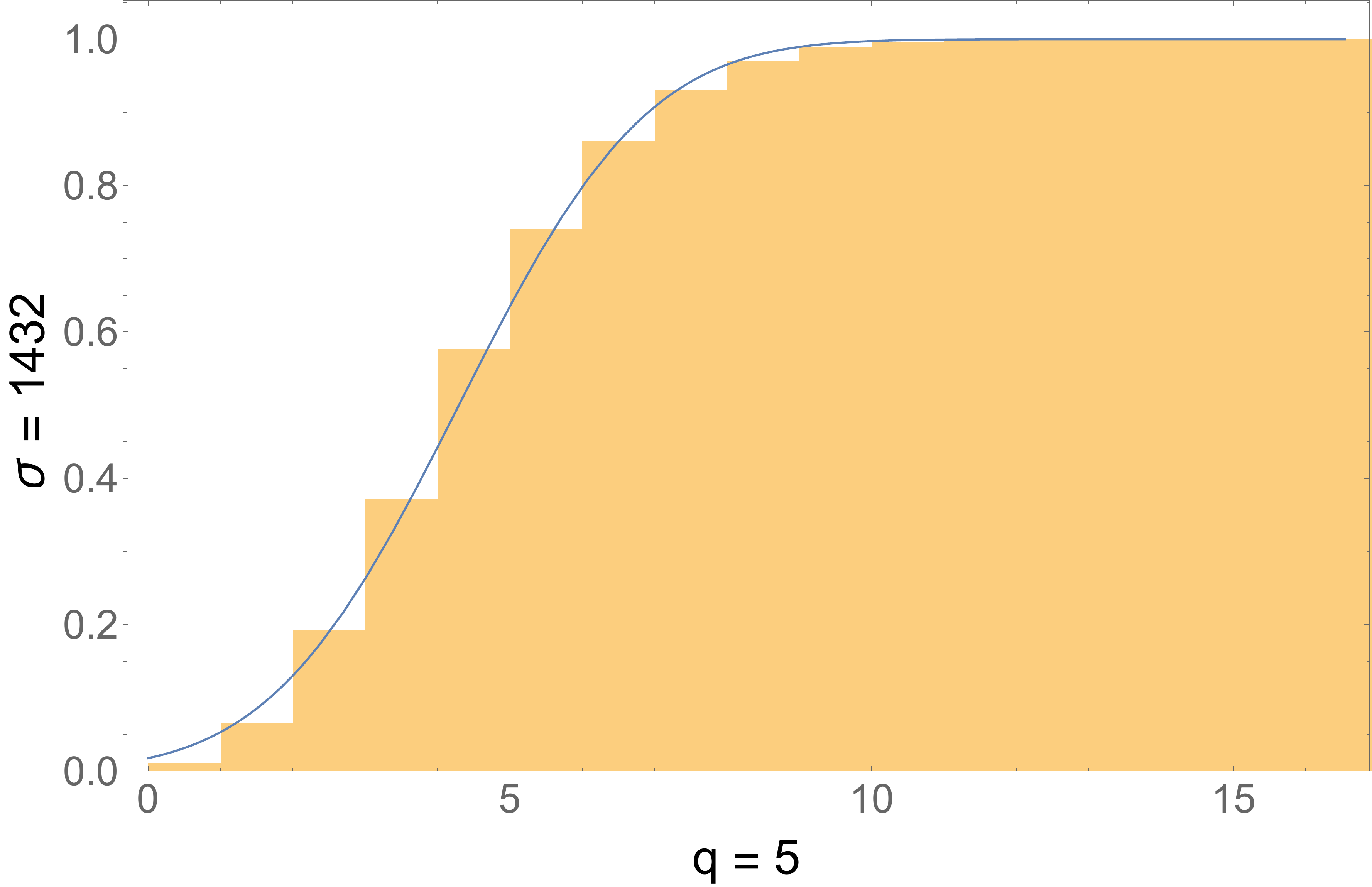}  
\caption{
Each plot contains a histogram (the shaded region) for the cumulative sum of the number of occurrences of the pattern $1432$ in a random sample of size~$10^4$ generated from the Mallows$(q)$ distribution using $n=1000$, plotted with the cumulative normal distribution function implied by Theorem~\ref{theorem:Stein}, for values of $q = 0.2, 0.5, 1, 1.3, 2, 5$, as indicated below each plot. 
}
\label{CLT:figure}
\end{figure}

\subsection{Comparison of patterns $1432$, $2341$ and $2413$}

In this section, we compare the patterns $1432$, $2341$, and $2413$, each of which has 3 inversions. 
Thus, we expect \emph{on average} the same number of these patterns in a random permutation, that is, $\e\, N_n(1432,q) = \e\, N_n(2341,q) = \e\, N_n(2413,q)$, for all $q>0$ and $n \geq 1$. 
However, the variance of the number of occurrences, denoted $\mbox{Var}(N_n(\sigma,q))$, differs for each pattern~$\sigma$. 
To see why this is the case, we note that the only difference between patterns in the formula for $b_n(\sigma,q)^2$ in Equation~\eqref{b:def} is in the terms $T(s,\sigma,q)$, $1 \leq s \leq m-1$, where $T(s,\sigma,q)$ is the probability that in a random permutation of length~$2m-s$, the pattern $\sigma$ occurs at the start and at the end (with possibly more occurrences allowed).  
We showed previously in Section~\ref{sect:1432} and in Table~\ref{Ov:1432} that for pattern $1432$ we have $\Ov_2(1432) = \Ov_3(1432) = \emptyset$, whence
\[ T(2,1432,q) = T(3,1432,q) = 0, \]
and also 
\[ T(1,1432,q) =  \frac{q^{12}+q^{11}+2 q^{10}+2 q^9+2 q^8+q^7+q^6}{[7]_q!}. \]
In fact, it is clear (see~\cite[Lemma~9]{Perarnau}) that, unless $\sigma\in\S_m$ is monotone, $\Ov_{m-1}(\sigma) = \emptyset$. Thus, we need only consider $T(s,\sigma,q)$ for $s=1,2$ in this section. 

Continuing for pattern~$2341$, we have $\Ov_2(2341) = \emptyset$, so $T(2,2341,q) = 0$;
also, by Table~\ref{Ov:2341}, which enumerates the elements of $\Ov(2341)$, we have 
\[ T(1,2341,q) = \frac{q^{15}+q^{14}+2 q^{13}+2 q^{12}+2 q^{11}+q^{10}+q^9}{[7]_q!}. \]
At this point, it is apparent that when $q=1$, the variances of $N_n(1432,q)$ and $N_n(2341,q)$ are in fact the same, since $|\Ov_s(1432)| = |\Ov_s(2341)|$ for all $1 \leq s \leq 3$. 
In addition, since the reversal of $1432$ is $2341$, by~\eqref{eq:reversal} we also have $b_n(1432,q) = b_n(2341,1/q)$ for all $q>0$, which is apparent by the symmetry about the vertical axis in Figure~\ref{stddev:compare} between the two corresponding curves. 

Finally, the pattern 2413 is distinct from the previous two patterns in the sense that $\Ov_2(2413) = \{362514, 462513\}$.  By Table~\ref{Ov:2413}, we have 
\[ T(2,2413,q) = \frac{q^{10}+q^9}{[6]_q!}, \]
\[ T(3,2413,q) = \frac{q^{11}+2 q^{10}+3 q^9+2 q^8+q^7}{[7]_q!}. \]
We have found that $b_{100}(2413,q)$ and $b_{100}(1432,q)$ cross at the point $x_{100} \approx  -0.2519754$, 
corresponding to a value of $q_{100} \approx 0.5974755$, 
which gives $b_{100}(2413,q_0) = b_{100}(1432,q_0) \approx 1.641219$ 
(see Figure~\ref{stddev:compare}).
As $n$ grows, this intersection appears to achieve a limiting value of $x_\infty \approx -0.2510049$, 
corresponding to $q_\infty \approx 0.5987148$, and taking $n$ large we have 
\[ \frac{b_n(1432,q_\infty)}{\sqrt{n}} \to 0.1667240. \] 
In addition, we observe numerically that the curve $b_{100}(2413,q)$ always stays strictly above $b_{100}(2341,q)$ for $q>0$.

\begin{figure}[htb]
\centering\includegraphics[height=6cm]{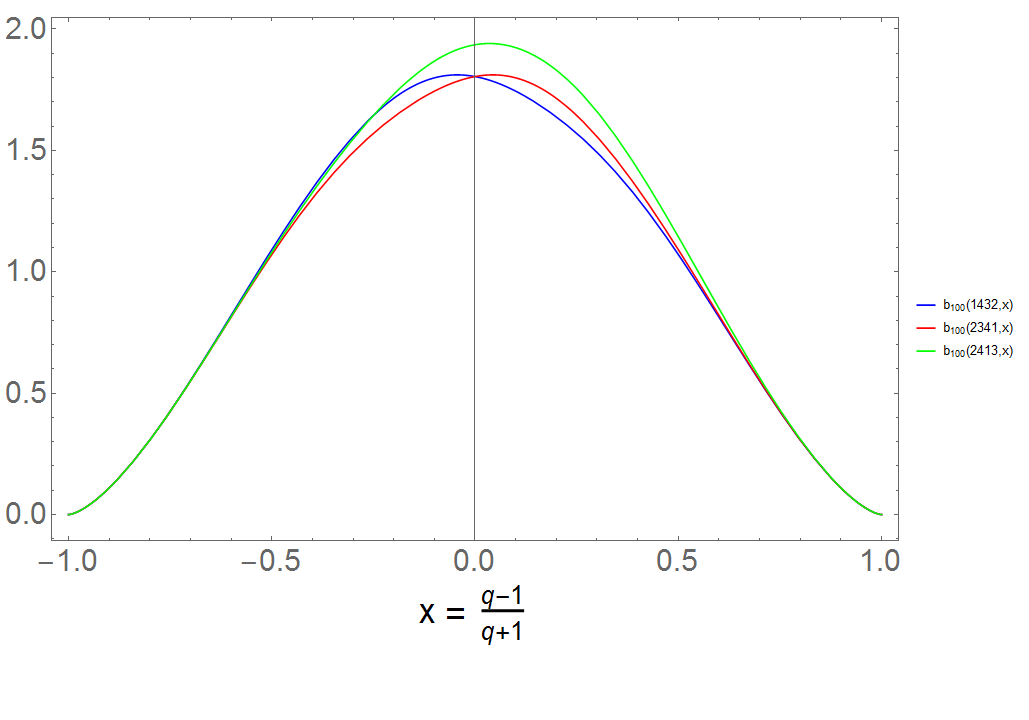}
\caption{A comparison of the standard deviations of $N_{100}(\sigma, q)$ for $\sigma = 1432$ {\color[rgb]{0,0,1} (blue)}, $2341$ {\color[rgb]{1,0,0} (red)}, $2413$ {\color[rgb]{0,1,0} (green)}.}
\label{stddev:compare}
\end{figure}


\begin{thebibliography}{}

\bibitem{Baldi} P. Baldi, Y. Rinott, and C. Stein. (1989).  A normal approximation for the number of local maxima of a random function on a graph. {\em Probability, statistics, and mathematics}, pages 59--81. Academic Press, Boston, MA

\bibitem{BasuBhatnagar2016} R. Basu and N. Bhatnagar.   (2016).  Limit Theorems for Longest Monotone Subsequences in Random Mallows Permutations. {\em Annales de l'Institut Henri Poincar\'e}, to appear, arXiv:1601.02003. 


\bibitem{Bhatnagar2014} N. Bhatnagar and R. Peled.  (2015). Lengths of {M}onotone {S}ubsequences in a {M}allows {P}ermutation.  {\em Probability Theory and Related Fields}, {\bf 151}(3):719--780.

\bibitem{cameronkilpatrick} N.T. Cameron and K. Killpatrick. (2015).  Inversion polynomials for permutations avoiding consecutive patterns.  {\it Adv. in Appl. Math.}, {\bf 67}:20--35.

\bibitem{ChenShao2004} L. H. Y. Chen and Q.-M. Shao. (2004).  Normal approximation under local dependence. {\it Ann. Probab.}, {\bf32}(3A), 1985--2028.

\bibitem{CraneESF} H. Crane. (2016).  The ubiquitous Ewens sampling formula (with comments and a rejoinder by the author). {\em Statistical Science}, {\bf31}(1):1--39.

\bibitem{CraneDeSalvo2015} H. Crane and S. DeSalvo.  (2015).  Pattern avoidance for random permutations. Preprint, \texttt{arXiv:1509.07941}.

\bibitem{DR} P. Diaconis and A. Ram.  Analysis of Systematic Scan Metropolis Algorithms Using Iwahori-Hecke Algebra Techniques.  {\it Michigan Journal of Mathematics}, {\bf48}(1):157--190.

\bibitem{SaganDokos2012} T. Dokos, T. Dwyer, B.P. Johnson, B. Sagan and K. Selsor.  (2012).  Permutation Patterns and Statistics.  {\em Discrete Mathematics}, {\bf312}:2760--2775.

\bibitem{Eliasym} S. Elizalde.  (2006).  Asymptotic enumeration of permutations avoiding generalized patterns. {\it  Adv. in Appl. Math.}  {\bf36},  138--155.

\bibitem{EliCMP} S. Elizalde.  (2013).  The most and the least avoided consecutive patterns. {\it Proc. Lond. Math. Soc.} {\bf106}, 957--979.

\bibitem{Elisurvey} S. Elizalde. (2016).   A survey of consecutive patterns in permutations.  Chapter in {\it Recent Trends in Combinatorics (IMA Volume in Mathematics and its Applications)}, Springer.

\bibitem{EliNoy} S. Elizalde and M. Noy.  (2003).  Consecutive patterns in permutations. \textit{Adv. Appl. Math.} {\bf30}, 110--125.

\bibitem{EliNoy2} S. Elizalde and M. Noy.  (2012).  Clusters, generating functions and asymptotics for consecutive patterns in permutations. {\it Adv. in Appl. Math.} {\bf49}, 351--374.

\bibitem{Ewens1972} W.J. Ewens. (1972).  The sampling theory of selectively neutral alleles. {\it Theoretical Population Biology}, {\bf3}:87--112.

\bibitem{FS} P. Flajolet and R. Sedgewick. (2009). {\it Analytic combinatorics}, Cambridge University Press, Cambridge.

\bibitem{FV} M.A. Fligner and J.S. Verducci (Eds.) (1993). {\it Probability Models and Statistical Analyses for Ranking Data} (Lecture Notes in Statistics), Springer.


\bibitem{GladkichPeled2016} A. Gladkich and R. Peled. (2016). On the Cycle Structure of Mallows Permutations. Preprint, \texttt{arXiv:1601.06991}. 

\bibitem{GnedinOlshanski2009} A. Gnedin and G. Olshanski. (2009).  A $q$-analogue of de {F}inetti's theorem. {\em Electronic Communications in Probability}, {\bf 16}, R78.

\bibitem{GnedinOlshanski2010} A. Gnedin and G. Olshanski. (2010).  $q$-exchangeability via quasi-invariance.  {\em Annals of Probability}, {\bf38}(6):2103--2135.


\bibitem{GJ79} I.P. Goulden and D.M. Jackson.  (1979). An inversion theorem for cluster decompositions of sequences with distinguished subsequences. \textit{J. London Math. Soc.} {\bf20}(2):567--576.

\bibitem{HoffmanRizzolo2016I} C. Hoffman, D. Rizzolo and E. Slivken. (2016).  Pattern-avoiding permutations and Brownian excursion, Part I: Shapes and fluctuations.  {\em Random Structures and Algorithms}, to appear, \texttt{arXiv:1506.04174}.

\bibitem{HoffmanRizzolo2016II} C. Hoffman, D. Rizzolo and E. Slivken. (2016).  Pattern-avoiding permutations and Brownian excursion, Part II: Fixed points. {\em Probability Theory and Related Fields}, to appear, \texttt{arXiv:1406.5156}.

\bibitem{SuenByJanson} S. Janson.  (1998).  New versions of {S}uen's correlation inequality.  {\em Random Structures and Algorithms}.  {\bf13}(3-4):467--483. 

\bibitem{Janson2016} S. Janson.  (2016).  Patterns in random permutations avoiding the pattern 132. {\em Combinatorics Probability and Computing}, to appear.

\bibitem{LLL} P. Erd\H{o}s and L. Lov\'asz.  (1975). Problems and results on 3-chromatic hypergraphs and some related questions. In A. Hajnal, R. Rado, and V. T. S\`os, eds. Infinite and Finite Sets (to Paul Erd\H{o}s on his 60th birthday). North-Holland. pp. 609--627.

\bibitem{NakamuraJanson2015} S. Janson, B. Nakamura, and D. Zeilberger (2015). On the Asymptotic Statistics of the Number of Occurrences of Multiple Permutation Patterns. 
\textit{J. Comb.} {\bf6}: 117-143.

\bibitem{Mallows1957} C. Mallows. (1957).  Non-null ranking models. {\em Biometrika\/} {\bf44}:114--130.

\bibitem{MinerPak} S. Miner, I. Pak.  (2014). The shape of random pattern-avoiding permutations. {\it Adv. in Appl. Math.} {\bf55}:86--130. 

\bibitem{Nakamura2013} B.~Nakamura. (2013) Approaches for enumerating permutations with a prescribed number of occurrences of patterns. 
{\it Pure Math. Appl. (PU.M.A.)} {\bf24}(2): 179-194.

\bibitem{PerarnauArXiv} G. Perarnau. (2012) A probabilistic approach to consecutive pattern avoiding in permutations, \texttt{arXiv:1208.5366}.

\bibitem{Perarnau} G. Perarnau. (2013) A probabilistic approach to consecutive pattern avoiding in permutations, {\it J. Combin. Theory Ser. A} {\bf120}:998--1011. 

\bibitem{PeresSchlag} Y. Peres and W. Schlag.  (2010).  Two {E}rd{\H o}s problems on lacunary sequences: chromatic number and {D}iophantine approximation, {\it Bull. Lond. Math. Soc.} {\bf42}(2):295--300

\bibitem{Raw} D. Rawlings. (2007). The $q$-exponential generating function for permutations by consecutive patterns and inversions. {\it J. Combin. Theory Ser. A} {\bf 114}:184--193.

\bibitem{SS} S. Starr.  (2009).  Thermodynamic Limit for the Mallows Model on $S_n$. {\em J. Math. Phys.}, {\bf50}:095208 (15 pages). 

\bibitem{Suen} S. Suen.  (1990). A correlation inequality and a Poisson limit theorem for nonoverlapping balanced subgraphs of a random graph.  {\it Random Structures and Algorithms}, {\bf1}(2):231--242.

\bibitem{vLW} J.H. van Lint and R.M. Wilson.  (2001). {\it A course in combinatorics}, second edition. Cambridge University Press, Cambridge.

\bibitem{Willis} Nicholas J. Willis, Annie K. Didier, Kevin M. Sonnanburg.  (2008).  How to Compute a Puiseux Expansion.  Preprint, \texttt{arXiv:0807.4674}.


\end{thebibliography}
\end{document}